\documentclass[dvips,11pt,letter ]{amsart}  
\usepackage[left=3.1cm,right=3.1cm,top=3.1cm,bottom=3.1cm]{geometry}
\usepackage{graphics}
\usepackage{graphicx}
\usepackage{multicol}
\usepackage{hyperref}
\usepackage{amsmath,amssymb,amsfonts, amsthm,epsfig}
\usepackage{verbatim}
\usepackage{enumerate}
\usepackage{float}
\restylefloat{figure}

\numberwithin{equation}{section}

\newtheorem{thma}{Theorem}[section]
\newtheorem{lemma}[thma]{Lemma}
\newtheorem{coro}{Corollary}

\newtheorem{prop}[thma]{Propostion}

\renewcommand{\t}{\theta}

\newcommand{\g}{{\gamma}}
\renewcommand{\a}{{\alpha}}

\newcommand{\df}[1]{\displaystyle{#1}}
\newcommand{\E}{{\mathbb{E}}}
\renewcommand{\P}{{\mathbb{P}}}
\newcommand{\R}{{\mathbb{R}}}

\newcommand{\Z}{{\mathbb{Z}}}

\newcommand{\der}[1]{\frac{\partial}{\partial #1}}

\newcommand{\half}{{\frac{1}{ 2}}}
\newcommand{\e}{{\epsilon}}

\newcommand{\st}{{\mbox{ such that }}}
\renewcommand{\half}{{\frac{1}{2}}}
\newcommand{\critical}{{\frac{1-\sqrt{1-x^2}}{x^2}}}
\newcommand{\criticall}{{\frac{-1+\sqrt{1+x^2}}{x^2}}}
\newcommand{\Pf}{{\rm{Pf}}}
\renewcommand{\(}{\begin{equation}}
\renewcommand{\)}{\end{equation}}

\begin{document}

\title{Particle Systems arising from an anti-ferromagnetic Ising model}
\author{Sunil Chhita}
\address{KTH, Stockholm, Sweden}
\email{chhita@math.kth.se}
\begin{abstract}
We study a low temperature anisotropic anti-ferromagnetic 2D Ising model through the guise of a certain dimer model.  This model has a bijection with a one-dimensional particle system equipped with creation and annihilation.  In the thermodynamic limit, we determine the explicit phase diagrams as functions of temperature and anisotropy.  Two values of the anisotropy are of particular interest - the `critical' value and the `independent' value.  At independence, the particle system has the same distribution as the two colored noisy voter model.  Its limiting measure under a natural scaling window is the Continuum Noisy Voter Model.  At criticality, the distribution of particles on a given horizontal line, is a Pfaffian point process whose kernel in the scaling window can be written explicitly in terms of Bessel functions.  
\end{abstract}

\keywords{Dimer Model, Antiferromagnetic Ising Model, Noisy Voter Model, Continuum Noisy Voter Model, Scaling Windows, Pfaffian Point process}

\maketitle

\section{Introduction}

\subsection{History}

The dimer model is a two-dimensional exactly solvable model, which can be described in the following manner.   A dimer configuration of a planar graph is a subset of edges in which each vertex is covered exactly once by an edge.   The dimer model is a probability measure on the set of all dimer configurations of the underlying planar graph.  Weights can be applied to the edges so that specific edges have higher or lower probabilities of being observed.

The dimer model was initially introduced in \cite{Kas:61}.  \cite{Fis:66} noticed the correspondence between the dimer model and the two-dimensional Ising Model.  Since its introduction, the dimer model has been studied extensively and a vast variety of techniques have been developed using many areas of mathematics.  Much of this work was completed by R. Kenyon and A. Okounkov during the late 90's and 00's (for an excellent survey of their body of results, see \cite{Ken:09}).  The dimer model is a rather unique statistical mechanical model - not only can the partition function be easily computed but by \cite{Ken:97} also local statistics can be computed.  

Dimer models on non-bipartite lattices are fundamentally different from dimer models on bipartite lattices.  For instance, we can define a height function if the underlying lattice is bipartite (see \cite{Ken:01}) which leads to connections with random surfaces (\cite{she:05}).  Recently, there has been renewed interest in the study of dimer coverings on non-bipartite lattices.  In particular, \cite{bou:08,Li:10,Li1:10} have shown that some but not all of the machinery for the dimer model on bipartite lattices transfers to dimer models on certain non-bipartite lattices. 

In this paper, we study a one-dimensional particle system which comes from a dimer model on a certain non-bipartite dimer model.  Studying one-dimensional particle systems using dimer models is not an uncommon idea. Indeed,  \cite{BOU:07} studied non-intersecting lattice paths from a dimer model on the {\it honeycomb lattice} while \cite{Mat:03} studied systems of annihilating random walks using the dimer model.

\subsection{The Dimer model on the Fisher lattice}

The so-called {\it Fisher lattice} is the graph obtained from a honeycomb lattice with a triangle decoration assigned to each vertex as depicted on the left-handside of Figure \ref{dimer:lattice}.   This type of graph and more general graphs  have been used to study the Ising model using dimer model techniques through the so-called {\it Fisher correspondence} which was first introduced in \cite{Fis:66}.

\begin{figure}[h!]
\begin{align*}
\begin{array}{cc}
{\includegraphics[height=8cm,width=8cm]{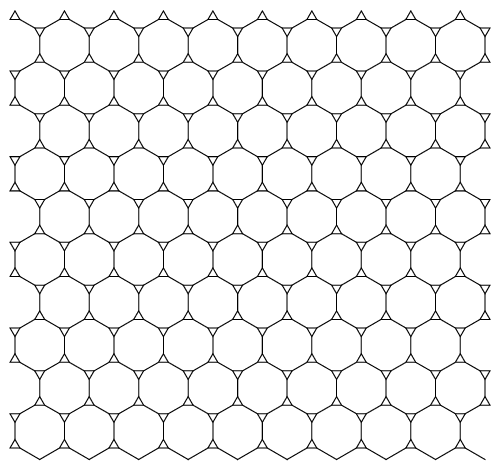}}
&
{\includegraphics[height=8cm,width=5cm]{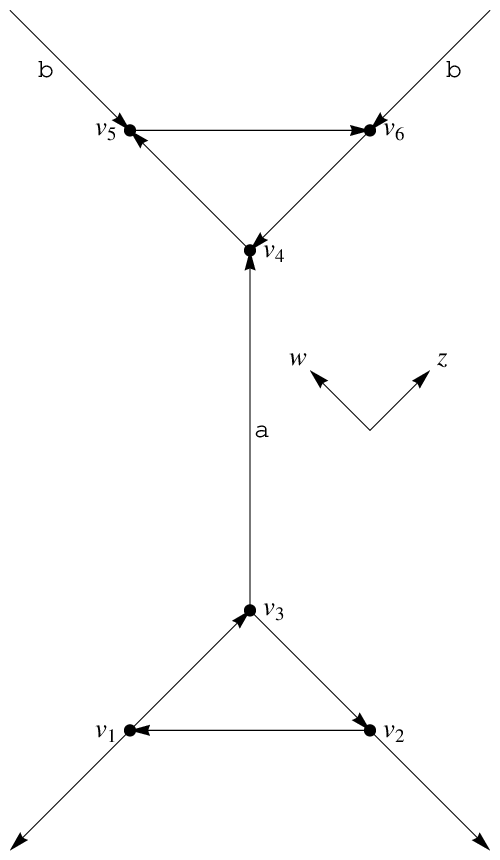}}
\end{array}
\end{align*}
\caption{The left-handside shows the Fisher lattice.  The right-handside shows the fundamental domain of the Fisher lattice studied in this paper, with its edge-weights, Kasteleyn orientation and vertex labels.  The edge from $v_1$ to $v_6$ crosses the curve $\g_1$ and the edge from $v_2$ to $v_5$ crosses the  curve $\g_2$ where $\g_1$ and $\g_2$ are curves on the dual of $G_1$ with $\g_1$ parallel to the vector $(-1,1)$ and $\g_2$ parallel to the vector $(1,1)$. }
\label{dimer:lattice}
\end{figure}

All vertical edges are referred to as  {\bf a} edges.  All edges outside the triangle decoration which are not vertical are referred to as {\bf b} edges.  We give weight $a$ to the {\bf a} edges, weight $b$ to the {\bf b} edges and weight 1 to the remaining edges.

Let $G$ denote the Fisher lattice on the plane.  Define the toroidal graph $G_n$ to be the quotient $G \slash n \mathbb{Z}^2$ where the action $n \mathbb{Z}^2$ are translations by $(n,0)$ and $(0,n)$ (horizontal and vertical translations).  Let $\mathcal{M}(G_n)$ denote the set of all dimer coverings of $G_n$.  We can define a probability measure $\mu_n$ for $M \in \mathcal{M}(G_n)$ so that the probability of $M$ is proportional to the product of all the edge weights; that is,
\(
	\mu_n(M)=\frac{a^{N_a^n} b^{N_b^n}}{Z(G_n)} \nonumber
\)
where $N_a^n=N_a^n(M)$ and $N_b^n=N_b^n(M)$ are the number of dimers covering {\bf a} and  {\bf b} edges for the dimer covering  $M$ and,  $Z(G_n)$ is the partition function: $Z(G_n)=\sum_{M \in \mathcal{M}(G_n)} a^{N_a} b^{N_b}$.  

\subsection{Anti-ferromagnetic Ising model and Setup}

Let $\Lambda_n$ represent the faces of the dodecagons in $G_n$.  Let $\Omega:= \{-1,+1 \}^{\Lambda_n}$ and let $\sigma_x \in \{-1,1 \}$ denote the spin at $x \in \Lambda_n$.   For $\sigma \in \Omega$, define the Hamiltonian by
\(  \label{intro:isinghamiltonian}
	H(\sigma)= -\sum_{v \in \Lambda_n}\sum_{v \sim w}J_{v,w}  \frac{ \sigma_v\sigma_w +1}{2}  
\)
where $v \sim w$ means that $v$ and $w$ are nearest neighbors and $J_{v,w}$ is the {\it interaction} term between the faces $v$ and $w$. For $\sigma \in \Omega$, define the probability measure,
\(\nonumber
		\nu_n (\sigma)= \frac{e^{-H(\sigma)}}{Z(\Lambda_n)}	
\)
where $Z(\Lambda_n)$ is the partition function: $Z(\Lambda_n)= \sum_{\sigma \in \Omega} e^{-H(\sigma)}$.  The above probability measure is called the \emph{Ising model}. 

The correspondence between the Ising model and dimer model can be seen in Figure \ref{intro:isingcorresp} and is as follows.  Faces sharing an edge have the same spin if there is a dimer covering the shared edge.  Otherwise, the faces have opposite spins.  
\(\nonumber
	J_{v,w}= \left\{ \begin{array} {l l}
			\log a & w \sim v \mbox{ and }wv  \mbox{  is an {\bf a} edge}\\
			\log b & w \sim v  \mbox{ and }wv \mbox{ is   a {\bf b} edge},\end{array} \right.
\)
where $wv$ is the shared edge (on the Fisher lattice) between the faces $w$ and $v$.  As there are two possibilities for the spin configurations for each dimer covering, we have
\begin{lemma}
\(\nonumber
Z(G_n)=2Z(\Lambda_n).
\)
\end{lemma}
\begin{figure}
\begin{center}
\epsfsize220pt
$$\scalebox{1}{\epsfbox{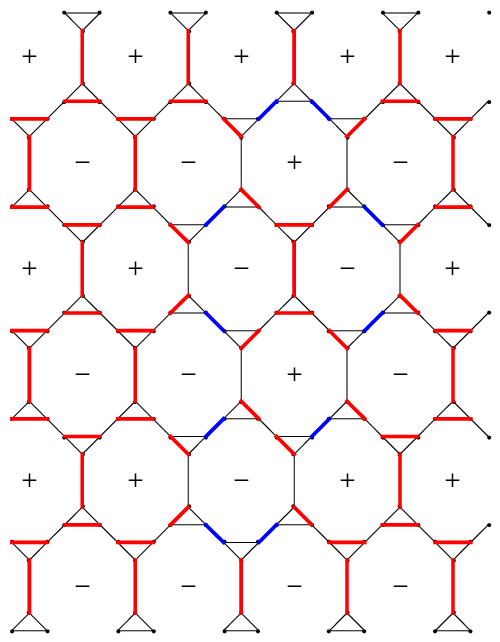}}$$
\caption{The correspondence between the Ising model and dimer coverings on the Fisher lattice. The blue edges represent dimers covering {\bf b} edges. The red edges are dimers covering the decoration and {\bf a} edges.  The `+' and `-' are spins in the corresponding Ising model.}
\label{intro:isingcorresp}
\end{center}
\end{figure}

We choose $a=x$ and $b=u x$ where $x<1$ and $u \in (0,1)$.  As we have $u<1/x$, then $J_{v,w}<0$ for $v \sim w$. This means that  adjacent sites have a high probability of having opposite spins.  This is called an {\it anti-ferromagnetic} interaction.  Furthermore, $J_{v,w}$ depends on the direction of the edge between the sites $v$ and $w$.  This is called an {\it anisotropic} interaction.   The {\it temperature} of the Ising model is equal to $1/\log(1/x)$ and the {\it anisotropy} is equal to $\log(1/u)/\log(1/x)$.  For fixed values of $x$, the parameter $u$ controls the anisotropy. 

Finally, we describe the concept of {\it thermodynamic limit}.  As previously mentioned, dimer coverings on the Fisher lattice with weights $a=x$ and $b=u x$ embedded on the torus $G_n$ give a probability measure $\mu_n (u,x)$.  The {\it thermodynamic limit} is the measure, $\mu^{u,x}=\lim_{n \rightarrow \infty} \mu_n(u,x)$, i.e. the limiting measure obtained when the system size is sent to infinity.  This measure is a {\it Gibbs} measure.  This limiting measure was shown to exist for a certain class of non-bipartite graphs in \cite{bou:08}.  Recently, \cite{Li1:10,Li:10} showed that this measure exists for the dimer model on the Fisher lattice with any choice of weights which is the case required for this paper.

\subsection{The particle model}

By considering the horizontal axis as space and the vertical axis as time, the dimer model, $\mu^{u,x}$, is in bijection the space-time diagrams of a one-dimensional particle system called the {\it particle model}.    This  is a probability measure on the set of all possible loops and rays with the same starting or finishing points on the diagonal grid.    

 An example of the correspondence can be seen in Figure \ref{intro:pathcorr}.  Each realization of the particle model is obtained from a dimer covering of the Fisher lattice by collapsing the {\bf a} edges and the triangle decoration.  This leaves a diagonal grid, where each dimer covering a {\bf b} edge in the dimer model can be thought of as a trajectory of a particle in the particle model.  By the structure of the underlying dimer model, particles can be created or annihilated in pairs.   The above construction shows that  the particle model is in bijection with the dimer model of the Fisher lattice. Figure \ref{intro:pathcorr} gives a pedagogical example of a dimer covering of the Fisher lattice along with the  corresponding particle trajectories.

The particle model has the property that particles can be created and annihilated in pairs.  We use the terminology creations and  annihilations to denote their corresponding dimer configurations which are defined as follows: a \emph{creation} is a local configuration of dimers covering two adjacent {\bf b} edges which are incident to the same inverted triangle decoration (i.e. a decoration of the form $\bigtriangledown$).  Conversely, if two adjacent {\bf b} edges are incident to the same oppositely orientated triangle decoration (i.e. a decoration of the form $\bigtriangleup$), then we call that local configuration of dimers an \emph{annihilation}.  Figure \ref{intro:pathcorr} includes an example of a creation and an annihilation.

\begin{figure}
\begin{align*}
\epsfsize220pt
\begin{array}{cc}
\scalebox{0.8}{\includegraphics{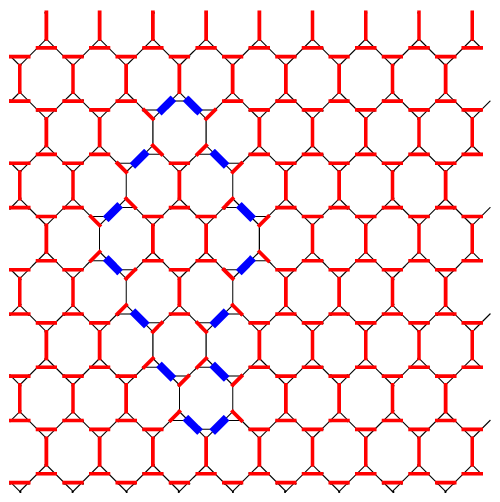}} & \scalebox{0.79}{\includegraphics{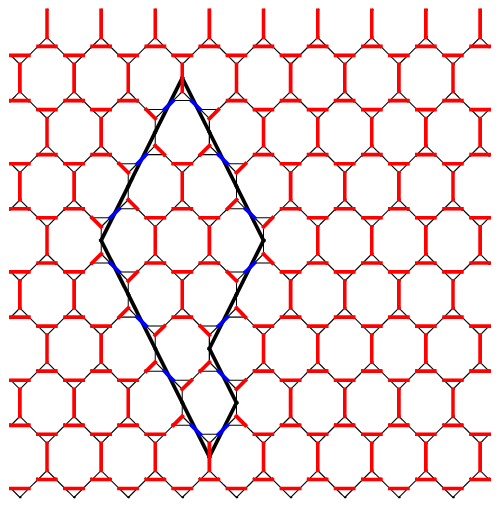}}
\end{array}
\end{align*}
\caption{A dimer configuration with its corresponding particle model.  The bottom two blue edges represent a creation while the top two blue edges represent an annihilation.  \label{intro:pathcorr}}
\end{figure}

\subsection{Scaling Limits and Scaling Windows}

The scaling limit of grid-based models is a subtle concept and we make no attempt to define it rigorously.  Informally, we can think of the {\it scaling limit}  as the limiting measure of some measure when sending the lattice spacing to zero.    For many statistical models, scaling limits are only non-trivial provided that the model is at {\it criticality}. In this context, criticality refers to the correlation length, defined in (\ref{intro:def:corlength}), being infinite. See \cite{fro:92} for discussions regarding criticality and \cite{sch:00} for the underlying philosophy of scaling limits of statistical mechanical models.

The concept of the scaling window is as subtle as the concept of the scaling limit.  In this case, we do provide the probability space for the particle model under the scaling window (see Section \ref{sec:no2}). Roughly speaking, a {\it scaling window} is the limiting measure of some measure when sending the lattice spacing to zero while simultaneously sending a certain parameter to zero.   Contrary to the scaling limit, the scaling window preserves microscopic behavior of the discrete model which means that the scaling window is non-trivial regardless of whether the model is critical.

In this paper, all scaling windows are taken with respect to the parameter $x$.  The lattice is re-scaled \emph{after} the thermodynamic limit has been taken.  We will only consider scalings which give a finite density of particles or creations in the scaling window.  Finally,  we use the following terminology:  for  $\alpha, \beta >0$, the scaling window $(x^\alpha,x^\beta)$ of a measure is the limiting measure obtained when sending a box of lattice points of size $(1/x^\alpha, 1/x^\beta)$ to a box in the continuum of size $(1,1)$, when $x\to 0$.

\subsection{Heuristic Description and Results} \label{sec:intro:state}
 
 \begin{figure}
\begin{center}
\epsfsize220pt
$$\scalebox{1}{\epsfbox{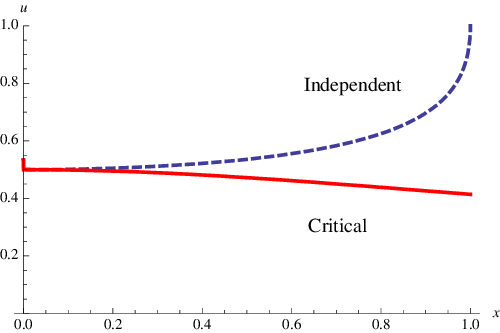}}$$
\caption{The Phase Diagram in terms of $u$ and $x$. The solid line represents criticality, where the partition function is non-analytic.  The dashed line represents independence.}
\label{intro:phase}
\end{center}
\end{figure}
 
Recall that  $\mu^{u,x}$ is the thermodynamic limit measure.   In this paper, we consider $u \in (0,1)$. Define $u_c:=(-1+\sqrt{1+x^2})/x^2$, $u_i:=(1-\sqrt{1-x^2})/x^2$, $\mu^{u,0}=\lim_{x\to 0 }\mu^{u,x}$. For $x \in (0,1)$, we have $0<u_c<1/2<u_i<1$.


Before giving the full statement of results, we first give a heuristic description in terms of the particle model.  Provided that $x>0$ and for any value of $u$, there is a small probability (dependent on $x$) of a creation of two particles at each lattice site.  Each particle performs a random walk which can be dependent on the locations of all other alive particles at that particular time-point.  If two particles meet, they annihilate. The dynamics of each particle depend intrinsically on the chosen values of $u$ and $x$. Under $\mu^{u,x}$, we can think of the value of $u$ controlling `the forces' between the particles: for $u<u_i$, there is an `attraction' between particles, for $u>u_i$, there is a `repulsion' between particles and $u=u_i$, there is no interaction between the particles.   Under $\mu^{u,0}$, there are no creations or annihilations of particles.  It turns out that there are no particles if $u \leq 1/2$ and there are particles performing non-intersecting walks if $u>1/2$ with a density of 
particles dependent on the value of $u$.  We now give a full description of the results of this paper.

 Under $\mu^{u,x}$, Theorem \ref{expect:thm:expectb}, Corollary \ref{expect:prop:Na} and Corollary \ref{expect:coro:NX} give explicit formulas for the expected number of {\bf b}, particles and creations per fundamental domain respectively  for $x \in (0,1)$ and $u \in (0,1)$.  
 By considering a low temperature expansion of the above results, we find that

\begin{thma}\label{expect:thm:behave}
Under $\mu^{u,0}$, we have 
\begin{align}\nonumber
	\mu^{u,0} (\mbox{particle})=\left\{ \begin{array}{ll} 0 & \mbox{if }u \leq \half \\
							2-\frac{2}{\pi} \arccos\left(-\frac{1}{2 u} \right) &									\mbox{if }u >\half \end{array} \right.
\end{align}
and 
\begin{align}\nonumber
	\mu^{u,0}(\mbox{creation})=0.
\end{align}
\end{thma}
In this limit, there is a phase transition at $u=1/2$.  As there are no creations (or annihilations), the particles form non-intersecting lattice paths for $u>1/2$.  

We can now consider the phase behavior for $u<1/2$ chosen independently of $x$, i.e. $u$ is not a function of $x$.  This leads to choosing a natural re-scaling for the scaling window of the particle model:  
\begin{thma} \label{uless:thm:PPP}
For $u<1/2$ chosen independently of $x$, the paths of the particle model  form loops which contract to points in the scaling window $(x,x)$.  In the scaling window $(x,x)$, these points are given by a  Poisson Point process with intensity $(1-2u^2-\sqrt{1-4u^2})/2$.
\end{thma}

Under $\mu^{u_i,x}$ with $x>0$, we find that the locations of particles in the particle model along a horizontal line are distributed according to an i.i.d. Bernoulli distribution with probability $x/(1+x)$.  Each trajectory of a particle is independent of the trajectories of the other particles and along each horizontal line, creations are i.i.d. Bernoulli with probability $c x^2$ (where $c$ is some constant).   This leads to the noisy voter model interpretation of the particle model which is described below. 

The (non-noisy) voter model on $\Z$ with two colors is a time dependent probability measure with state space $\{R,B\}^{\mathbb{Z}}$.  The update of each site is given by randomly choosing the color of a neighboring site (see \cite{Lig:85}).  The noisy voter model of two colors with noise $p$ was introduced in \cite{gra95}. Each site  chooses at random the color of a neighboring site with probability $1-p$ or flips its color with probability $p$.      \cite{Fon:06} constructed the scaling window $(x,x^2)$ of the noisy voter model with noise $c x^2$ ($c$ is a constant) and called it the  Continuum Noisy Voter Model (CNVM).  This construction relies on using the Brownian Web (see \cite{Fon:04}) and  the dual Brownian Web. The correspondence for the particle model at  the special case when $u=u_i$ and the noisy voter model is as follows:  the paths in the particle model are shown, in Section \ref{sec:uindpt}, to represent the boundaries between the two colors in the noisy voter model.  It is not surprising that

\begin{thma} \label{uindpt:thm:main}
Under $\mu^{u_i,x}$, the particle model has the same distribution as the color boundaries of  the two color noisy voter model.   In the scaling window $(x,x^2)$, the particle model converges weakly to the CNVM.
\end{thma}

We can now consider the behavior for $u_\g:=(1-\sqrt{1-x^2 \g })/(\g x^2)=1/2+\g x^2/8+O(x^4)$ for $\g \in \R \backslash \{0\}$ and $u_0=1/2$.  Notice that we have $u_{-1}$ is equal to the critical value of $u$, i.e. $u_{-1}=u_c$ and that $u_1$ is equal to the independent value of $u$, i.e. $u_1=u_i$.   We can find the  stationary measure for particle locations at time $t \in \R$ in the scaling window $(x,x^2)$:

\begin{thma} \label{intro:mainthm}
Let $u=u_\g$ with $\g \in \R \backslash \{ 1\}$. In the scaling window $(x,x^2)$ the distributions of the particles along  an arbitrary horizontal line converges weakly to a Pfaffian point processes,$\P^{\g}_0$, with  kernel 
\begin{align} \nonumber
M_{\g}^{res}(y,y)= \left( \begin{array}{cc}
	0 & e(\gamma)\\
	-e(\gamma) & 0 \end{array} \right)
\end{align}
for $y \in \R$ and for $y_1 \in \R$ and for $y_2 \in \R$ with $y_1<y_2$
\begin{align} \nonumber
M^{res}_{\g }(y_1, y_2) =\left( \begin{array}{cc}
	-E_1^\g(y_1-y_2|) & -E_2^\g(|y_1-y_2|) \\
	E_2^\g(|y_1-y_2|) & E_1^\g(|y_1-y_2|) \end{array} \right),
\end{align}
where for 
$e_1(\g)=2+\sqrt{2(1-\g)}$, $e_2(\g)=(2-\sqrt{2(1-\g)})\mathbb{I}_{\g>-1}+(-2+\sqrt{2(1-\g)})\mathbb{I}_{\g<-1}$ and $\alpha>0$, we have
\(\nonumber
e(\gamma)=\frac{1}{\pi i} \int_{e_2(\g)}^{e_1(\g)}  \frac{2-2 \g+p^2 }{2\sqrt{4+8\g+4 \g ^2-12p^2+4\g p^2+p^4}} dp,
\)
\(\nonumber
	 E_1^\g (\alpha)= -\frac{1}{\pi i } \int_{e_2(\g)}^{e_1(\g)} \frac{2 p e^{-\alpha p} }{\sqrt{4+8\g+4 \g ^2-12p^2+4\g p^2+p^4}} dp
\)
and
\(\nonumber
	E_2^\g (\alpha)= -\frac{1}{\pi i } \int_{e_2(\g)}^{e_1(\g)} \frac{(-2-2\g-p^2) e^{-\alpha p} }{2\sqrt{4+8\g+4 \g ^2-12p^2+4\g p^2+p^4}} dp.
\)
When $\g>1$, we have $\overline{e_1(\gamma)}=e_2(\gamma)$.  In the special case when $\g=-1$, we have $e(-1)=2/\pi$, 
\(\nonumber
	E_1^{-1} (\alpha)=B_I(0,4\alpha)-S_L(0,4\alpha)=\frac{1}{\pi} \int_0^\pi e^{-4 \alpha \sin \theta}d\theta
\)
and 
\(\nonumber
E_2^{-1} (\alpha)=B_I(0,4\alpha)-S_L(-1,4\alpha)=-\frac{1}{\pi i} \int_0^\pi e^{i\theta-4 \alpha \sin \theta}d\theta,
\)
where $B_I(n,z)$ is the modified Bessel Function of the first kind of order $n$ and $S_L(n,z)$ is the modified Struve function of order $n$.
\end{thma}

Recall that a \emph{Pfaffian point process} is the analog of a determinantal point process for an anti-symmetric matrix whose entries are given by some integral kernel with the determinant replaced by a Pfaffian. 
 

 From Theorem \ref{intro:mainthm}, we can determine the covariance between two particles on the same horizontal line which is given by
 \( \label{intro:covariance}
 	C(\alpha,\g):=E_1^\g (\alpha)^2-E_2^\g(\alpha)^2
\)
for $\g \not = 1$, where $\alpha$ is the distance between two particles.  For $\g = 1$, we have $C(\alpha,1)=0$ for all $\alpha \geq 0$.  Figure \ref{intro:pic:covariance} shows a plot of this function for $\alpha=0.2$.  It can be shown that $C(\alpha,\g)$ is positive if $\g<1$, is negative if $\g>1$ and is 0 if $\g=1$.

  \begin{figure}
\begin{center}
\epsfsize220pt
$$\scalebox{1}{\epsfbox{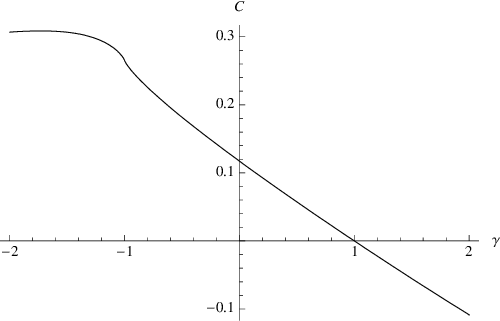}}$$
\caption{A plot of $C(0.2,\g)$, defined in (\ref{intro:covariance}), for $g\in (-2,2)$.  }
\label{intro:pic:covariance}
\end{center}
\end{figure}

As in \cite{cha:88}, define the correlation length to be
\( \label{intro:def:corlength}
	\xi(\g):= -\lim_{\alpha \to \infty} \frac{\alpha}{\log ( |C(\alpha, \gamma)|)}. 
\)
This measures the rate of decay of the covariance.  In Section \ref{sec:inbet}, we prove the following lemma.
\begin{lemma}\label{intro:lem:corlen}
\begin{align} \nonumber
	\xi(\g)=	\left\{ \begin{array}{ll}
				\frac{1}{2 e_2(\g)} & \g<-1\\
				\infty & \g=-1 \\
				\frac{1}{2 e_2(\g)}& -1<\g<1 \\
				0 & \g=1\\
				\frac{1}{4} & \g>1
			\end{array} \right.
\end{align}
where $e_2(\g)$ is defined in Theorem \ref{intro:mainthm}.
\end{lemma}

The  discontinuity at $\g=1$ is due to the definition of the correlation length not to any explicit change in phase behavior.  One further interesting consequence of the above lemma is the following: at criticality, there is still scale invariance after the scaling window $(x,x^2)$ has been taken. 
 
Finally we make the remark that for $u=u_\g$, the creations (and annihilations) form a dense set regardless of the value of $\g$ in the scaling window $(x,x^2)$.  This is a consequence of 
 \(\nonumber
	 \mu^{u_\g,x} ( \# \mbox{ of Creations per fundamental domain})=\Theta (x^2)
\) 
where $\Theta(x^2)$ means bounded above by $C x^2$ and below by $c x^2$ for small $x$.

\subsection{Overview of the paper}

We begin by giving formulas for the dimer model on the Fisher lattice, the particle model probability space and a inclusion-exclusion formula for Pfaffian point processes.  This is all contained in Section \ref{sec:no2}.  In Section \ref{sec:exp}, we calculate exactly the expected number of dimers covering {\bf a}, {\bf b} edges and creations in the thermodynamic limit.  This leads to proving Theorem \ref{expect:thm:behave}.   In Section \ref{sec:hpart}, we focus on the behavior of the model when $u=u_c$ and provide the proof of Theorem \ref{intro:mainthm} when $\g=-1$.    In Section~\ref{sec:ulessthanuc}, we prove Theorem~\ref{uless:thm:PPP}.   The case when $u=u_i$ is considered in Section \ref{sec:uindpt}, and there, we show the proof that the dimer model when $u=u_i$ is equal in distribution to the noisy voter model.  This leads to the convergence to the Continuum Noisy Voter Model.   In Section \ref{sec:inbet}, 
we prove the rest of Theorem \ref{intro:mainthm} (for $\g \not = 1, -1$) and also find an explicit expression for the correlation lengths for all values of $\g$.  We also provide an Appendix which gives more of the details of the model as well as specific examples of the computations we have used. Finally, we mention that some of the more technical computations have relied on using computer algebra software.

{\bf Acknowledgements:} I am very grateful for the help of Richard Kenyon, who provided his valuable insight through countless hours of discussion.  I'm also very thankful for the aid of C\'{e}dric Boutillier who gave a lot of comments and assistance while writing this paper.  I would like to thank David Brydges for many discussions on statistical mechanical models, Jon Warren for pointing out the Brownian Web references, Zhongyang Li for comments on this paper, Richard Arratia for sending a copy of his PhD thesis, \cite{Arr:79} and Ben Young for some useful comments.   I also very thankful for the anonymous referee whose comments and suggestions dramatically improved this paper.  The author is supported by the grant KAW 2010.0063 from the Knut and Alice Wallenberg Foundation.

\section{Setup}  \label{sec:no2}

\subsection{Partition Function, Local Probabilities and Notation}  \label{dimer:sect:fisher}

Here, we explain some of the theory of dimer models.  
All the results in this subsection were first proved in \cite{bou:08} for a certain class of dimer model on non-bipartite graphs and were extended in \cite{Li1:10, Li:10} to the setting we require.

We refer to the graph $G_1$ as the fundamental domain. Orient the edges of $G_1$  so that the number of counter clockwise edges per face is odd.     This is called the {\it Kasteleyn orientation}. Such an orientation is given in Figure \ref{dimer:lattice} for the fundamental domain.   Assume that $G$ (the infinite planar graph) has the periodic Kasteleyn orientation induced by $G_1$.  As $G$ is $\mathbb{Z}^2$ periodic, we can write any vertex of $G$ as a vertex of $G_1$ plus  a translation of the fundamental domain, where the first co-ordinate of the translation is in the direction of the vector $(1,1)$ and the second co-ordinate is in the direction of the vector $(-1,1)$ (i.e.  $v+(x,y)$ is the vertex $v$ in the fundamental domain translated by  $(x-y,x+y)$).

Let $V(G_1)=\{v_i\}$ be the set of vertices of $G_1$ labeled as per Figure \ref{dimer:lattice}. Let $\g_{1}$ (and resp. $\g_{2}$) be a path in the  dual of $G_1$, winding around the torus in the direction of the vector $(-1,1)$  (and (1,1)).   Write $e_{ij}$ for the weight of the edge $\overrightarrow{v_i v_j}$. For parameters $z$ and $w$, let $w_{ij}=z^{n_1}w^{n_2} e_{i j}$, where 
\begin{align} \nonumber
	n_1 (\mbox{resp. }n_2)=		\left( \begin{array}{ll} 
		1 & \overrightarrow{v_i v_j} \mbox{ crosses } \g_1 (\mbox{resp. } \g_2) \mbox{ from below} \\
		0 & \overrightarrow{v_i v_j} \mbox{ does not cross } \g_1 (\mbox{resp. }\g_2)  \\
		-1 & \overrightarrow{v_i v_j} \mbox{ crosses } \g_1 (\mbox{resp. }\g_2) \mbox{ from above} \\
\end{array} \right.
\end{align}
Let $K(z,w)$ be the $6 \times 6$ anti-symmetric matrix given by
\begin{align} \nonumber
	(K(z,w))_{i,j} = \left\{ \begin{array}{ll}
				w_{ij} & \mbox{if } v_i \sim v_j, \mbox{ and } v_i \to v_j \\
				-w_{ij} & \mbox{if } v_i \sim v_j, \mbox{ and } v_j \to v_i\\
				0 & \mbox{otherwise} \end{array}\right.
\end{align}
where $v_i \sim v_j$ means that $v_i$ and $v_j$ share an edge and $v_i \to v_j$ denotes an arrow from $v_i$ to $v_j$ in the Kasteleyn orientation.  This gives

\begin{align} 
K(z,w)=\left( \begin{array}{cccccc}
0 &  -1 &  1 &  0 &  0 &  b/z \\
1 &  0 &  -1 &  0 &  b/w &  0 \\ 
-1 &  1 &  0 &  a &  0 &   0 \\
0 &  0 &  -a &  0 &  1 &  -1 \\
0 &  -b w &  0 &  -1 &  0 &  1 \\
-b z &  0 &  0 &   1 &  -1 &  0  
\end{array} \right). \label{Kasteleyn}
\end{align}
where $a= x$ and $b=u x$.  

The relevance of the above discussion is that we can write the free energy and probabilities of observing a local configuration  in the thermodynamic limit as expressions involving $K(z,w)$.  The characteristic polynomial for the fundamental domain, $P(z,w)$, is defined to be the determinant of $K(z,w)$. This gives
\begin{align}
	P(z,w)=a^2 +2b^2+a^2b^4 + ab(1-b^2)(z+\frac{1}{z} +w + \frac{1}{w}) +b^2(1-a^2) (\frac{z}{w} + \frac{w}{z}) \label{charpoly}
\end{align}
The free energy can be written in terms of the characteristic polynomial, namely
\begin{align} 
 	-\log Z :=  -\lim_{n \to \infty}\frac{1}{n^2} \log Z(G_n)=-\frac{1}{2(4 \pi ^2)} \int_{|z|=1} \int_{|w|=1} \log P(z,w) \frac{ dw}{w} \frac{dz}{z} \label{setup:logZ}
\end{align}
Let $E=\{e_1=u_1 u_2, \dots , e_m=u_{2m-1} u_m \}$ be a subset of edges, with each $u_i$ representing a distinct vertex.  The probability of observing $E$ in thermodynamic limit is given by 
\( \label{sec2:localstats}
	\P_\mu(e_1, \dots, e_m) = \left( \prod_{i=1}^m K(u_{2i},u_{2i+1}) \right) \Pf \left( (K^{-1} (u_i,u_j) \right)^T)_{1 \leq i ,j \leq 2 m}
\)
where $T$ represents the transpose of a matrix and assuming $v$ and $\tilde{v}$ are in the same fundamental domain, 
\(  \nonumber
	K^{-1} (v,\tilde{v}+(x,y)) = \frac{1}{(2 \pi i)^2} \int_{\mathbb{T}^2} K (z,w)_{v,\tilde{v}}^{-1} w^x z^y \frac{dw}{w} \frac{dz}{z}.
\)
This formula is known as the {\it local statistics formula}, originally formulated for dimer models on bipartite graphs \cite{Ken:97}. 

For the rest of the paper, we use the following notation.  We denote $v_i(n,m)= v_i -(n,m)$ where $i \in \{1, 6\}$. In other words, $v_i(n,m)$ is the vertex $v_i$ in the $(-n,-m)$ fundamental domain.  Note that $v_i=v_i(0,0)=v_i+(0,0)$. We also let
\begin{equation}
	H_u(m,n)=\frac{1}{(2\pi i )^2} \int_{\mathbb{T}^2} \frac{z^m w^n- (-1)^{m+n} \delta_{u_c}(u) }{P(z,w)} \frac{dz}{z} \frac{dw}{w}. \label{definitionofH}
\end{equation}

\subsection{Critical and independent values}

In this subsection, we introduce the critical and independent values.  More explicit explanations of how these are obtained can be found in Appendix \ref{app:secA}.   

A dimer model is said to be {\it critical} if there exists values of $(z,w) \in \mathbb{T}^2$ such that $P(z,w)=0$.  The dimer model studied in this paper, $\mu^{u,x}$ is critical when $u=(-1+\sqrt{1+x^2})/x^2$ and $u=(1+\sqrt{1+x^2})/x^2$.  We define $u_c$, the \emph{critical value},  to be $(-1+\sqrt{1+x^2})/x^2$, which lies in $(0,1)$ for $x<1$.

This particular dimer model, $\mu^{u,x}$, has other values of interest, namely $(1-\sqrt{1-x^2})/x^2$ and $(1+\sqrt{1-x^2})/x^2$.     We define $u_i$, the \emph{independent value}, to be  $(1-\sqrt{1-x^2})/x^2$.

\subsection{Particle Model probability space} \label{dimer:sec:prob}

In this subsection, we define the underlying probability space for the particle model by use of a certain metric.  We use the same space defined in \cite{aiz:99} and \cite{Fed:08}.  Let $\dot{\R}^2$ denote $\R^2 \cup \{ \infty \}$, i.e. the compactification of $\R^2$.  Let $d$ be the spherical metric, which is defined to be $d: \dot{\R}^2 \times \dot{\R}^2 \to \R$ by
\( \nonumber
	d(u,v)=\inf \int (1+|\phi|^2)^{-1} ds
\)
where the infimum is over all smooth curves connecting $u$ and $v$, $\phi$ is parameterized by the arc length $s$ and  $|\cdot|$ is the Euclidean metric. 

Let $\mathcal{H}$ be the space of all collections of curves in $\dot{\R}^2$.  We need to define a metric on $\mathcal{H}$ in order to determine its open sets. Let $\g_1$ and $\g_2$ be two curves with their parametrization at time $t$ given by $\g_1(t)$ and $\g_2(t)$.  Consider a complete separable metric space $S$ of continuous curves in $\dot{\R}^2$, with distance given by
\(\nonumber
	D(\g_1, \g_2)= \inf \sup_{t \in [0,1]} d( \g_1(t), \g_2(t))
\)
where the infimum is over all choices of parametrizations of $\g_1$ and $\g_2$.  The collection of all closed sets of $S$ is exactly $\mathcal{H}$.  Let $\mathcal{F}_1$ and $\mathcal{F}_2$ denote two closed sets of curves.  Then, $D$ induces the Hausdorff metric, $D_{\mathcal{H}}$, defined by
\(\nonumber
	D_{\mathcal{H}}(\mathcal{F}_1, \mathcal{F}_2)) < \e \Rightarrow ( \forall \g_1 \in \mathcal{F}_1, \exists \g_2 \in \mathcal{F}_2 \st D(\g_1,\g_2) <\e \mbox{ and vice versa})
\)
Therefore, $(\mathcal{H}, D_{\mathcal{H}})$ defines a complete separable metric space.  Let $\mathcal{F}_{\mathcal{H}}$ denote the Borel $\sigma$ field associated with the metric $D_\mathcal{H}$.  Let $\mathcal{M}_\mathcal{H}$ be the space of probability measures taking values in $ (\mathcal{H},\mathcal{F}_{\mathcal{H}})$.  The measures of the particle model are in the space $\mathcal{M}_{\mathcal{H}}$ for all values of $u$ and $x$.

\subsection{Distribution of many particles}\label{global:invedge}

This self-contained subsection gives an inclusion-exclusion formula for the Pfaffian.  More details on Pfaffians can be found in \cite{god:93}.  For $1 \leq i \leq n$, let $X_i$ represent the existence of a particle at $x_i$, where $x_i \in \Z$ or $x_i \in \R$.  This is equivalent to the appropriate {\bf a} edge not being covered by a dimer in the underlying dimer configuration.     The complement of  $X_i$ shall be denoted $X_i^c$ which represents  a dimer covering the appropriate {\bf a} edge in the underlying dimer configuration.  Let $\P$ denote the law of observing the particles. Let $K_n(x_1,\dots,x_n)=\{k_{i,j}\}_{i,j=1}^{2n}$ represent the $2n$ by $2n$ anti-symmetric matrix with $k_{2i-1,2i}=v=\P(X_i^c)$ for $1 \leq i \leq n$ and $k_{i,j}$ denote the appropriate inverse Kasteleyn entries multiplied by $x$, the edges weight of an {\bf a} edge,  so that 
\(
	\P(X_1^c,\dots , X_n^c)=\Pf (K_n(x_1, \dots ,x_n)) \label{hp:sc:eq:nop}
\)
The above definition of $K_n$ ensures that the edge weights are encoded into the Pfaffian (by linearity of the Pfaffian).  Define,  $K_n(\hat{x}_1,\dots, \hat{x}_k, x_{k+1},\dots,x_n)=\{ \hat{k}_{i,j} \}_{i,j=1}^{2n}$, with $\hat{k}_{2i-1,2i}=-1+v$ for $1\leq i \leq k$ and $\hat{k}_{i,j}=k_{i,j}$.
Let $I^{sk}_n=\{a_{i,j}\}_{i,j=1}^{2n}$ represent the antisymmetric matrix with $a_{2i-1,2i}=1$ for $i\in\{1,\dots,n\}$ and $a_{i,j}=0$ otherwise. Then

\begin{lemma}\label{pfaffexp}
\( \nonumber
	\P(X_1,\dots , X_n)=\Pf(I^{sk}_n-K_n(x_1,\dots,x_n))=(-1)^n\Pf (K_n(\hat{x}_1, \dots, \hat{x}_n))
\) 
\end{lemma}

The proof does not require the edges covered to have the same probability of being covered.   The result is stated in this fashion for convenience later in the paper.

\begin{proof} The proof follows by using an iteration, with the key step relying on using the recursive definition of the Pfaffian.  Consider 
\( 
	\P(X_1,X_2^c, \dots , X_n^c)=\P(X_2^c, \dots , X_n^c)-\P (X_1^c, \dots , X_n^c) \label{hp:sc:eq:dec1}
\)
with $\P(X_2^c, \dots , X_n^c)=\Pf ( K_{n-1}(x_2, \dots, x_n))$ and $\P (X_1^c, \dots , X_n^c)$ is given in (\ref{hp:sc:eq:nop}).  On the other hand, by using the recursive definition of the Pfaffian and letting $A=K_n(\hat{x}_1, x_2, \dots, x_n)$
\begin{align}
	\Pf(K_n(\hat{x}_1, x_2, \dots, x_n))= (-1+v) \Pf(A_{\hat{1}, \hat{2}})+\sum_{i=3}^{2n} (-1)^{i} k_{1, i}\Pf (A_{\hat{1},\hat{i}} ) \label{hp:sc:eq:dec2}
\end{align}
where $A_{\hat{1},\hat{i}}$ is the matrix obtained by removing both the 1st and $i^{th}$ row and column.  Notice that $ \Pf(A_{\hat{1}, \hat{2}})$ is exactly $\Pf( K_{n-1}(x_2, \dots, x_n))$ while $v \Pf(A_{\hat{1}, \hat{2}})+\sum_{i=3}^{2n} (-1)^{i} k_{1, i}\Pf (A_{\hat{1},\hat{i}} )=\Pf (K_n(x_1, x_{2},\dots,x_n))$.  Therefore, substituting (\ref{hp:sc:eq:dec2}) into (\ref{hp:sc:eq:dec1}) gives 
\(
	\P(X_1,X_2^c, \dots , X_n^c)=-\Pf(K_n(\hat{x}_1, x_2, \dots, x_n)). \label{hp:sc:eq:iter}
\)
Applying the steps to obtain (\ref{hp:sc:eq:iter}) iteratively gives the result.

\end{proof}

 
 \section{Thermodynamic Limit} 
 \label{sec:exp}
 
By introducing a generic method, we compute the expected number of {\bf b} edges, particles and creations per fundamental domain in the thermodynamic limit.   We use these expressions to prove Theorem \ref{expect:thm:behave}.

 
 \subsection{Expectations for $0<x<1$ and $0<u<1$}
 
   Let $N_b$ denote the number of dimers covering {\bf b} edges per fundamental domain.  By the setup of the model, there are two {\bf b} edges in each fundamental domain. Let $N_{a^c}$ denote the number of {\bf a} edges not covered by a dimer per fundamental domain.   Let $N_X$ denote the number of creation singularities per fundamental domain.  Write $\E[ \cdot]$ for the expected value of each of these quantities with respect to $\mu^{u,x}$.  In this subsection, we find exact results for $\E [N_b]$, $\E[N_{a^c}]$ and $\E[N_X]$.  Figure \ref{sec2:fig:diff} shows a plot of $\E[N_b]$ and $\E[N_{a^c}]$ for $x=0.5$.     

   For $G_n$, define
\( \nonumber
	\overline{N}_b^n=\frac{\sum_{\mbox{{\bf b} edges in }G_n} \mathbb{I}_b}{n^2}=\frac{N_b^n}{n^2},
\)
\( \nonumber
	\overline{N}_{a^c}^n=1-\frac{\sum_{\mbox{{\bf a} edges in }G_n} \mathbb{I}_a}{n^2}=1-\frac{N_{a}^n}{n^2}
\)
and
\( \nonumber
	\overline{N}_{X}^n=\frac{\sum_{X\mbox{ in }G_n} \mathbb{I}_X}{n^2}=\frac{N_X^n}{n^2},
\)
where $N_b^n$ is the number of {\bf b} edges covered by dimers for a dimer covering of $G_n$, $N_{a}^n$ denotes the number {\bf a} edges covered by dimers for a dimer covering of $G_n$ and $N_X^n$ denotes the number of creations for a dimer covering of $G_n$.

  \begin{figure}
\begin{center}
\epsfsize220pt
$$\scalebox{1}{\epsfbox{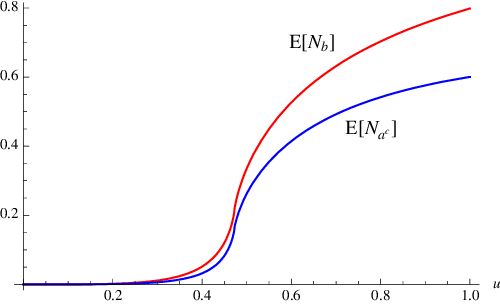}}$$
\caption[Expectations of $a$ and $b$ edges]{A plot of $\E[N_b]$ and $\E[N_{a^c}]$ against $u \in (0,1)$ for $x=0.5$.  Notice that the difference between the two curves is $2 \E[N_X]$.}
\label{sec2:fig:diff}
\end{center}
\end{figure}
 
\begin{thma} \label{expect:thm:expectb} For $0<x<1$ and $0<u<1$
\begin{align} \nonumber
	\E[N_b] 	=  -\frac{2u^2 x^2}{1-u^2 x^2} +	\left\{ \begin{array}{ll} 
					f(x,u,r) & \mbox{if } u<u_c \\
					f(x,u,1) &  \mbox{if } u=u_c \\
					f(x,u,1/r) &   \mbox{if } u>u_c\\ \end{array}\right.
\end{align}
where $f(x,u,r)$ is a continuous function given by
\begin{align} \nonumber
	f(x,u,r)	&=		C(x,u)((u^2 +r^2 u^2) (-1 + x^2)  + r (1 + u^4 x^4 - 2 						u^2 (1 + 2 x^2))\mathcal{K}\left( k \right) \\
		 	&+		C(x,u)(-1 + r^2) u^2 (-1 + x^2)\left(\Pi \left( r k , k   \right) - \Pi\left( \frac{ k }{r} , k \right) \right)
\end{align}
for $r\ne 1$ where 
\(\nonumber
	k=k(r,s)=\frac{s-r}{r s-1},
\)
\(\nonumber
	C(x,u)= \frac{ \sqrt{r s} (1 + u^2 x^2) }{\pi r ( r s-1) u^2 (-1 + x) (1 + x) (-1 + u^2 x^2)},
\)
\(\nonumber
	r=r(u,x)=\frac{1-2u^2 + u^4 x^4 +(1 - u^2 x^2) \sqrt{(1-2u^2+u^2 x^2)(1+2u+u^2 x^2)}}{2 u^2(1+x)^2},
\)
\(\nonumber
	s=s(u,x)=\frac{1-2u^2 + u^4 x^4 +(1 - u^2 x^2) \sqrt{(1-2u^2+u^2 x^2)(1+2u+u^2 x^2)}}{2 u^2 (1-x)^2}, 
\)
$\mathcal{K}(k)$ represents the complete elliptic integral of the first  kind with modulus $k$ and $\Pi(n,k)$ represents the complete elliptic of the   third kind with characteristic $n$ and modulus $k$.
When$u=u_c$, we have  
\( \nonumber
	f(x,u_c,1)= 2 \sqrt{1+x^2} \left( 1- \frac{2}{\pi} \mathrm{arccotan}(x) \right)
\)
Furthermore, when $u=u_i$ we have
\begin{equation*}
      \E[N_b]=1-\sqrt{\frac{1-x}{1+x}}. 
\end{equation*}

\end{thma}
 The elliptic integrals are defined for $u>u_c$ by an analytic continuation argument (see \cite{LAW:89}).

\begin{proof}
As we have defined $\log Z(G_n)$ to be the partition function of the dimer covering on the Fisher lattice embedded in $G_n$, we have that  $-b\frac{\partial}{\partial b} \log Z(G_n)$ is equal to the expected number of {\bf b} edges for the dimer covering on the Fisher lattice on $G_n$.  Therefore, we have
   \begin{equation}
	\E[\overline{N}_b^n]=-\frac{b}{n^2} \frac{\partial}{\partial b} \log Z(G_n).
   \end{equation}
In order to compute the $\E[N_b]$ we have to take limits as $n$ tends to infinity.  By \cite{Li1:10, Li:10}, this limit exists for the dimer model on the Fisher lattice and we have
\begin{equation}
      \E[N_b]=- b\frac{\partial}{\partial b} \log Z
\end{equation}
where $\log Z$ is defined in~\eqref{setup:logZ}.  It remains to compute $- \frac{\partial}{\partial b} \log Z$.  This is done in Lemma  \ref{app:therm:exp}.

\end{proof}

\begin{lemma} \label{expect:prop:Na}
For $0<x<1$ and $0<u<1$,
\begin{align}  \nonumber
	\E[N_{a^c}]	= -\frac{x^2}{1-x^2}	+\left\{ 	\begin{array}{ll}
						g(x,u,r) 	& \mbox{if } u<u_c \\
						g(x,u,1)& \mbox{if } u=u_c\\
						g(x,u,1/r)& \mbox{if } u>u_c
					\end{array} \right.
\end{align}
where 
\begin{align} \nonumber
	g(x,u,r)	&= 		D(x,u) (u^2(1+x^2)+r^2 u^2 (1+x^2)-r(1-2u^2+u^4 x^4)) K 							\left( k  \right) \\
			&+		D(x,u) (-1+r^2) u^2(1+x^2) \left(\Pi \left(r k , k  \right) - 								 \Pi\left(\frac{k }{r} , k \right) \right) \nonumber
\end{align}
for $r \not =1$, where $D(x,u)=C(x,u)(-1+u^2 x^2)/(1+u^2 x^2)$, $k$, $C(x,u)$, $r$, $s$, $K$ and $\Pi$ are defined in Theorem \ref{expect:thm:expectb}.  When $u=u_c$, we have
\begin{align} \nonumber
	g(x,u_c,1)=  \frac{1+ x^2}{1-x^2} \left(1- \frac{2}{\pi } \mathrm{arccotan}(x)) \right).
\end{align}
Furthermore, for $u=u_i$ we have
\begin{equation}
      \E[N_{a^c}]= \frac{x}{1+x}.
\end{equation}
\end{lemma}

\begin{proof}
The proof follows from a similar argument used in Theorem \ref{expect:thm:expectb} and a similar calculation to Lemma \ref{app:therm:exp}.
\end{proof}


The next lemma shows that $\E [N_X]$ can be computed using $\E [N_b]$ and $\E [N_{a^c}]$.

\begin{lemma} \label{expect:lem:sing}
\( \nonumber
	\E[N_X]=\frac{1}{2} \left( \E[N_b]-\E[N_{a^c}] \right).
\)
\end{lemma}

\begin{proof}

For $G_n$, the number of dimers covering {\bf b} edges which do not belong to creations, per fundamental domain, is equal to  $\overline{N}_{a^c}^n$.   We have that $\overline{N}_{b}^n -\overline{N}_{a^c}^n$ is equal to the proportion of fundamental domains with dimers covering both {\bf b} edges of that fundamental domain.  Hence, we have $2\overline{N}_{X}^n= \overline{N}_{b}^n-\overline{N}_{a^c}^n$.  Taking expectations gives
\(
	\E[\overline{N}_{X}^n]=\frac{1}{2} \left( \E [\overline{N}_{b}^n]-\E [ \overline{N}_{a^c}^n] \right).
\)
By \cite{Li1:10, Li:10}, taking limits as $n$ tends to infinity, the left-hand-side converges to $\E [ N_X]$ and the right-handside converges to $(\E[N_b]-\E[N_{a^c}])/2$.

\end{proof}

As a direct consequence, we have:
\begin{coro} \label{expect:coro:NX}
For $0<x<1$ and $0<u<1$
\begin{align} \nonumber
\E [N_{X}]	= - \half \frac{x^2(1-2u^2+u^2 x^2)}{(1-x^2)(1- u^2 x ^2)} +	\left\{ 	\begin{array}{ll}
						\half \left( f(x,u,r)-g(x,u,r) \right)	& \mbox{if } u<u_c \\
						\half \left(f(x,u,1)-g(x,u,1)\right)	& \mbox{if } u=u_c\\
						\half \left(f(x,u,1/r)-g(x,u,1/r)\right)& \mbox{if } u>u_c
					\end{array} \right.
\end{align}
where the functions $f$ and $g$ are defined in Theorem \ref{expect:thm:expectb} and Lemma \ref{expect:prop:Na}.
\end{coro}

We conclude this subsection with some remarks about the model and the modulus of the elliptic integrals given in Theorem~\ref{expect:thm:expectb}.  The terms $r(u,x)$ and $s(u,x)$ are two of the roots of a quartic polynomial, $A(v)$ which is defined ~\eqref{app:expect:thmproof:av}.  The other two roots of $A(v)$ are given by $1/r(u,x)$ and $1/s(u,x)$.  More details of the roots and $A(v)$ can be found in Appendix~\ref{app:secA}.  The modulus of the elliptic integrals is real when $u<u_i$ and is complex when $u>u_i$.  This follows from the fact that for $u<u_i$, $r(u,x)$ and $s(u,x)$ are real and for $u>u_i$ we have $1/r(u,x)$ and $s(u,x)$ are complex.  Note that when $u=u_c$, we have $r(u,x)=1$ and we no longer have any elliptic integrals in the expression of $\E[N_b]$.  Furthermore, when $u=u_i$, the modulus of the elliptic integrals is zero.  Therefore, the model has five cases, namely; $u<u_c$, $u=u_c$, $u_c<u<u_i$, $u=u_i$ and $u>u_i$.

\subsection{Expectations in the limit $x\to 0$}

In this subsection, we first state two Lemmas which compute the expansions of  $\E [N_b]$ and  $\E [N_X]$ around $x=0$.  Their proofs are given in Appendix~\ref{app:localprobs}.  We also give the proof of Theorem \ref{expect:thm:behave}.

\begin{lemma} \label{expect:lem:expand}
\begin{align}
\E [N_b] 	= 	\left\{		\begin{array}{ll}
					\frac{2u^2(4u^2-1+\sqrt{1-4u^2})}{1-4u^2} x^2+O(x^3)  & \mbox{if } u<u_c \mbox{ (fixed)} \\
					\frac{2 x}{\pi} +O(x^2)	& \mbox{if } u=u_c \\
					x+O(x^2)	&\mbox{if } u=u_i	\\
					2- 2\frac{\theta}{\pi} +O(x) & \mbox{if } u>u_i  \mbox{ (fixed)} \\
					\end{array} \right.
\end{align}
where $\t=\arccos(-1/(2u))$.  Furthermore, we have $\E [N_b]=\E [N_{a^c}] +O(x^2)$ for $u=u_c$ and $u=u_i$ and $\E [N_b]=\E [N_{a^c}] +O(x)$ for $u>u_i$.
\end{lemma}


The proof is given in Appendix~\ref{app:localprobs}.  It is based on the analysis of the roots of $A(v)$, defined in~\eqref{app:expect:thmproof:av}, and the expansions of the elliptic integrals in Theorem~\ref{expect:thm:expectb}.   The expansions of the elliptic integrals we used do not hold when $u$ tends to $1/2$ simultaneously as $x$ tends to 0. 
The expansion for $\E [N_X]$ can be also computed for certain values of $u$.  The proof of the following lemma is given in Appendix~\ref{app:localprobs}.

\begin{lemma} \label{exp:lem:sing}
\begin{align} \nonumber
	\E [N_X]	= \left\{	\begin{array}{ll}
				\frac{x^2}{2}(1-2u^2-\sqrt{1-4u^2})+O(x^3) & \mbox{if } u<u_c \mbox{ (fixed)}\\
				\frac{x^2}{4}+O(x^3) 	& \mbox{if } u=u_c \\
				\frac{x^2}{4} +O(x^3)& \mbox{if } u=u_i	 \\
				(-\frac{1}{\pi} \sqrt{4u^2-1}) x^2 \log x +O(x^2). &\mbox{if } u>u_i \mbox{ (fixed)}
				\end{array} 	\right.
\end{align}
\end{lemma}

We can now prove Theorem \ref{expect:thm:behave}.

\begin{proof}[Proof of Theorem \ref{expect:thm:behave}]
The probability of any cylinder set under $\mu^{u,x}$ can be computed, using the Kasteleyn method.  Taking the limit as $x \to 0$ gives the probability of the cylinder set under the measure $\mu^{u,0}$.  The family of measures $\mu^{u,x}$ are automatically tight by constructing the measures $\mu^{u,x}$ through the thermodynamic limit.  It remains to establish the behavior under $\mu^{u,0}$, 

For $u \in (0,1)$, using Markov's inequality we obtain
\( \label{exp:main:cheby}
	\lim_{x \to 0} \mu^{u,x}(N_X=0)= \mu^{u,0}(N_X=0) =1.
\)
As the distribution of the creations is the same as the distribution of the annihilations, 
(\ref{exp:main:cheby}) implies with probability 1, there are no annihilations or creations for $u \in (0,1)$.   By Lemma \ref{expect:lem:expand}, in the limit $x \to0$ and $u > 1/2$, particles have density $2- 2 \t/ \pi$.  For $u \leq 1/2$, we can apply Markov's inequality to give
\( \nonumber
	\mu_{u,0}(N_{a_c}=0) =1
\)
\end{proof}


We end this section with a `duality statement' which has important implications in understanding the underlying Ising model. For the following lemma, we  drop the assumption that $u \in (0,1)$.  Let $\mathcal{A}$ denote any configuration of {\bf a} edges not covered by dimers in the dimer configuration
\begin{lemma} \label{exp:lem:dual}
For $u<1/x$ and $x>0$, we have
\( \nonumber
	\mu^{u,x}(\mathcal{A})= \mu^{1/(u x^2)}(\mathcal{A}).
\)
\end{lemma}

From this lemma, one can immediately determine that the two critical points $(-1+\sqrt{1+x^2})/x^2$ and $(1+\sqrt{1+x^2})/x^2$ are equivalent (in distribution) in terms of the particle model because  $(1+\sqrt{1+x^2})/x^2=1/(u_c x^2)$. 

Furthermore, if we choose $x=1$ and $u=u_c$, the particle model is equivalent to the two dimensional critical ferromagnetic Ising model on the square grid.  Indeed, choosing $x=1$, there is no interaction from the {\bf a} edges because the spins on each horizontal line have no contribution to the Hamiltonian (\ref{intro:isinghamiltonian}).   By Lemma \ref{exp:lem:dual}, when $x=1$ the particle model with $u=u_c$ is equivalent to the particle model with $1/u_c$.  When we take $u=1/u_c=\sqrt{2}+1$, we obtain the ferromagnetic Ising model on the square grid  at critical temperature (due to the anti-ferromagnetic and ferromagnetic Ising models are dual to each other on the square grid).

\begin{proof}[Proof of Lemma \ref{exp:lem:dual}] 

For this proof of this lemma, we keep track of the chosen parameters by writing $[\alpha,\beta]$ after each expression, where $\log \alpha^{-1}/\log \beta^{-1}$ is the anisotropy and $\log 1/\beta$ is the temperature of the Ising model, i.e. the weight of the {\bf a} edge is $\alpha$ and the weight of the {\bf b} edge is $b= \alpha \beta$ in this case.  Let $C(z,w)_{i,j}[u,x]$ denote the cofactor matrix of $K(z,w)[u,x]$.  The relevant entries of $C(z,w) [u,x]$ are listed in Appendix \ref{app:cofactor}.  Finally, let $K^{-1}(v_i,v_j(n,m))[u,x]$ for $i, j \in \{3,4 \}$ be the inverse Kasteleyn entry between vertex $v_i$ and $v_j(n,m)$.   The weight of the {\bf a} edges is the same in each regime. It remains to check that Pfaffian for computing local probabilities remains unchanged too.  

By a direct substitution we have
\( \label{eqn:charpolyduality}
	P(z,w)[u,x]=u^4 x^4 P(-z,-w)[1/(u x^2),x]. 
\)
For $i,j \in \{3,4\}$, by direct substitution, we also have 
\(\label{eqn:minorduality}
	C(z,w)_{i,j}[u,x]= u^4 x^4 C(-z,-w)_{i,j}[1/(u x^2),x]. 
\)
As we have 
\( \nonumber
K^{-1}(v_i,v_j(n,m)) = \frac{1}{(2\pi i)^2} \int_{|z|=1} \int_{|w|=1} \frac{C(z,w)_{i,j}}{P(z,w)} z^n w^m \frac{dw}{w} \frac{dz}{z},
\)
by using (\ref{eqn:charpolyduality}) and (\ref{eqn:minorduality}), we have
\( \nonumber
	K^{-1}(v_i,v_j(n,m))(1/(u x^2), x)=(-1)^{n+m}K^{-1}(v_i,v_j(n,m))(u,x).
\)	
In other words, the Fourier representations of these inverse Kasteleyn entries differ by a factor of $(-1)^{n+m}$.  The factor $(-1)^{n+m}$ always affects an even number of rows and columns and so the Pfaffian remains unchanged by this factor.
\end{proof}




 \section{The particle model for $u=u_c$} \label{sec:hpart}

This section focuses on the critical case in which $u=u_c=\criticall$. A realization of the particle model at this choice of parameterization is depicted in Figure \ref{uc:fig:sim} for $x=0.1$.  

Although the results in this paper for $u=u_c$ do not rely on the choice of vertical scaling, we will impose that the scaling window for $u=u_c$ satisfies Donsker's scaling (i.e. scaling space by the square root of time) and gives \emph{non-trivial paths}, that is, paths with length  bigger than $\e>0$ under the rescaling.  A heuristic argument for our choice of re-scaling is as follows:  suppose we choose the scaling window $(x^\alpha, x^{2 \alpha})$.  It suffices to compute  the expected number of dimers covering {\bf b} edges in a box of size $(x^{-\alpha}, x^{-2 \alpha})$.  Using Lemma \ref{expect:lem:expand}, this is $\Theta(x^{1-3\alpha})$ because the number of fundamental domains in a box of size $(x^{-\alpha}, x^{-2 \alpha})$ is $x^{-3\alpha}$.  We see non-trivial paths in the scaling window $(x^{-\alpha},x^{-2\alpha})$ if the expected number of dimers covering {\bf b} edges is of order $x^{-2 \alpha}$.   Therefore, we choose $\alpha=1$.

This section is organized as follows: first, we introduce local dynamics in the thermodynamic limit which gives some intuition for the model at $u=u_c$. 
We then focus on finding the appropriate correlation kernel in the scaling window $(x,x^2)$ and prove Theorem \ref{intro:mainthm} for $u=u_c$ (or $\g=-1$).  

  \begin{figure}
\begin{center}
\epsfsize220pt
$$\scalebox{1}{\epsfbox{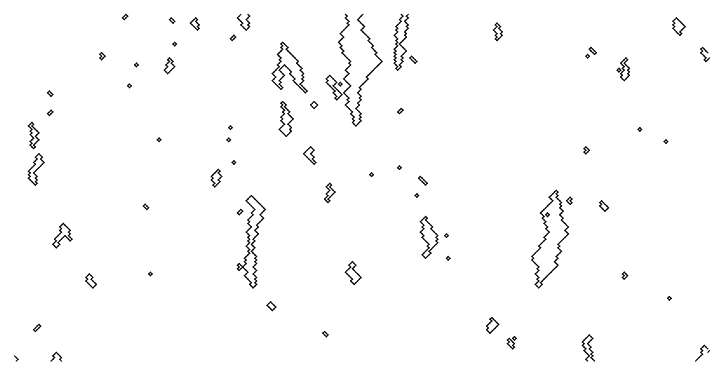}}$$
\caption[A simulation of the particle model: $u=u_c$ and $x=0.1$]{A simulation of the particle model on a 100 by 100 grid with $u=u_c$ and $x=0.1$ made using Glauber dynamics.}
\label{uc:fig:sim}
\end{center}
\end{figure}

\subsection{Local Dynamics}\label{sec:uc:localdynamics}

This subsection concentrates on the behavior of the model at $u=u_c$ before the scaling window is taken. We compute random walk estimates for the trajectories of a pair of particles separated by distance $n$, where $n$ is considered to be smaller than order $1/x$, to provide some intuition of the model before the scaling window is taken.

The entries of the inverse Kasteleyn matrix for $u=u_c$ are linear combinations of $H_u$ defined in~\ref{definitionofH}. We have
\begin{lemma}\label{hp:uc1lemma}
For $m\geq n$ and $m=o(1/x)$ 
\begin{align}
H_{u_c} (m,-n)= \sum_{i=0}^3 x^{i-2} F_i(m,n)+O(x^2)
\end{align}
where 
\(
	F_0(m,n)=(-1)^{m-n} (m+n),
\)
 
\(
	F_1(m,n)=\frac{2 (-1)^{m-n} (m+n)^2 }{\pi }
\)
\(
 F_2(m,n)=-\frac{1}{6} (-1)^{m-n} (m+n) \left(1+2 m^2+4 m n+2 n^2\right)  
\) 
and
\begin{align}
	F_3(m,n)&= \frac{(-1)^{m-n}}{9 \pi} (m + n) \nonumber \\
	&\times (-7 - 20 m^2 + 12 m^4 + 200 m n + 48 m^3 n - 20 n^2 + 
   72 m^2 n^2 + 48 m n^3 + 12 n^4).
\end{align}

\end{lemma}

The proof can be found in Appendix \ref{app:uc:local}.  The expansion also holds for $n \geq m$.   Recall that $v_i (n,m)$ denotes the vertex $v_i$ in the $(-n,-m)$ fundamental domain.  For $(y,t) \in( \mathbb{Z} \times \mathbb{Z}) \cup ( (2\mathbb{Z}+1) \times (2\mathbb{Z}+1))$, let  $U(y,t)$ denote the event that there is a particle at $(-y,t)$ where we think of $y$ as space and $t$ as time. Let $\P$ denote the law of the particles under $\mu_{u_c,x}$.

\begin{prop} \label{hpart:dis:jtprob}
The joint probability of two particles separated by distance $n$ at time $t$ with $n=o(1/x)$, is given by
\(
	\P (U (0,t), U(n,t))=x^2-\frac{4+8n}{ \pi } x^3+O(x^4).
\)
\end{prop}

\begin{proof}
The proof of this proposition is a calculation using the inclusion-exclusion formula given in Lemma \ref{pfaffexp}.   Using Lemma \ref{hp:uc1lemma} and the cofactor matrix given in Appendix \ref{app:cofactor}, we have that
\begin{align}
	K^{-1} (v_3,v_3(n,-n)) &= u^2 x^2 ( H_{u_c} (n+1,-n-1)- H_{u_c}(n-1,-n+1) ) \\
	&=1-\frac{8 n x}{\pi} +(1+4n^2) x^2-\frac{4(-3+34 n +32n^3)x^3}{9\pi}+O(x^4). \label{hpart:lemproof:forma}
\end{align}
We have that $K^{-1} (v_3,v_3(n,-n))=-K^{-1} (v_4,v_4(n,-n))$ by the cofactor matrix given in Appendix \ref{app:cofactor}.  Similarly, we can compute $K^{-1} (v_3,v_4(n,-n))$ and it is given by
\(
	K^{-1}(v_3,v_4(n,-n))=\frac{2}{\pi} -2 n x+\frac{(10+8n) x^2}{3\pi}-3 n x^3+O(x^4). \label{hpart:lemproof:formb}
\)
By the antisymmetry of $K(z,w)$ and the cofactor expansion given Appendix \ref{app:cofactor}, we have $K^{-1}(v_3,v_4(n,-n))=-K^{-1}(v_4,v_3(n,-n))$.

From Lemma~\ref{expect:prop:Na}, we have that 
\begin{equation}
\begin{split}
      \P(U(n,t))=\P(U(0,t))=1-\mu^{u_c,x} (\mbox{{\bf a} edge})=-\frac{x^2}{1-x^2} +\frac{1+x^2}{1-x^2}\left(1-\frac{2}{\pi} \mathrm{arccotan}(x) \right)
\end{split}
\end{equation}
because a particle is present if there is no dimer covering the corresponding {\bf a} edge.  By taking a series expansion of the right hand side of the above equation, we have
\begin{equation} \label{u=u_c:seriesexpprob}
       \P(U(n,t))=\P(U(0,t))= \frac{2}{\pi}x-x^2+O(x^3).
\end{equation}
From the inclusion-exclusion formula (Lemma \ref{pfaffexp}), we can expand the Pfaffian and we obtain
\(  \label{hpart:lemproof:formc}
	\P(U(0,t),U(n,t))=x^2K^{-1}(v_3,v_3(n,-n))^2-x^2 K^{-1}(v_3,v_4(n,-n))^2+ \P(U(0,t)) \P(U(n,t)).
\)
Substituting (\ref{hpart:lemproof:forma}), (\ref{hpart:lemproof:formb}) and the expressions for $\P(U(0,t))$ and $\P(U(n,t))$ into  (\ref{hpart:lemproof:formc}) gives the result.
\end{proof}

Let $U(n,t) \to U(n+1,t+1)$ denote the event that a particle enters the site $(-n-1, t+1)$ from the site $(-n,t)$.   For the dimer model, this corresponds to a local configuration of dimers covering the edges $(v_4,v_5)$ and $(v_6, v_1(-1,0))$ for the appropriate fundamental domain.  However, if the edge $(v_4,v_5)$ is covered, then the edge $(v_6, v_1(-1,0))$ must be covered by a dimer. Therefore, the event $U(n,t) \to U(n+1,t+1)$ in the particle model is equivalent to a dimer covering the edge $(v_4,v_5)$ for the corresponding fundamental domain in the dimer model.  Analogously, we can define $U(n,t) \to U(n-1,t+1)$.  We now compute the local dynamics of a pair of particles after one time step.

\begin{prop} \label{hpart:ld:rwapprox}
For two particles separated by a distance $n=o(1/x)$, we have
\begin{itemize}
\item $\df{ \P\left( U(n,t) \to U(n-1,t+1), U(0,t) \to U(1,t+1)| U(0,t),U(n,t)\right)=\frac{1}{4} +\frac{x}{\pi}+O(x^2)}$
\item $\df{ \P\left(U(n,t) \to U(n+1,t+1), U(0,t) \to U(-1,t+1)|U(0,t),U(n,t)\right)=\frac{1}{4} -\frac{x}{\pi}+O(x^2)}$
\item $\df{ \P\left(U(n,t) \to U(n-1,t+1), U(0,t) \to U(-1,t+1)|U(0,t),U(n,t)\right)=\frac{1}{4}+O(x^2)}$
\item $\df{ \P\left(U(n,t) \to U(n+1,t+1), U(0,t) \to U(1,t+1)|U(0,t),U(n,t)\right)=\frac{1}{4}+O(x^2)}$.
\end{itemize}
\end{prop}

\begin{proof}
We give the proof of the lemma for 
\( \nonumber
\P\left(U(n,t) \to U(n+1,t+1), U(0,t) \to U(-1,t+1)|U(0,t),U(n,t)\right),
\)
i.e. the probability that two particles are further apart after one time step.  The other quantities can be found analogously.

The event $\{U(n,t) \to U(n+1,t+1), U(0,t) \to U(-1,t+1)\}$ and the event that there are particles at $(0,t)$ and $(-n,t)$ is the event that a dimer covers the edge $(v_4 (0,0),v_6 (0,0))$ and a dimer covers the edge $(v_4 (n,-n),v_5 (n,-n))$, i.e. the particle on the left moves to the left and the particle on the right moves to the right.  By the local statistics formula~\eqref{sec2:localstats}, the probability of this event is given by
\begin{align}
\Pf \left( \begin{array}{llll}
0 & K^{-1}(v_4,v_6) & K^{-1}(v_4,v_4(n,-n)) &K^{-1}(v_4,v_5(n,-n)) \\
\dots & 0 &  K^{-1}(v_6,v_4(n,-n)) &K^{-1}(v_6,v_5(n,-n))  \\
\dots & \dots & 0 &  K^{-1}(v_4,v_5)  \\
\dots & \dots & \dots & 0
\end{array} \right).  \label{hp:eq:tog}
\end{align}
We can compute $K^{-1}_{n,-n}(v_*,v_*(n,-n))$ using Lemma \ref{hp:uc1lemma} and the cofactor matrix given in Appendix \ref{app:cofactor}.  These are given by
\( \nonumber
	K^{-1}(v_4,v_4(n,-n))=1 - \frac{8 n x}{\pi} + (1 + 4 n^2) x^2+O(x^3),
\)
\( \nonumber
	K^{-1}(v_4,v_5(n,-n))=1 +\frac{1-8n}{\pi} x + (3/4 - n + 4 n^2) x^2+O(x^3),
\)
\( \nonumber
	 K^{-1}(v_6,v_4(n,-n))=1 +\frac{1-8n}{\pi}x +(3/4 - n + 4 n^2) x^2+O(x^3),
\)
and
\( \nonumber
	K^{-1}(v_6,v_5(n,-n))=1 + \frac{(2 - 8 n) x}{\pi} + (3/4 - 2 n + 4 n^2) x^2+O(x^3).
\)
The expressions for $K^{-1}(v_4,v_5)$ and $K^{-1}(v_4, v_6)$ can be found in Lemma \ref{app:uc:exactcomputation2}.  Computing the Pfaffian given in (\ref{hp:eq:tog}) gives
\( 
	\frac{x^2}{4} -\frac{2n x^3}{\pi} + O(x^4).   \label{hp:eq:tog2}
\)
To compute $\P\left(U(n,t) \to U(n+1,t+1), U(0,t) \to U(-1,t+1)|U(0,t),U(n,t)\right)$, it suffices to divide (\ref{hp:eq:tog2}) by the joint probability of seeing two particles separated by distance $n$ which is given in Proposition \ref{hpart:dis:jtprob}.   This  gives  the second equation in the statement of the proposition.  
\end{proof}

The above propositions provide some intuition about the model before the scaling window is taken.  Using Proposition~\ref{hpart:dis:jtprob},  we can compute the covariance between two particles at distance $n$.  This is given by $x^2(1-4/\pi^2)+O(x^3)$ because $P(U(y,t))=2/\pi x+O(x^2)$ which means that there is an attraction between a pair of particles.   Proposition~\ref{hpart:ld:rwapprox} describes the one-step transitions of a pair of particles:  given two particles at time $t$, the particles at time $t+1$ are closer together with probability $1/4+x/\pi$ ,  further apart  with probability $1/4-x/\pi$ or  to the left or to the right of their locations at time $t$ with probability $1/2$.

\subsection{Location of Particles in the Scaling Window} \label{sec:uc:locs}

This subsection concentrates on finding the measure for the locations of particles along a horizontal line in the scaling window $(x,x^2)$ for $u=u_c$.  In particular, we prove Theorem \ref{intro:mainthm} in the special case when $u=u_c$ or $\g=-1$. 

Under the scaling window $(x,x^2)$, the computations from Section~\ref{sec:uc:localdynamics} do not hold.  In order to prove Theorem~\ref{intro:mainthm}, we first find entries of $K^{-1}(v_i,v_j(n,n))$ for $i,j\in\{3,4\}$ under the scaling window $(x,x^2)$ for $1/n$ of order $x$.

\begin{lemma} \label{uc:lem:longrangeh}
Suppose that $u=u_c$ and $n x= \alpha$ where $\alpha \in [0,\infty)$.  We have
\(
	K^{-1}(v_3, v_3(n,-n))= E_1(\a)+O(x^2)
\)	
Furthermore,
\(
	K^{-1}(v_3,v_4(n,-n))=E_2(\a)+E_3(\a) x+O(x^2)
\)	
where  $E_1(\alpha), E_2(\alpha)$ are defined in Theorem \ref{intro:mainthm} and 
\(
	E_3(\a)=\frac{2}{\a \pi} - 4B_I(2,4\a)+4 S_L(-2,4a)
\)
where $B_I(n,z)$ stands for the modified Bessel function of the first kind and $S_L(n,z)$ denotes the modified Struve Function. 

\end{lemma}

\begin{proof} 
The calculation is only given for $K^{-1}(v_3, v_3(n,-n))$ as the other term is analogous and can be completed with exactly the same steps.  The main idea behind the computation is to re-arrange the underlying integral for $K^{-1}(v_3, v_3(n,-n))$ so that the only terms involving $n$ in the integrand are of the form $(1+ x\beta+O(x^2))^n$ where $\beta$ is independent of $x$ and $n$ while ensuring that the integrals are well-defined, then use the fact that $(1+x \beta+O(x^2))^n = e^{\alpha \beta} +O(x)$ for $n x = \alpha$ and finally re-arrange the expression to obtain a standard integral. 

Using the cofactor expansion given in Appendix~\ref{app:cofactor} we have
\begin{equation}
	  K^{-1}(v_3,v_3(n,-n)) = \frac{1}{(2 \pi i)^2}\int_{\mathbb{T}^2} \frac{u_c^2 x^2 w^{-1-n} z^{n-1} \left(-w^2+z^2\right)}{P(z,w)} \frac{dz}{z} \, \frac{dw}{w}
\end{equation}
A computation shows that
\begin{equation*}
      zwP(z,w)= u_c^2 x^2 (1+2w -x^2) (z-y_-(w))(z-y_+(w))
\end{equation*}
where
\( \label{hpart:uc:ypm}
y_{\pm}(w)= \frac{-1 - 3 w - w^2 - w x^2 \pm \sqrt{
 1 + 4 w + 6 w^2 + 4 w^3 + w^4 + 4 w x^2 + 8 w^2 x^2 + 4 w^3 x^2}}{1+2w -x^2}.
\)
This means that we can write
\(
	K^{-1}(v_3,v_3(n,-n)) = \frac{1}{(2 \pi i)^2}\int_{\mathbb{T}^2} \frac{w^{-1-n} z^{n-1} \left(-w^2+z^2\right)}{(z-y_-(w))(z-y_+(w))(1 + 2 w - x^2)}dz\, dw. \label{hpart:kastforuc}
\)

For equation (\ref{hpart:kastforuc}), we can take a residue with respect to $z$ at $y_-(w)$ because $|y_-(w)|<1$ for all $|w|=1$.   This results in an integrand containing the term $\sqrt{(1+w)^2+4 w x^2}$ in both the numerator and the denominator, that is, we obtain
\begin{equation} \label{hpart:kastforuc2}
      \frac{1}{2\pi i} \int_{|w|=1} -\frac{ w^{-1-n} y_-(w)^{n-1} \left(-w^2 +y_-(w)^2 \right)}{2 (w+1)(1 + 2 w - x^2) \sqrt{(1+w)^2+4 w x^2}} dw
\end{equation}
because 
\begin{equation}
 1 + 4 w + 6 w^2 + 4 w^3 + w^4 + 4 w x^2 + 8 w^2 x^2 + 4 w^3 x^2=(1 + w)^2 ((1+w)^2 + 4 w x^2)
\end{equation}
We first want to rewrite the integrand in~\eqref{hpart:kastforuc2} without a square root term involving the term of  integration.  To do so, set $x=1/2 \sqrt{-2-1/s-s}$ and take the change of variables $w=(s-y^2)/(1-y^2 s)$.  This means that $(1+w)^2 + 4 w x^2$ becomes
\begin{equation}
 \frac{(-1 + s^2)^2 y^2}{s (-1 + s y^2)^2}
\end{equation}
under the change of variables.  This transformation, $w=(s-y^2)/(1-y^2 s)$,  changes the contour of integration to a unit semi-circle  from $-1$ to $1$ passing through the point$-i$ which shall be denoted $C_1$.     Finally, we can set $-\sqrt{s}=R$ which removes all the square root terms from the integrand.  After some simplification, we obtain
\( \label{hpart:kastforuc3}
	K^{-1}(v_3,v_3(n,-n)) = \frac{1}{2\pi i }\int_{C_1} -\frac{8  R^2 \left(-1+R^2\right) (R-y) (-1+R y) }{(R+y) (1+R y) \left(-1-3 R^2+R \left(3+R^2\right) y\right)} f(y) dy
\)
where 
\(
	f(y)=\left(\frac{(R- y) (3 R + R^3 - y - 
    3 R^2 y)}{(-1 +R  y) (-1 - 3 R^2 + 3 R y + R^3 y)} \right)^n \frac{1}{ \left(-y+R \left(3+R^2-3 R y\right)\right)}.
\)
We now rewrite the integrand of~\eqref{hpart:kastforuc3} involving terms of the form $(1+\beta x+O(x^2))^n$ where $\beta$ is independent of $x$ and $n$.  To do so, we use the transformation $y=\frac{-t+R}{-1+t R}$ which moves the contour $C_1$ to a unit semi-circle  from $1$ to $-1$ passing through $i$.  Denote this new contour by $C_2$.  We also set $c_x= -(4R+4R^3)/(1+6 R^2+R^4)$. After simplification, we obtain 
 \( \label{uc:lemproof:tform}
	K^{-1}(v_3,v_3(n,-n)) = \frac{8R^2}{2 \pi i} \int_{C_2}  \frac{  \left(1+R^2-2 R t\right) (t+R (-2+R t))(c_x+t)^{n-1} }{(1+6 R^2+R^4)^2 t^{n-1}\left(c_xt+1\right)^{1+n}}dt.
\)
By undoing the previous transformations, we have that $R=-\sqrt{-1 - 2 x^2 + 2 \sqrt{x^2 + x^4}}$.  We find that $R=-i+i x+O(x^2)$ and that $c_x=-2 i x+O(x^3)$.  Setting $n=\alpha/x$ we can take a series expansion of the integrand in (\ref{uc:lemproof:tform}). This gives 
\(
	\left(\frac{2 e^{-(2 i\a)/t + 2 i\a t}}{t} \right) \left(1+x\frac{(1+t^2)(-2\a-i t+2 \a t^2)}{t^3} \right)+O(x^2).
\)
Applying the bounded convergence theorem to (\ref{uc:lemproof:tform}) along with the change of variables $t=e^{i \t}$ gives
\begin{align}
	K^{-1}(v_3,v_3(n,-n))&=\frac{1}{\pi} \int_0^{\pi} e^{-4 \a \sin \t}\left(1+ 2 i x \cos \t (-1+4\a \sin \t )\right) d\t+O(x^2) \\
	&= \frac{1}{\pi} \int_{0}^{\pi}  e^{-4 \a \sin \t} d\t +O(x^2).
\end{align}
The above integral can be written exactly in terms of Bessel and Struve functions. 

\end{proof}

In the proof of Lemma \ref{uc:lem:longrangeh}, we find an argument for the `correct' horizontal scaling.  This can be shown by the series expansion of integrand in (\ref{uc:lemproof:tform}).  Taking a larger scaling, i.e. for $\e>0$ set $n=\a/ x^{1+\e}$, leads to (\ref{uc:lemproof:tform}) converging to zero. Taking a smaller scaling, i.e. set $n=\a/x^{1-\e}$, (\ref{uc:lemproof:tform}) converges to 1 as $x \to 0$.  In essence, $n=\a/x$ provides the right horizontal scaling to ensure appropriate decay of the inverse Kasteleyn entry.   

Lemma~\ref{uc:lem:longrangeh} can also be generalized to $K^{-1}(v_i,v_j(m,n))$ for $i,j\in\{1,\dots, 6\}$ and $(m-n) x=\alpha_1$ and $(m+n) x^2 = \alpha_2$ for $\alpha_1,\alpha_2>0$. In other words, it is possible to generalize the above proof to compute  all of the entries of $K^{-1}$ in the scaling window $(x,x^2)$.

For $1\leq i \leq m$, let $x_i \in x \Z \subset \R$ denote possible particle locations on the lattice $x \mathbb{Z}$.   Let $\rho^m_x (x_1,\dots ,x_m)$ denote the $m$ point correlation function of seeing particles at $(x_1,\dots , x_m)$ with the corresponding measure on $\R$ denoted by $\P_x^{u_c}$, then

\begin{prop} \label{uc:prop:pfafflimit}
The $m$-point correlation function for the re-scaled particle locations, \\$\rho_m (y_1,\dots , y_m)$ is given by
\( \label{uc:eq:ncorx}
	\rho^x_m (x_1, \dots, x_m)= \Pf \left( M^{res}_{u_c,x} (x_i,x_j) \right)_{i,j=1}^m +O(x^{2m+2})
\)
where $M^{res}_{u_c,x}$ is antisymmetric matrix with 
\begin{align}
M^{res}_{u_c,x} (x_i,x_i)= \left( \begin{array}{cc}
	0 & \frac{2}{\pi}x-x^2 \\
	-\frac{2}{\pi}x +x^2 & 0 \end{array} \right)
\end{align}
and
\begin{align}
M^{res}_{u_c,x}(x_i, x_j) =\left( \begin{array}{cc}
	-xE_1(|x_j-x_i|) & -xE_2(|x_j-x_i|)-xE_3(|x_j-x_i|) \\
	xE_2(|x_j-x_i|)+xE_3(|x_j-x_i|) & xE_1(|x_j-x_i|) \end{array} \right).
\end{align}
where $E_1(\a),E_2(\a)$ and $E_3(\a)$ are defined in Lemma \ref{uc:lem:longrangeh}
\end{prop}


\begin{proof}[Proof of Proposition \ref{uc:prop:pfafflimit}]
This is a direct consequence of Lemma \ref{uc:lem:longrangeh} and the underlying Pfaffian structure of computing the probability of the location of $m$ particles for this particular dimer model given in Lemma~\ref{pfaffexp}.
\end{proof}

The above calculations provide a proof for Theorem \ref{intro:mainthm} for the case $u=u_c$.  The proof follows an argument from \cite{BOU:07}.


\begin{proof}[Proof of Theorem \ref{intro:mainthm} for $u=u_c$]  

As $\P_x^{u_c}$ and $\P_0^{u_c}$ are both defined in the same space  of measures, it is enough to prove that $\P_x^{u_c}$ converges weakly to $\P_0^{u_c}$.  Due to considering point processes,  tightness of the measures is guaranteed (see \cite{DVJ:89}).  It remains to show convergence of finite dimensional distributions. Let $N(\cdot)$ denote number of particles in a set, $\{ A_1,\dots , A_k\}$ denote a family of disjoint Borel sets in $\R$ and $n_1, \dots,n_k \in \mathbb{N}_0$,  we need to show
\( \label{hpart:thmproof:fidi}
	\lim_{x\to 0}\P_x^{u_c}(N(A_1)=n_1,\dots , N(A_k)=n_k)= \P_0^{u_c}(N(A_1)=n_1,\dots , N(A_k)=n_k).
\)

In order to prove such an equality, we need to write the measures $\P_x^{u_c}$ and $\P_0^{u_c}$ in terms of their $m$-point correlation functions, which can be achieved by writing their respective generating functions. Denote $n!=\prod_{j=1}^k n_j!$, $|n|=\sum_{j=1}^k n_j$,  $z=(z_1,\dots, z_k)$ and  $z^n=z_1^{n_1} \dots z_{k}^{n_k}$.   The Taylor series of the generating function of $\P_x^{u_c}$ about the point $z=0$ is
\begin{align} \nonumber
	Q_x(z)	&=	\E_x\left[ \prod_{j=1}^k (1-z_j)^{N(A_j)!} \right] \\
			&=	\sum_{p \in \mathbb{N}_0^k} \frac{(-z)^p}{p!} \E_x \left[ \prod_{j=1}^k\frac{N(A_j)!}{(N(A_j)-p_j)!} \right]. \nonumber
\end{align}
$Q_x(z)$ can also be expressed as a power series with coefficients $z^n$ given by  $\P_x^{u_c}(N(A_1)=n_1, \dots , N(A_k)=n_n)$. The two representations of the generating function give
\( \nonumber
	\P_x^{u_c}(N(A_1)=n_1, \dots , N(A_k)=n_n)=\left. \frac{(-1)^n}{n!} \frac{\partial^n}{\partial z^n} Q_x(z)\right|_{z=(1,\dots,1)}.
\)
Similarly, the generating function of $\P_x^{u_c}$ is given by
\( \nonumber
	Q_0(z) =	\sum_{p \in \mathbb{N}_0^k} \frac{(-z)^p}{p!} \E_0 \left[ \prod_{j=1}^k\frac{N(A_j)!}{(N(A_j)-p_j)!} \right].
\)
Therefore, (\ref{hpart:thmproof:fidi}) is equivalent to showing
\( \nonumber
	\lim_{x \rightarrow 0} \left. \frac{\partial^n}{\partial z^n} \frac{(-1)^n}{n!} Q_x(z)\right|_{z=(1,\dots,1)} =\left. \frac{\partial^n}{\partial z^n}\frac{(-1)^n}{n!} Q_0(z)\right|_{z=(1,\dots,1)}.
\)

Without loss of generality, suppose that the family of Borel sets, $\{A_1, \dots , A_k\}$ is open.   Let $\tilde{A}=\{y \in \mathbb{Z}: y x \in A\}$ for any set open set $A$. Using Proposition \ref{uc:prop:pfafflimit} gives
\( \nonumber
	 \E_x \left[ \prod_{j=1}^k\frac{N(A_j)!}{(N(A_j)-n_j)!} \right] = \sum_{x_1 \in x \tilde{A}_1, \dots , x_{x_{|n|}} \in x \tilde{A}_k } \Pf \left(M^{res}_{u_c,x}(x_i, x_j) \right)_{i,j=1}^{{|n|}}
\)
where the sum is over distinct edges. Taking the limit in $x \to 0$, the Riemann sums converge to 
\( \nonumber
	\int_{A_1^{n_1}} \dots  \int_{A_k^{n_k}} \Pf \left(M^{res}_{-1}(y_i, y_j) \right)_{i,j=1}^{{|n|}} dy_1^{n_1} \dots dy_{k}^{n_k}
\)
where $M^{res}_{-1}(y_i, y_j)$ is defined in Theorem \ref{intro:mainthm}. This is exactly the $m$ point correlation function for $\P_0^{u_c}$.  By noting that all entries of  $\left(M^{res}_{u_c,x}(x_i, x_j) \right)_{i,j=1}^{{|n|}}$ are uniformly bounded by 1 (using Proposition \ref{uc:prop:pfafflimit}), we can use Hardamard's inequality,  
\( \nonumber
	\left| \E_x \left[ \prod_{j=1}^k\frac{N(A_j)!}{(N(A_j)-n_j)!} \right] \right| \leq \prod_{j=1}^k |A_j| (2|n|)^{|n|/2}.
\)
This implies that $Q_x(z)$ is an absolutely convergent series.  For $z$ in a compact set,  by the Lebesgue Dominated Convergence the derivatives of $Q_x(z)$ converge uniformly to $Q_0(z)$ in the limit $x \to 0$.  This completes the proof of convergence of finite dimensional distributions. 
\end{proof}


\section{$u<1/2$ (not dependent on $x$):} \label{sec:ulessthanuc}
In this section, we prove Theorem \ref{uless:thm:PPP}.
\begin{figure}[h!]
\begin{center}
\epsfsize220pt
\scalebox{1}{\epsfbox{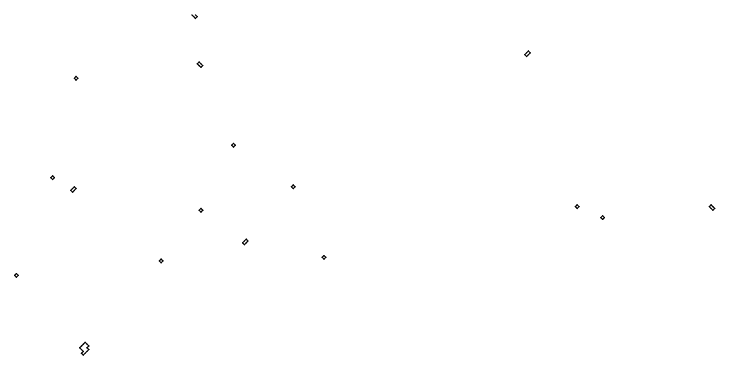}}
\caption{A simulation of the particle model on a grid of 100 by 100, with $u=0.4$ and $x=0.1$ using Glauber dynamics}
\label{ulessthanuc:fig:sim}
\end{center}
\end{figure}
Define $R(m,n)$ to be the event of observing a dimer covering the edge $(v_6(m,n) , v_1(m-1,n))$.  In order to compute any probabilities, we need the following lemma.

\begin{lemma}\label{uless:lem:H}
For $u<1/2$ (fixed), we have $\lim_{x \to 0}x^2 H_u(m,n)$ exists and is given by
\begin{equation} \label{uless:H0}
    H^0_u(m,n)=\lim_{x \to 0} x^2 H_u(m,n)=\frac{1}{(2\pi i)^2} \int_{|w|=1} \int_{|z|=1} \frac{ z^m w^n}{(u w+uz+w z)(1+u z+uw)} dz \, dw
\end{equation}

Let $r_1$ and $r_2$ be the roots of the polynomial $u + w+ uw^2$ with $|r_1|<1$ and $|r_2|>1$ and set $s_1=u/(u+r_1)$ and $s_2=u/(u+r_2)$. Then, for $n,m \geq 1$,
\(
	H_u^0(m,n)= \frac{ (-u)^{m+n}(m+n)!}{u(r_1-r_2)} \left(  
\frac{{_2 F_1 \left(1,-m,1+n,s_1^{-1}\right)}}{m! n!} 
 -s_2 \sum_{k=0}^{n-1}  
\frac{\left( -s_2 \right)^{n-1-k}  }{k! (m+n-k)!}  \right) \label{uless:lem:Hfirstterm}
\)
where ${_2F_1}$ is the confluent hypergeometric function of the second kind.

For $m,n \geq 1$, we also have
\( \nonumber
	H_u^0(-m,n)=(-1)^{n+m} \left( \frac{1-\sqrt{1-4u^2}}{2u} \right)^{n+m}.
\)
\end{lemma}

The proof of Lemma \ref{uless:lem:H}  can be found in Appendix \ref{app:c:uless}. From the above lemma, as $u<1/2$, we have that both $H_u^0(m,n)$ and $H_u^0(-m,n)$ decay exponentially.  In fact, we can compute the asymptotic expansion of $H_u^0(m,n)$: we can write $s_1=1-1/(2u^2)-\sqrt{1-4u^2}/(2u^2)$ and $s_2=1/s_1$ with $|s_1|>1$. As $1/s_1 = -u^2+O(u^4)$, we have
\(
	{_2 F_1 \left(1,-m,1+n,s_1^{-1}\right)}=u^2+O(u^3).
\)
The second term in (\ref{uless:lem:Hfirstterm}) is $u^2/(m+n)!+O(u^4)$. Setting $m=n+k$, we have 
\(
	H_u^0(n+k,n)= C \frac{(-u)^{2n+k}(2n+k)!}{(n+k)! n!} +O(u^{2n+k+1}). 
\)
where $C$ is a constant.  

From the series expansion of $H_u$, we have $H_u(m,n)=1/x^2 H_u^0(m,n) +O(1)$. The next lemma compute the joint probability of observing two rightward leaving {\bf b} edges.  

\begin{prop} \label{uless:prop:jointprob}
For $u<1/2$ (fixed) and for $m,n \in \Z$ and  $|m|\geq |n|$ with $k=||m|-|n||$.
\(
	\P(R(0,0) , R(m,n))= u^4x^4\left(1-\frac{1}{\sqrt{1-4u^2}} \right)^2+ C x^2(-u)^{4n+2k}+O(u^{4n+2k+1})
\)
where $C$ is a constant.
\end{prop}

\begin{proof}
We only compute the result for $m\geq n>1$ as the other cases follow from very similar calculations.  Using the local statistics formula  (\ref{sec2:localstats}), $\P(X(0,0) , X(m,n)) $ is given by
\begin{align} \label{uless:jtprobs}
u^2 x^2 \Pf  \left( \begin{array}{llll}
		0& K^{-1}(v_6,v_1{(-1,0)}) & K^{-1}(v_6,v_6{(m,n)}) & K^{-1}(v_6,v_1{(m-1,n)})\\
		\dots  & 0 & K^{-1}(v_1,v_6{(m+1,n)}) &K^{-1}(v_1,v_1{(m,n)}) \\
		\dots  &  \dots  & 0& K^{-1}(v_6,v_1{(-1,0)}) \\
		\dots & \dots  & \dots & 0 \end{array} \right) 
\end{align}
Two of the interaction terms can be computed  using Lemma \ref{uless:lem:H} and the cofactor matrix given in Appendix \ref{app:cofactor}.  We have
\begin{equation}
\label{uless:eq:int}	
  \begin{split}
K^{-1}(v_6,v_6{(m,n)}) &=-K^{-1}(v_1,v_1{(m,n)}) \\
			&=ux^2 H_u(m-1,n)-ux^2 H_u(m+1,n) \\
			&= uH_u^0(m-1,n)-u H_u^0(m+1,n)+O(x^2).
\end{split}
\end{equation}
It still remains to find $K^{-1}(v_1,v_6{(m+1,n)})$ and $ K^{-1}(v_6,v_1{(m-1,n)})$.  We have
\begin{align}
	K^{-1}(v_6,v_1(m-1,n))&=\frac{1}{x(2\pi i)^2} \int_{|z|=1} \int_{|w|=1} \frac{ z^{m-1} w^{n}}{u w+uz+w z} dw \, dz +O(1)\\
	&=\frac{(m+n-1)!}{x(m-1)!( n)!} (-u)^{m+n-1}+O(1).
\end{align}
The remaining interaction entry $ K^{-1}(v_1,v_6{(m+1,n)})$ is given by
\begin{equation}\nonumber
\begin{split}
  K^{-1}(v_1,v_6{(m+1,n)})&=\frac{1}{(2\pi i)^2 x} \int_{\mathbb{T}^2} \frac{z^{m+1} w^{n}}{(1+uw+uz) w z} dw \,dz\\
   &+\frac{x}{(2\pi i)^2} \int_{\mathbb{T}^2} \frac{(u^2 + u^3 w^3 + u z + 2 u^2 w z + u^3 w^2 z)z^{m+1} w^n}{ (1 + u w + u z)^2 (u w + u z + w z)} \frac{dw}{w} \frac{dz}{z} +O(x^3).
\end{split}
\end{equation}
For the right hand side of the above equation, we have the first term is zero for $u<1/2$ while the second term is a linear combination of
\( \label{uless:eqn:intIterm}
	I_u(m,n)= \frac{x}{(2\pi i)^2}  \int_{|z|=1} \int_{|w|=1} \frac{ z^n w^m}{(u w+uz+w z)(1+uz+u w)^2} dw \,dz.
\)
$I_u(m,n)$ can be computed by differentiating $xH_u^0(m,n)$ by $r_1$ and $r_2$.  This can be seen by comparing the single integral formulas of $I_u(m,n)$ and $x H_u^0(m,n)$ after taking the residue at $z=-uw/(u+w)$ in both double integral formulas where the double integral formula of $xH_u^0$ can be found using~\eqref{uless:H0}.  We can therefore find, by using Lemma~\ref{uless:lem:H}, that $I_u(m,n)$ is of order $xu^{2n+k+1}$.  In the expansion of the Pfaffian given in~\eqref{uless:jtprobs}, we have that $u^2 x^2 K^{-1}(v_1,v_6{(m+1,n)})$ is multiplied to $ K^{-1}(v_6,v_1{(m-1,n)})$ and hence this term contributes a maximum of $O(u^{4n+2k+1})$.

To compute $K^{-1}(v_6,v_1{(-1,0)})$, we can take a series expansion of the integrand in terms of $x$ and use residue calculus to find the integral following the method given in Appendix~\ref{app:uless:examples}.  We obtain
\begin{equation}
\label{uless:eq:nonint}
\begin{split}
K^{-1}(v_6,v_1{(-1,0)}) &= \frac{1}{x(2\pi i)^2} \int_{|z|=1} \int_{|w|=1} \frac{1}{z(u w+uz+w z)} dw \, dz \\
&-\frac{x}{(2\pi i)^2} \int_{|z|=1} \int_{|w|=1} \frac{u^3 w + 2 u^2 w^2 + u w^3 + u^3 z + u^2 w^3 z}{w (1 + u w + u z) (u w + u z + w z)^2} dw\, dz+O(x^3)\\
&=u x\left(1-\frac{1}{\sqrt{1-4u^2}} \right) +O(x^2).
\end{split}
\end{equation}
We can use all the expansions computed  above and substitute them back into the matrix given~\eqref{uless:jtprobs} and take its Pfaffian which gives the result.

\end{proof}

Proposition \ref{uless:prop:jointprob} can be generalized to show  that the difference between the joint probability of $k$ edges and the product of the probabilities of each edge is an exponentially decaying term.  This can be achieved using a complete graph representation of the Pfaffian, that is, the Pfaffian of a matrix $A$ represents counting the weighted number of perfect matchings of a complete graph with the edge weights given by $|A_{i,j}|$ (with the signs given by the  \emph{Pfaffian orientation}), see \cite{god:93} for more details.  By the complete graph representation, the joint probability of $k$ edges is equal to the product of the probabilities of each edge plus an exponentially decaying term.    However, we are interested in generalizing Proposition \ref{uless:prop:jointprob}  to consider the joint distributions of creations as opposed to {\bf b} edges.  Let $Y(a,b)$ denote a creation at $(a,b)$ with $a,b \in \Z$ and $(a+b) \mod 2=0$.

\begin{lemma} \label{uless:creat:decay}
Suppose that for each $k$ and $l$, there exists a constant $c_{k,l}$ such that $|(i_k,j_k)-(i_l,j_l)|> c_{k,l}/\sqrt{x}$.  We have
\( \nonumber
\P (Y(i_1,j_1),\dots, Y (i_m,j_m))=\prod_{k=1}^m \P(Y (i_k,j_k)) +O(x^{2m+1})
\)
\end{lemma}

The result also holds when the distance between particles is $1/x^\e$ for $\e>0$.

\begin{proof}
The proof is very similar to the proof of Proposition \ref{uless:prop:jointprob}.  There is a creation if there are dimers covering the edges $(v_6, v_1(-1,0))$ and $(v_5,v_2(0,-1))$.  By~\eqref{sec2:localstats}, the underlying matrix for $\P (Y(i_1,j_1),\dots, Y (i_m,j_m))$ is a $4m \times 4m$ matrix with diagonal $4 \times 4$ blocks given by
\begin{equation} \nonumber 
\left(  \nonumber
\begin{array}{llll} 
  0 & K^{-1}(v_6,v_1(-1,0)) & K^{-1} (v_6,v_5) & K^{-1}(v_6,v_2(0,-1)) \\
  -K^{-1}(v_6,v_1(-1,0)) & 0 &  K^{-1} (v_1,v_5(1,0)) & K^{-1}(v_1,v_2(1,-1)) \\
   -K^{-1}(v_6,v_5) & -K^{-1} (v_1,v_5(1,0)) &0 &  K^{-1}(v_5,v_2(0,-1)) \\
   -K^{-1}(v_6,v_2(0,-1)) & -K^{-1} (v_1,v_2(1,-1)) & -K^{-1}(v_5,v_2(0,-1))&0 \\ \end{array} \right)
\end{equation}
and off-diagonal $4 \times 4$ blocks of the form
\begin{align}
\left(  \nonumber
\begin{array}{llll}
K^{-1}(v_6,v_6(m,n)) & K^{-1}(v_6,v_1(m-1,n)) & K^{-1}(v_6,v_5(m,n))  & K^{-1}(v_6,v_2(m,n-1)) \\
 K^{-1}(v_1,v_6(m+1,n)) & \dots &\dots & \dots \\
K^{-1}(v_5,v_6(m,n)) & \dots & \dots &\dots \\
K^{-1}(v_2,v_6(m,n+1))  & \dots & \dots & \dots \\
\end{array}
\right)
\end{align}
where `$\dots$' denotes the appropriate entry.  By the cofactor matrix given in Appendix \ref{app:cofactor} and a series expansion around $x=0$, each entry of the above  off-diagonal block matrix is some linear combination of $H_u^0(m,n)$ or $I_u(m,n)$ which is defined in (\ref{uless:eqn:intIterm}).  By using Lemma~\ref{uless:lem:H} and the argument for approximating $I_u(m,n)$ given in Proposition~\ref{uless:prop:jointprob},  each entry of the off-diagonal block matrix decays exponentially with distance.  We can use the complete graph representation of the Pfaffian and ignore the matchings of the complete graph which contain no exponentially decaying edge weights as they are absorbed into the error term.  Therefore, the largest contribution of $\P (Y(i_1,j_1),\dots, Y (i_m,j_m))$ comes from the subset of the set of matchings which do not contain any exponentially decaying edge weights.  This is exactly $\prod_{k=1}^m \P(Y (i_k,j_k))$.
\end{proof}

Let $A_n$ be the event of seeing a string of {\bf b} edges crossing an $n$ by $n$ square.  This corresponds to seeing a particle traverse an $n$ by $n$ square in the particle model. We have

\begin{prop}\label{uless:prop:decay}
For $n=x^{\e-1}$ and for all $0<\e<1$, we have
\( \nonumber
	\lim_{x \to 0} \P(A_n) =0.
\)
\end{prop}

The above proposition means that the probability of seeing non-trivial paths in the scaling window $(x,x)$ is zero.  In other words, the paths of the particles from each creation contract to points in the scaling window $(x,x)$.  

\begin{proof}
The event $A_n$ is contained in the event of having a {\bf b} edge at both the top and bottom of the $n$ by $n$ square at any position. Therefore
\( \nonumber
	\P(A_n) \leq n^2 \sup_{k} \P(X(0,0) \cap X(n-k,k)).
\)
By a symmetrical consideration, the maximum probability occurs when the $L^1$ distance between the two {\bf b} edges is minimized, i.e. in Proposition  \ref{uless:prop:jointprob} the exponential term is smallest when $k=n/2$.   Using Proposition \ref{uless:prop:jointprob} gives
\( \nonumber
	\P(A_n) \leq n^2 \left(u^4x^4\left(1-\frac{1}{\sqrt{1-4u^2}} \right)^2+ C x^2(-u)^{2n} \right).
\)
By taking $n=x^{\e-1}$ with $0<\e<1$,   $\P(A_n)$ tends to zero as $x$ tends to zero. 
\end{proof}

Finally, we need to check that the creation points are singletons in the scaling window. Let $B_n((a,b))$ represent a ball of radius $n \in \Z_+$ around $(a,b)$ where $a, b \in \Z$ with $(a+b) \mod 2=0$.   Let $Y_1$ and $Y_2$ be the locations of two distinct creations. 

\begin{lemma} \label{uless:lem:singletons}
For $n=x^{\e-1}$ and for all $0<\e <1$, we have
\( \nonumber
	\lim_{x \to 0} \P ( Y_2 \in B_n ((a,b))|Y_1=(a,b)) =0.
\)

\end{lemma}

This lemma implies that given the creation $Y_1$, the probability that there is another creation which converges to $Y_1$ in the scaling window $(x,x)$ is zero. In other words, given a creation, the probability that it is a doubleton is zero in the scaling window $(x,x)$.   As the expected number of creations is locally finite in the scaling window $(x,x)$, the probability that any creation is a doubleton in the scaling window is zero. 

\begin{proof}
We can write 
\( \nonumber
	 \P ( Y_2 \in B_n ((a,b))|Y_1=(a,b))= \sum_{(c,d) \in B_n(a,b) } \P(Y_2=(c,d)|Y_1=(a,b)).
\)
$ \P(Y_2=(c,d),Y_1=(a,b))$ can be computed explicitly and it is given by a Pfaffian of an 8 by 8 matrix whose 4 by 4 diagonal and off-diagonal blocks are the same as those given in Lemma~\ref{uless:creat:decay}.   In fact, we have for a constant $C>0$,
\( \nonumber
	 \P ( Y_2 =(c,d)|Y_1=(a,b)) \leq C x^2 +O(x^3)
\)
because the maximum of $ \P(Y_2=(c,d),Y_1=(a,b))$ is order $x^4$.   Therefore, we can bound the probability of $Y_2 \in B_n((a,b))$ given $Y_1=(a,b)$ which gives 
\( \nonumber
	 \P ( Y_2 \in B_n ((a,b))|Y_1=(a,b)) \leq C x^2 n^2.
\)
where the factor of $n^2$ comes from multiplying by the area of $B_n((a,b))$.   As we have taken $n=x^{\e-1}$, we get the result.

\end{proof}

Using the above lemmas and propositions, we can now prove Theorem \ref{uless:thm:PPP}.  

\begin{proof}[Proof of Theorem \ref{uless:thm:PPP}]
By Proposition \ref{uless:prop:decay}, each loop generated by a string of {\bf b} edges is finite and does not appear in the scaling window $(x,x)$.  Therefore, each loop of {\bf b} edges will be contracted to a single point under re-scaling.  As the density of creations is finite in the scaling window, by Lemma \ref{uless:lem:singletons} all creations are singletons.  As tightness is guaranteed for point processes (see \cite{DVJ:89}),  it is enough to consider the locations of creations and show that the finite dimensional distributions of the corresponding measures converge in the scaling window $(x,x)$.  We shall follow the approach developed in the proof of Theorem \ref{intro:mainthm} for the case $u=u_c$ (see Section \ref{sec:uc:locs}).

Let $\P_x$ denote the point process in $\R^2$ arising from the locations of creations in a loop.  Let $A_i$, for $i \in \{1,\dots, k\}$, be disjoint, non-empty simply connected domains and let  $\tilde{A_i}$ be the union of lattice points in  $((2 \Z \times 2 \Z) \cup ((2\Z +1)\times ((2\Z +1))$ so that $x \tilde{A}_i \to A_i$ as $x \to 0$.   Let $N(A_i)$ denote the number of creations in a set $A_i$ for $1 \leq i \leq k$.   Denote $n!=\prod_{j=1}^k n_j!$, $|n|=\sum_{j=1}^k n_j$,  $z=(z_1,\dots, z_k)$ and  $z^n=z_1^{n_1} \dots z_{k}^{n_k}$.  For $n \in \mathbb{N}^k_0$, we have
\( \nonumber
	\E_x \left[ \prod_{j=1}^k\frac{N(A_j)!}{(N(A_j)-n_j)!} \right]= \sum_{(i_1,t_1)\in \tilde{A}_1}  \dots \sum_{ (i_{|n|}, t_{|n|})  \in \tilde{A}_{k}}  \P(Y(i_1,t_1),\dots Y(i_{|n|},t_{|n|}))
\)
where the sum is over distinct lattice sites.  We can split the sum into two with the first sum having the creations separated by at least $1/x^{1/2}$  and the second sum containing the remaining terms.  In other words, for $a_k=(i_k,t_k)$ we have for
\( \nonumber
	A=\{(a_1,\dots, a_{|n|}):  |a_i-a_j|>x^{-1/2} \mbox{ for all  } 1\leq i,j \leq |n| \mbox{ with } a_1 \in \tilde{A}_1,\dots, a_{|n|} \in \tilde{A}_k \}
\)
and for $A^c=(\tilde{A}_1 \times \dots \times \tilde{A}_k )\backslash A$, we can write
\( \nonumber
	\E_x \left[ \prod_{j=1}^k\frac{N(A_j)!}{(N(A_j)-n_j)!} \right]=\left(\sum_A + \sum_{A^c} \right) \P(Y(a_1),\dots, Y(a_{|n|}))
\)
The summand is of order $x^{-2|n|}$.  For the second sum of the above equation, we are summing over less than $x^{-2|n|+1}$ lattice points which means that in the limit as $x$ tends to zero, the second sum goes to zero.   For the first sum, as each edge is separated by a distance of at least $1/\sqrt{x}$, using Lemma \ref{uless:creat:decay} and taking the limit as $x \to 0$  gives 
\( \nonumber
	 \int_{(A_1)^{n_1}} \dots \int_{(A_k)^{n_k} } \left( \frac{1}{2}-u^2 -\frac{1}{2} \sqrt{1-4u^2} \right)^{|n|} dx_1^{n_1} \dots dx_k^{n_k}.
\)
Let
\( \nonumber
	Q_x(z)	=	\sum_{p \in \mathbb{N}_0^k} \frac{(-z)^p}{p!} \E_x \left[ \prod_{j=1}^k\frac{N(A_j)!}{(N(A_j)-p_j)!} \right]
\)
and set
\( \nonumber
	Q_0(z)	=	\sum_{p \in \mathbb{N}_0^k} \frac{(-z)^p}{p!} \int_{(A_1)^{n_1}} \dots \int_{(A_k)^{n_k} } \left( \frac{1}{2}-u^2 -\frac{1}{2} \sqrt{1-4u^2} \right)^{|n|} dx_1^{n_1} \dots dx_k^{n_k}
\) 
i.e. a generating function for a Poisson point process.   Notice that each generating function can also be written as a power series whose coefficients are given by the finite dimensional distribution probabilities for the corresponding measure on the sets $A_1, \dots, A_m$.    Following the same convergence argument given in the proof of Theorem \ref{intro:mainthm} for $u=u_c$, it is clear that the derivatives of $Q_x(z)$ (with respect to $z$) converge to $Q_0(z)$.  This proves the convergence of finite dimensional distributions.
\end{proof}




\section{The particle model for $u=u_i$} \label{sec:uindpt}

In this section, we concentrate on the particle model for $u=u_i$ and prove Theorem \ref{uindpt:thm:main} by showing that the particles in the particle model have the same distribution as the boundaries of colors for the two colored noisy voter model, i.e. at $u=u_i$ the particle model is a Markov process.  Figure \ref{uc:fig:sim} shows a simulation of the particle model for $x=0.1$ and $u=u_i$.

  \begin{figure}
\begin{center}
\epsfsize220pt
$$\scalebox{1}{\epsfbox{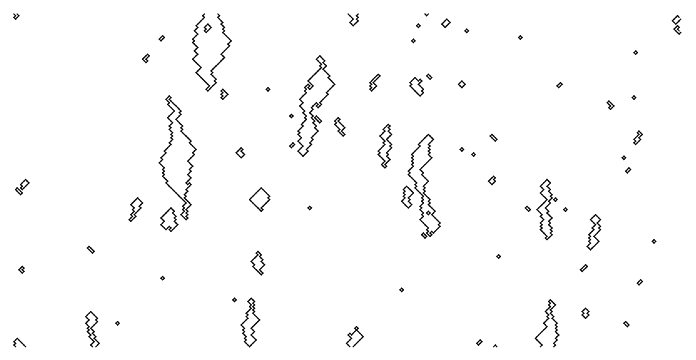}}$$
\caption[A simulation of the Particle Model: $u=u_i$ and $x=0.1$]{A simulation of the particle model on a 100 by 100 grid with $u=u_i$ and $x=0.1$ made using Glauber dynamics.}
\label{ui:fig:sim}
\end{center}
\end{figure}

We can compute $H_{u_i}(m,n)$, defined in~\eqref{definitionofH}, and it is given by 
\begin{lemma}\label{cor:lemmacritical2}
For $m, n \geq 0$ and $0<x<1$, we have
\(
	H_{u_i}(m,-n)= \frac{(-1)^{m+n} (1+\sqrt{1-x^2})^2}{4x^3} \left( \frac{\sqrt{1-x^2}}{1+x} \right)^{m+n}.	\nonumber
\)
\end{lemma}
This is an exact result and the proof can be found in Appendix \ref{app:ui:local}.  As a consequence, we can compute the entries of $K^{-1} (v_i, v_j(n,-n))$ exactly.  

\begin{lemma}
	For $u=u_i$ and for $0<x<1$, 
\( \nonumber
	K^{-1} (v_3, v_3(n,-n))=-K^{-1} (v_3, v_4(n,-n))
\)	
\end{lemma}
\begin{proof}
	We can compute $K^{-1}(v_3,v_3(n,-n))$ using Lemma \ref{cor:lemmacritical2} and the cofactor matrix found in Appendix \ref{app:cofactor} which gives
\begin{align}
		K^{-1} (v_3, v_3(n,-n))&= -u^2 x^2 H_{u_i} (n-1,1-n)+ u^2 x^2 H_{u_i}(n+1,-n-1)  \nonumber \\
		&= -\frac{1}{1-x^2} \left( \frac{\sqrt{1-x^2}}{1+x} \right)^{2n}. \nonumber
\end{align}
Similarly, we can compute $K^{-1}(v_3,v_3(n,-n))$ using Lemma \ref{cor:lemmacritical2} and the cofactor matrix found in Appendix \ref{app:cofactor} which gives
\begin{align}
	K^{-1} (v_3, v_4(n,-n))&=\frac{1}{1-x^2} \left( \frac{\sqrt{1-x^2}}{1+x} \right)^{2n}. \nonumber
\end{align}

\end{proof}

As in Section~\ref{sec:hpart}, define $U(i,t)$ to be the event that there is a particle at $-i$ at time $t$. 
\begin{prop} \label{uc2}
For $n\geq1$, we have
\( \nonumber
	\P (U(0,t),U(n,t))=\frac{x^2}{(1+x)^2}.
\)
\end{prop}

\begin{proof}
We can calculate the joint probability via the inclusion-exclusion formula given in Lemma \ref{pfaffexp}.  Hence, we require
\begin{align}
 \Pf\left( \begin{array} {cccc}
0 & \P(U(0,t)) & -x K^{-1} (v_3, v_3(n,-n))& -xK^{-1} (v_3, v_4(n,-n)) \\
\dots & 0  &-xK^{-1} (v_4, v_3(n,-n))& -xK^{-1} (v_4, v_4(n,-n)) \\
\dots &  \dots & 0 & \P(U(n,t))\\
\dots & \dots& \dots &0\end{array}\right)\label{hpart:critical}
 \end{align}
 where $v_i$ for $i \in \{1, \dots, 6\}$ represents a vertex $i$ in the fundamental domain and $v_i(m,n)$ represents vertex $i$ a distance $(m,n)$ from the fundamental domain.
 
From Lemma~\ref{expect:prop:Na}, $\P(U(0,t))=\P(U(n,t))=x/(1+x)$.  It remains to calculate the interaction between the two edges, which can be found by invoking Lemma \ref{cor:lemmacritical2}. Due to $K^{-1} (v_3, v_3(n,-n))=-K^{-1} (v_3, v_4(n,-n))$ for $u=\critical$, by symmetry  (\ref{hpart:critical}) becomes  
\( \nonumber
 \Pf\left( \begin{array} {cccc}
0 & \P(U(0,t)) & 0 & 0 \\
\dots & 0  & -xK^{-1}(v_4,v_3(n,-n)) & xK^{-1}(v_4,v_3(n,-n))\\
\dots & \dots & 0 & \P(U(n,t))\\
\dots & \dots & \dots &0\end{array}\right).
\)
The Pfaffian of the above matrix represents counting the number of weighted perfect matchings of a complete graph with four vertices $w_1, \dots,w_4$, with the edges weighted given by the absolute value of the entry of the above matrix, that is, the absolute value of the entry $(i,j)$ is the weight of the edge $(w_i,w_j)$.  We can only have one possible contribution (the edges $w_1w_2$ and $w_3 w_4$) which gives the desired result.
\end{proof} 

The above method can be generalized to consider the locations of an arbitrary number of particles.  

\begin{prop}
For $i_1<\dots <i_n$, 
\( \nonumber
	\P (U(i_1,t),\dots,U(i_n,t))= \frac{x^{n}}{(1+x)^n}.
\)
\end{prop}

\begin{proof}

Using the inclusion-exclusion formula given in Lemma \ref{pfaffexp},  $\P (U(i_1,t),\dots,U(i_n,t))$  can be written as a Pfaffian of a matrix with block diagonals given by
\( \nonumber
	\left( \begin{array}{ll}
		0 &\P (U(i_j,t)) \\
		-\P (U(i_j,t))  & 0 \end{array} \right)
\)
for $1 \leq j \leq n$ and off-diagonal blocks given by
\(\nonumber
	\left( \begin{array}{ll}
		-x K^{-1} (v_3,v_3(i_k,-i_k)) &-x K^{-1} (v_3,v_4(i_k,-i_k)) \\
		-x K^{-1} (v_4,v_3(i_k,-i_k)) &-x K^{-1} (v_4,v_4(i_k,-i_k))\end{array} \right)
\)
for $k \in \{1, \dots, n\}$.  Using  $-K^{-1} (v_3,v_3(i_k,-i_k))=K^{-1} (v_4,v_3(i_k,-i_k))$ and applying the complete graph representation of the Pfaffian inductively described in Proposition \ref{uc2}, we can write
\( \nonumber
	\P (U(i_1,t),\dots,U(i_n,t))= \prod_{j=1}^n \P(U(i_j,t))=\frac{x^n}{(1+x)^n}.
\) 
\end{proof}

The locations of the particles are independent and Bernoulli($\frac{x}{1+x}$).  It now remains to look at the dynamics.

\begin{prop} \label{indpt:prop:bedges}

For an arbitrary horizontal line, dimers covering {\bf b} edges are iid with distribution Bernoulli($1/2 - \frac{\sqrt{1 - x}}{2 \sqrt{1 + x}}$).
\end{prop}
\begin{proof}
We consider an arbitrary horizontal line.  Suppose that $L_i$ represents the event that a dimer covers a {\bf b} edge going left (i.e. a dimer covers the edge $(v_5,v_2(0,-1))$) from the fundamental domain corresponding to $U(i,t)$ for $t$ fixed. Let $R_i$  denote the event that a dimer covers a {\bf b} edge going right (i.e. a dimer covers the edge $(v_6,v_1(-1,0))$)  from the fundamental domain corresponding to $U(i,t)$ for $t$ fixed.   We first consider the probability of two left going {\bf b} edges a distance $n$ away.  Without loss of generality, we can consider one of these edges to be located at the origin. This is given by
\begin{align}
&\P( L_0, L_n)=u^2 x^2 \nonumber \\
&\Pf \left( \begin{array}{cccc}
0& K^{-1} (v_5,v_2(0,-1)) & K^{-1}(v_5,v_5(n,-n)) & K^{-1}(v_5, v_2(n,-n-1)) \\
\dots & 0 & K^{-1}(v_2,v_5(n,-n+1)) & K^{-1}(v_2,v_2(n,-n)) \\
\dots & \dots & 0& K^{-1}(v_5,v_2(0,-1))\\ 
\dots & \dots & \dots & 0 \end{array}
\right) \nonumber
\end{align}
where we have used `$\dots$' to denote the lower triangle entries (the matrix is anti-symmetric).  Then, explicit calculation using Lemma \ref{cor:lemmacritical2} gives
\(	\nonumber
	K^{-1}(v_5,v_5(n,-n)) = -K^{-1}(v_2, v_5(n,-n+1))
\)
and 
\( \nonumber
	K^{-1}(v_2,v_2(n,-n)) = -K^{-1}(v_5, v_2(n,-n-1)).
\)
Therefore, using row and column operations on the above matrix gives
\begin{eqnarray*}
\P (L_0, L_n)& =& u^2 x^2 (K^{-1}(v_5,v_2(0,-1)))^2\\
&=& \frac{1-\sqrt{1-x^2}}{2(1+x)}.
\end{eqnarray*}
which can be seen by Lemma~\ref{expect:thm:expectb} because $\E[N_b]/2 = u x K^{-1}(v_5,v_2(0,-1))$ by~\eqref{sec2:localstats}. 

A similar result for $\P(R_0,R_n)$ is true, but is reliant on 
\( \nonumber
	K^{-1}(v_6,v_6(n,-n)) 	= 	K^{-1}(v_1, v_6(n+1,-n)).
\) 
and
\( \nonumber
	K^{-1}(v_1,v_1(n,-n)) 	= 	K^{-1}(v_6, v_1(n-1,-n)).
\)

The above calculation can be generalized to find the joint probabilities of observing dimers covering any collection of rightward and leftward {\bf b} edges located on the same horizontal line.  Again, this joint probability is the product of the probabilities of observing a dimer covering each edge. Furthermore, we can show that the probability of observing $n$ creations along a horizontal line, is given by the product of the probabilities of a dimer covering each edge.  For example, 

\( \nonumber
	\P(L_0,R_0, L_n,R_n)=u^4 x^4 (K^{-1}(v_5,v_2(0,-1)) )^2(K^{-1} (v_6,v_1(-1,0)))^2.
\)
\end{proof}

A further calculation shows that the trajectories of particles are independent.  This relies on computing the possible trajectories of the particles given their locations.  This calculation follows in the same vein as the previous propositions.

\begin{lemma}
Given an {\bf a} edge (i.e no particle), the probability that there is a creation directly above (or below) is $(1-\sqrt{1-x^2})/2$.  
\end{lemma}

\begin{proof}
The probability of a creation is equal to $u^2 x^2 (K^{-1}(v_5,v_2(0,-1)) )(K^{-1} (v_6,v_1(-1,0)))$ which is equal to 
$(1-\sqrt{1-x^2})/(2+2x)$.  The above event contains the event that there is an {\bf a} edge covered by a dimer on the previous level.  To find the probability of a creation given there is no particle on the previous level, it suffices to divide the probability of a creation by the probability that there is an {\bf a} edge covered by a dimer.  This gives $(1-\sqrt{1-x^2})/2$.

\end{proof}

\begin{proof}[Proof of Theorem \ref{uindpt:thm:main}]

For $u=\critical$, the dimer covering of the Fisher lattice can be thought as follows:
\begin{itemize}
\item Take an arbitrary horizontal line along the {\bf a} edges.  With product measure, we can place dimers on {\bf a} edges with an iid  Bernoulli ($\frac{1}{1+x})$ distribution. 
\item If there is no dimer covering an {\bf a} edge on the row below (or above), then choose a dimer to cover either the left or right {\bf b} edge directly above (or below) with probability $\half$.  This is a direct consequence of seeing one {\bf b} edge conditioned on there not being a {\bf b} edge on the row below (above).  The surrounding dimers are automatically selected from this probabilistic choice.  
\item If there is a dimer covering an {\bf a} edge on the row below (above), then with probability $(1 - \sqrt{1 - x^2})/2$, cover both {\bf b} edges directly above (below)  {\bf b} edges. 
\end{itemize}

Due to the independence of the entries,  the model can be viewed as a one dimensional noisy voter model with two colors, say red and blue.  This description of this noisy voter model differs by the one above in the following sense:
for time $t$, the state space is given by $\{R,B\}^{2\Z}$ for $t\in 2\Z$ and $\{R,B\}^{2\Z+1}$ for $t \in 2 \Z+1$. If the two neighboring sites on the previous level have different colors, the site chooses its color randomly between the two colors.    If the two neighboring sites on the previous level sites have the same color, the site flips its color with probability $(1 - \sqrt{1 - x^2})/2$ to the other color.  

The particles in the dimer model represent the boundary between the two colors.  A site contained within a connected region of the same color can flip color with Bernoulli$(1 - \sqrt{1 - x^2})/2$.  This represents the noise. Hence, the particle model is in direct correspondence with the noisy voter model. The scaling window $(x,x^2)$ of the noisy voter model was constructed in \cite{Fon:06} and named the Continuum Noisy Voter model.

\end{proof}

\section{The scaling around $u_c$ and $u_i$}\label{sec:inbet}

This section focuses on the behavior exhibited when $u=u_\g=(1- \sqrt{1- \g x^2})/( \g x^2)$ with $\g  \in \R \backslash \{-1,0,1\}$, $u_{0}=1/2$  and  $|\g|=o(1/x)$(i.e. $u=1/2-\g x^2/8$).  We choose the same scaling window used for $u=u_c$ and $u=u_i$ and find the stationary measure for the locations of particles on a horizontal line in the scaling window $(x,x^2)$.

\begin{lemma} \label{inbet:lem:res} 

Suppose that $n x= \alpha$ where $\alpha \in [0,\infty)$,  $\g \in \R \backslash \{1\}$ and $u=u_\g$. We have
\( \nonumber
	K^{-1}(v_3,v_3(n,-n))=  E_1^\g (\alpha) +O(x)
\) 
and 
\( \nonumber
	K^{-1}(v_3,v_4(n,-n))= -E_2^\g (\alpha)+O(x)
\) 
where
\( \nonumber
	 E_1^\g (\alpha)= \frac{1}{\pi  } \int_{e_2(\g)}^{e_1(\g)} \frac{2 p e^{-\alpha p} }{\sqrt{-4-8\g-4 \g ^2+12p^2-4\g p^2-p^4}} dp
\)
and 
\( \nonumber
	E_2^\g (\alpha)=  \frac{1}{\pi  } \int_{e_2(\g)}^{e_1(\g)} \frac{(-2-2\g-p^2) e^{-\alpha p} }{2\sqrt{-4-8\g-4 \g ^2+12p^2-4\g p^2-p^4}} dp
\)
where $e_1(\g)=2+\sqrt{2(1-\g)}$ and $e_2(\g)=(2-\sqrt{2(1-\g)})\mathbb{I}_{\g>-1}+(-2+\sqrt{2(1-\g)})\mathbb{I}_{\g<-1}$.

\end{lemma}

We have that $e_1,e_2$ are real for $\g<1$ and are complex for $\g>1$.  Note that $E_1$ and $E_2$ are real for all values of $\g \not =1$. 

Lemma \ref{inbet:lem:res} provides the inverse Kasteleyn entries which can be used to compute correlations.   For $\g>1$, we have that $\overline{e_1(\g)}=e_2(\g)$. For $\g<1$,  it can be shown that $|E_1^{\g}(\alpha)|>|E_2^{\g}(\alpha)|$.  This implies that any pair of particles is positively correlated.  However, for $\g>1$, it can be shown that $|E_1^{\g}(\alpha)|<|E_2^{\g}(\alpha)|$, which shows that any pair of particles is negatively correlated.

The entry for $K^{-1}(v_3,v_4)$ can also be computed up to $O(x)$ in a similar fashion as the other entries in Lemma \ref{inbet:lem:res} . Let $e(\gamma)=\lim_{x \to 0}1/x (1- x |K^{-1}(v_3,v_4)|)$ which is the density of seeing a particle because
\begin{equation}
 x|K^{-1}(v_3,v_4)|= x \Pf \left( \begin{array}{cc}
                            0 & K^{-1}(v_3,v_4) \\
                            -K^{-1} (v_3,v_4) & 0
                           \end{array} \right)
\end{equation}
 is the probability of seeing an {\bf a} edge.  We have
\( \label{inbet:probofedge}
	e(\gamma)=\lim_{x\to 0} \frac{1}{x}\left(1- x|K^{-1}(v_3, v_4)| \right)=\frac{1}{\pi i} \int_{e_2(\g)}^{e_1(\g)}  \frac{2-2 \g+p^2 }{2\sqrt{4+8\g+4 \g ^2-12p^2+4\g p^2+p^4}} dp.
\)
Under the scaling window $(x,x^2)$ and the limit $x \to 0$, the particle locations are given by Theorem \ref{intro:mainthm}.

\begin{proof}[Proof of Theorem \ref{intro:mainthm} for $u \not = u_c$]
 The structure of the matrix comes from Lemma \ref{pfaffexp}.  The relevant entries come from Lemma \ref{inbet:lem:res}.  The proof of weak convergence is the same as the proof in Theorem \ref{intro:mainthm} for $u=u_c$.
\end{proof}

It remains to compute the inverse Kasteleyn entries. 

\begin{proof}[Proof of Lemma \ref{inbet:lem:res}]
We only compute $K^{-1}(v_3,v_3(n,-n))$ as $K^{-1}(v_3,v_4(n,-n))$ can be computed following the same computational steps.  The main idea of the proof is similar to the main idea of the proof of Lemma~\ref{uc:lem:longrangeh}.  That is, we rearrange and expand the integral until we have a term of the form $(1+\beta x +O(x^2))^n$ where $\beta$ is independent of $x$ and $n$ while ensuring that the integral is well-defined.

By the cofactor expansion given in Appendix~\ref{app:cofactor}, we have
\(  \nonumber
	K^{-1}(v_3,v_3(n,-n))=\frac{1}{(2\pi i)^2} \int_{\mathbb{T}^2} \frac{-u^2 x^2(w^{-n+1} z^{n-1}-w^{-n-1} z^{n+1})}{ P(z,w)} \frac{dz}{z}\, \frac{dw}{w}.
\)	
Taking the transformation $z= v w$ gives 
\(  \nonumber
	K^{-1}(v_3,v_3(n,-n))= \frac{1}{(2\pi i)^2} \int_{\mathbb{T}^2} \frac{u^2 x^2 v^{ n-1} ( v^2-1)} { v w P( v w, w)} dw \,dv.
\)
Using Appendix~\ref{app:sec:factorization}, we can write
\begin{equation}
      v wP(vw,w)= u x^2 (1-u^2 x^2)v (1+v) (w-Y_-(v))(w-Y_+(v))
\end{equation}
where both $Y_\pm(w)$ are defined in~\eqref{app:eqn:Ypm}.  As we have $|Y_+(w)|<1$ for $|w|=1$ and $|Y_-(w)|>1$ for $|w|=1$, we have 
\(  
	K^{-1}(v_3,v_3(n,-n))=\frac{1}{2\pi i} \int_{|v|=1}  \frac{u v^{ n-2} ( v-1)}{(1-u^2 x^2)A(v)} d v \label{inbet:lempr:trans}
\)
 where $A(v)$ is defined in~\eqref{app:expect:thmproof:av}.  
 
Recall the expressions $e_1(\g)$ and $e_2(\g)$ and denote $e_3(\g)=3-\g + 2 \sqrt{2(1-\g)}$ and $e_4(\g)=3-\g-2 \sqrt{2(1-\g)}$.  From Appendix~\ref{app:expect:thmproof:av}, we have that $(A(v))^2$ is a quartic polynomial in $v$.  Using the expansions of the roots of $(A(v))^2$ given in~\eqref{app:rootrug} and~\eqref{app:rootsug}, we can write
 \begin{align}
 	(A(v))^2=(v - (1 + e_1 x + e_3 x^2)) (v - (1 + e_2 x + e_4 x^2))\nonumber \\
		\times (v - (1 - e_1 x + e_3 x^2)) (v - (1 - e_2 x + e_4 x^2))+O(x^3). \label{inbet:lempr:den}
 \end{align}
 Taking the change of variables $v= 1-x \xi$ transforms $(A(v))^2$ given in (\ref{inbet:lempr:den}) into $(\tilde{A}(\xi))^2$ which is given by
 \begin{align}
 	(\tilde{A}(\xi))^2:= x^4( (\xi - ( e_1  + e_3 x)) (\xi - ( e_2  + e_4 x)) \nonumber \\
	\times (\xi + ( e_1  - e_3 x)) (\xi + ( e_2  - e_4 x))+O(x^6).  \nonumber
 \end{align}
 In other words, this transformation ensures that the roots of the quartic polynomial will not pairwise contract to the same point when taking the limit as $x$ tends to 0.  
 Under these  change of variables, $v=1-x \xi$,  (\ref{inbet:lempr:trans}) becomes
 \(  \nonumber
	K^{-1}(v_3,v_3(n,-n))=-\frac{1}{2\pi i}\int_{|\xi-1/x|=1/x} \frac{x u ( 1-x \xi)^{n-2}( ( 1-x \xi)-1)}{ \tilde{A}(\xi)} d\xi. 
\) 
 The contour of integration can be deformed to a contour independent of $x$ surrounding the two roots of $A(\xi)$ with positive real part, namely a contour surrounding $e_1(\g)$ and $e_2(\g)$ which means we can now take limits as $x$ tends to 0 of the integrand.  The numerator of the above integrand is equal to  
 \(  \nonumber
 	-\half e^{-\alpha \xi} \xi x^2 +O(x^3)
\)
as where $A(\xi)$ is equal to 
\(  \nonumber
	-\frac{1}{4} x^2 \sqrt{4+8\g + 4\g^2-12\xi^2+4 \g \xi^2+\xi^4} +O(x^3)
\)
 By the dominated convergence theorem, we obtain
 \(  \nonumber
 	K^{-1}(v_3,v_3(n,-n))=-\frac{1}{2\pi i}\int \frac{2 e^{-\alpha \xi} \xi}{\sqrt{4+8\g + 4\g^2-12\xi^2+4 \g \xi^2+\xi^4}} d\xi.
\)
where the integral is over a contour surrounding $e_1(\g)$ and $e_2(\g)$.  Splitting up the contour into two line integrals gives the required result.

\end{proof}

Although Lemma \ref{inbet:lem:res} incorporates Lemma \ref{uc:lem:longrangeh}, the technique used to calculate Lemma  \ref{uc:lem:longrangeh} is more general.  Moreover, we can use the same technique found in Lemma  \ref{uc:lem:longrangeh} when the exponents of $z$ and $w$ are different.   The calculation for Lemma \ref{inbet:lem:res} cannot be generalized in this way.

\subsection{Correlation Length}

This subsection focuses on proving Lemma \ref{intro:lem:corlen}. Recall  that the correlation length $\xi(\g)$ given in  (\ref{intro:def:corlength}), is defined in terms $C(\alpha, \g)$ where $C(\alpha,\g)=E_1^\g (\alpha)^2-E_2^\g(\alpha)^2$ for $\g\not=1$ and $C(\alpha,1)=0$ for all $\alpha \geq 0$.

\begin{proof}[Proof of Lemma \ref{intro:lem:corlen}]

We  consider first $\g<-1$ and $-1<\g<1$ simultaneously.  When $\g<-1$ and $-1<\g<1$ we have $e_2(\gamma)<e_1(\gamma)$.  The numerator of the integral expansion of $E_1^\g(\a)+E_2^\g(\a)$ is $-(\g+1)-2p-1/2p^2$ which has roots $e_1(\gamma)$ and $e_2(\gamma)$.  Therefore, we can write
\(  \nonumber
	E_1^\g (\a ) +E_2^\g (\a)= \frac{1}{2\pi} \int_{e_2(\g)}^{e_1(\g)} e^{-\alpha p} \sqrt{\frac{(p-e_1(\g))(p-e_2(\g))}{((e_1(\g)-4)-p)(p-(e_2(\g)-4))}}\, dp.
\)
Similarly, the numerator of the integral expansion of  $E_1(\g)-E_2(\g)$ is $\g+1-2p+1/2p^2$ which has roots $e_1(\g)-4$ and $e_2(\gamma)-4$.  Therefore, we can write
\(  \nonumber
E_1^\g (\a ) -E_2^\g (\a)= \frac{1}{2\pi} \int_{e_2(\g)}^{e_1(\g)} e^{-\alpha p} \sqrt{\frac{(p-(e_1(\g)-4))(p-(e_2(\g)-4))}{(e_1(\g)-p)(p-e_2(\g))}}\, dp.
\)
For large values of $\alpha$, we can show that $E_1^\g(\a) \pm E_2^\g(\a)$ are approximately $e^{-\alpha e_2(\gamma)}$.  Indeed for  $\g<-1$ and $-1<\g<1$, large values of $\a$ and for $\e>0$, we can write for $C_1=1/(2\pi) \sqrt{(e_1(\g)-e_2(\g))/(4(4-e_1(\g)+e_2(\g)))}$
\begin{align}
	e^{\alpha e_2(\gamma)} (E_1^\g (\a ) +E_2^\g (\a))&=\frac{1}{2\pi} \int_0^{e_1(\gamma)-e_2(\gamma)} e^{-\alpha p} \sqrt{\frac{p(p-e_1(\g)+e_2(\g))}{(e_1(\g)-e_2(\g)-4-p)(p+4))}}\, dp.  \nonumber\\
	&= C_1 \int_0^{\e} e^{-\alpha p} \sqrt{p} \, dp  +O(\alpha^{-1})  \nonumber \\
	&= C_1 \int_0^{\infty } e^{-\alpha p} \sqrt{p} \,dp +O(\alpha^{-1})  \nonumber\\
	&=\frac{C_1 \sqrt{\pi}}{ 2 \sqrt{ \a^3}} +O(\alpha^{-2})  \nonumber
\end{align}
Using a similar argument, we can write for $C_2=1/(2\pi) \sqrt{4(4-e_1(\g)+e_2(\g))/(e_1(\g)-e_2(\g))}$
\begin{align}
	e^{\alpha e_2(\gamma)} (E_1^\g (\a ) -E_2^\g (\a))=\frac{C_2 \sqrt{\pi}}{2 \sqrt{ \a}} +O(\a^{-1}).  \nonumber
\end{align}
As $|C(\alpha,\g)|=|(E_1(\g)+E_2(\g))(E_1(\g)-E_2(\g))|$, we have for large values of  $\alpha$
\(  \nonumber
	\log|C(\a,\g)|=  -2 \alpha e_2(\g) +o(\log \alpha) 
\)
This means that $\xi(\gamma)$ is equal to $1/(2e_2(\g))$ for $\g<-1$ and $-1<\g<1$.  
 
The above argument does not work for $\g=-1$ or for $\g=1$.  However, by taking a series expansion in $\alpha$ of the terms computed in Lemma~\ref{uc:lem:longrangeh}, we have   $C(\alpha,-1) \sim 1/\alpha^2$.  This means that $\xi(-1)=\infty$. For $\g=1$, we have $C(\alpha, 1)=0$ for all $\a \geq 0$ which implies that $\xi (1)=0$. 

For $\g>1$, the numerator of the integral expansion of $E_1^\g(\alpha)+E_2^\g(\alpha)$ is $-(\g+1)-2p-1/2p^2$ which has roots $e_1(\g)$ and $e_2(\g)$ where $\overline{e_1(\g)}=e_2(\g)$.  Therefore, using the change of variables $p \mapsto 2+ip$, we can write
\begin{align}
E_1^\g (\a ) +E_2^\g (\a) &= \frac{1}{2\pi} \int_{e_2(\g)}^{e_1(\g)} e^{-\alpha p} \sqrt{\frac{(p-e_1(\g))(p-e_2(\g))}{((e_1(\g)-4)-p)(p-(e_2(\g)-4))}} \, dp \nonumber \\
&= \frac{e^{-2 \alpha}}{2\pi} \int_{-r}^r \sqrt{ \frac{r^2 -p^2}{r^2-(p+4i)^2}} e^{-i \alpha p} \, dp \label{ch4:corlength:finala}
\end{align}
where $r=+\sqrt{2(\g-1)}$.  Note that $E_1^\g(\alpha)+E_2^\g(\alpha)$ is real.  We have that 
\(  \nonumber
	\frac{1}{2\pi} \int_{-r}^r \sqrt{r^2-p^2}e^{-i \alpha p} dp=\frac{r}{2 \a}  B_J[1,\a r]
\)
where $B_J[n,z]$ is the Bessel J function of the first kind of order $n$.  The right hand side of the above equation decays at rate $\alpha^{-3/2}$ as $\alpha \to \infty$.  

Suppose that 
\(
	\lim_{\a \to \infty} \alpha^{3/2} \int_{-r}^r  \sqrt{ \frac{r^2 -p^2}{r^2-(p+4i)^2}} e^{-i \alpha p}\, dp = 0.\label{ch4:corlength:gbig1}
\)
We have that $1/\sqrt{r^2-(p+4i)^2}$ is analytic on $[-r,r]$. By Mergelyan's theorem \cite{Gam:69}, we can approximate $1/\sqrt{r^2-(p+4i)^2}$ uniformly by polynomials on $[-r,r]$.  Therefore we can write $1/\sqrt{r^2-(p+4i)^2}=\sum_j p_j(p)$ for $p\in[-r,r]$ where $p_j(z)=\a_j z^j$ where $\a_j \in \mathbb{C}$.  By the assumption, for $\e >0$ there exists $N$ such that
\(  \nonumber
	\lim_{\a \to \infty} \left| \alpha^{3/2}  \int_{-r}^r  \sum_{j=1}^N p_j(p) \sqrt{r^2-p^2} e^{-i \a p} dp  \right| < C\e
\)
where $C$ is an arbitrary constant.  However, we have that
\(  \nonumber
	\lim_{\a \to \infty} \left| \alpha^{3/2}  \int_{-r}^r p_{N+1} (p)  \sqrt{r^2-p^2} e^{-i \a p} dp \right| \not = 0
\)
because the above integral can be computed explicitly and can be shown to decay at rate $\alpha^{-3/2}$.  This means that
\(  \nonumber
	\lim_{\a \to \infty} \left| \alpha^{3/2}  \int_{-r}^r  \sum_{j=1}^{N+1} p_j(p) \sqrt{r^2-p^2} e^{-i \a p} dp  \right| > C\e
\)
which contradicts (\ref{ch4:corlength:gbig1}).  Therefore, we have 
\(\label{ch4:corlength:final1}
	\lim_{\a \to \infty} \alpha^{3/2} \int_{-r}^r  \sqrt{ \frac{r^2 -p^2}{r^2-(p+4i)^2}} e^{-i \alpha p} dp \not = 0.
\)

We have a similar argument for $E_1^\g(\alpha)-E_2^\g(\alpha)$. The numerator of the integral expansion of $E_1^\g(\alpha)-E_2^\g(\alpha)$ is given by $\g+1-2p+1/2p^2$ which has roots $e_1(\g)-4$ and $e_2(\g)-4$ where $\overline{e_1(\g)-4}=e_2(\g)-4$.  Therefore, using the change of variables $p \mapsto 2+ip$, we can write
\begin{align}
E_1^\g (\a ) -E_2^\g (\a) &= \frac{1}{2\pi} \int_{e_2(\g)}^{e_1(\g)} e^{-\alpha p} \sqrt{\frac{((e_1(\g)-4)-p)(p-(e_2(\g)-4))}{(p-e_1(\g))(p-e_2(\g))}} dp \nonumber \\
&= \frac{e^{-2 \alpha}}{2\pi} \int_{-r}^r \sqrt{ \frac{r^2-(p+4i)^2}{r^2 -p^2}} e^{-i \alpha p} dp \label{ch4:corlength:finalb}
\end{align}
where $r=+\sqrt{2(\g-1)}$.  We have that 
\(  \nonumber
	\frac{1}{2\pi} \int_{-r}^r \frac{e^{-i \alpha p}}{\sqrt{r^2-p^2}} dp=\frac{1}{2}  B_J[0,\a r]
\)
where $B_J[n,z]$ is the Bessel function of the first kind of order $n$.  The right hand-side of the above equation decays at rate $\alpha^{-1/2}$ as $\alpha \to \infty$.  

Suppose that 
\( 
	\lim_{\a \to \infty} \alpha^{1/2} \int_{-r}^r  \sqrt{ \frac{r^2-(p+4i)^2}{r^2 -p^2}} e^{-i \alpha p} dp = 0.\label{ch4:corlength:gbig2}
\)
We have that $\sqrt{r^2-(p+4i)^2}$ is analytic on $[-r,r]$. By Mergelyan's theorem \cite{Gam:69}, we can approximate $\sqrt{r^2-(p+4i)^2}$ uniformly by polynomials on $[-r,r]$.  Therefore we can write $\sqrt{r^2-(p+4i)^2}=\sum_j p_j(p)$ for $p\in[-r,r]$ where $p_j(z)=\a_j z^j$ where $\a_j \in \mathbb{C}$.  By the assumption, for $\e >0$  there exists $N$ such that
\(  \nonumber
	\lim_{\a \to \infty} \left| \alpha^{1/2}  \int_{-r}^r  \sum_{j=1}^N p_j(p) \frac{e^{-i \a p}}{\sqrt{r^2-p^2}}\, dp \right| < C\e
\)
where $C$ is an arbitrary constant.  However, we have that
\(  \nonumber
	\lim_{\a \to \infty} \left| \alpha^{3/2}  \int_{-r}^r p_{N+1} (p)   \frac{e^{-i \a p}}{\sqrt{r^2-p^2}}\,dp \right|   \not = 0
\)
because the above integral can be computed explicitly and can be shown to decay at rate $\alpha^{-1/2}$.  This means that 
\(  \nonumber
	\lim_{\a \to \infty} \left| \alpha^{1/2}  \int_{-r}^r  \sum_{j=1}^N p_j(z) \frac{e^{-i \a p}}{\sqrt{r^2-p^2}} dp  \right| > C\e
\)
which contradicts (\ref{ch4:corlength:gbig2}). Therefore, we have 
\( \label{ch4:corlength:final2}
	\lim_{\a \to \infty} \alpha^{1/2} \int_{-r}^r  \sqrt{ \frac{r^2-(p+4i)^2}{r^2 -p^2}} e^{-i \alpha p} dp \not = 0.
\)

Therefore, using (\ref{ch4:corlength:finala}), (\ref{ch4:corlength:final1}), (\ref{ch4:corlength:finalb}) and (\ref{ch4:corlength:final2}), we have that for $\g>1$ and large values of $\alpha$,
\(  \nonumber
	\log|C(\alpha,\g)| = \log|(E_1^\g(\alpha) + E_2^\g(\alpha))(E_1^\g(\alpha) - E_2^\g(\alpha))|=-4 \a + o(\log(\a^2)).
\)
We can substitute the above expression for $\log|C(\alpha,\g)|$  into the definition of $\xi(\gamma)$.

\end{proof}


\section{Conclusion} \label{sec:con}

In this paper, we considered a new parameterization of the dimer model on the Fisher lattice.  We were able to completely determine the behavior in the thermodynamic limit when $x \to 0$ and in the scaling window $(x,x^2)$ when $u=u_i$ and obtain partial results in the scaling window for $u=u_\g$, $\g \not=1$.  

For $0<x<1$ and $u=u_c$, David Wilson conjectured (personal communication) that the scaling limit of the particle model with appropriate re-scaling is conformally invariant.  In the same way that the scaling limit of percolation (see \cite{cam:06}) required SLE and CLE (see \cite{Wen:04} or \cite{law:05}),  we expect that under the correct aspect ratio, the scaling limit of the paths in the particle model are given by a random family of curves and loops belonging to $\mathrm{SLE}_3$.

As noted before, we have been unable to find the limiting measure for the scaling window $(x,x^2)$ when $u=(1-\sqrt{1-\g x^2})/(\g x^2)$ for $\g \not =0$ or $1$.  It remains an open problem to determine the scaling window measure for $\g \not =1$.  In particular, we expect that for $\g=-1$, the scaling window measure is scale invariant but not rotationally invariant.   This is due to the underlying dimer model (or Ising model) being anisotropic.    

If we take $u=u_i+\e^2$ where $\e$ is order $x^\alpha$ for $\alpha \in (0,1)$, then some computations show that
\(
	\E[N_b]=\E[N_a]+O(x)=C \e +O(x)
\)
where $C$ is a constant and 
\(
	\E[N_X]=x^2+ \e O(x^2 \log x).
\)
By taking the scaling window $(x^\alpha, x^{2 \alpha})$ implies that the density of particles is of constant order.  For $\alpha < 2/3$, there is a dense set of creations in the scaling window.  For $\alpha>2/3$, then there are no creations or annihilations seen in the scaling window $(x^\alpha, x^{2 \alpha})$.  We expect that the limiting measure to be Dyson Brownian motions.  For $\alpha=2/3$, the creations have a finite density.  Roughly speaking, we expect the repulsion of the particles to be weak enough to allow some annihilations.  


\begin{appendix}

\section{Overview of Appendix}

This appendix provides some additional information about the model including details of the some computations used in this paper.  In Appendix~\ref{app:secA}, we analyze the characteristic polynomial for different values of $u$ and $x$.  In Appendix~\ref{app:cofactor}, we give some of the entries of the cofactor matrix.  In Appendix~\ref{app:localprobs}, we give the computations that were used in Section~\ref{sec:exp}.  Finally, in Appendix~\ref{app:H} we give an expansion of $H$, defined in~\ref{definitionofH}, in terms of $x$ for $u<1/2$, $u=u_c$ before taking the scaling window.  We also compute $H$ exactly when $u=u_i$.

\section{Analysis of $P(z,w)$} \label{app:secA}

The characteristic polynomial for the dimer model on the Fisher lattice with parameters $u$ and $x$ is given by
\(
	P(z,w)=x^2+2u^2 x^2 + u^4 x^6+ \left( u x^2 -u^3 x^4 \right) \left(w+\frac{1}{w}+z+\frac{1}{z} \right) + \left( u^2 x^2-u^2 x^4 \right) \left(\frac{z}{w} + \frac{w}{z} \right)
\)

\subsection{Critical Values}
The \emph{critical values} of the dimer model is defined to be the parameters which give $P(z,w)=0$ for $(z,w) \in \mathbb{T}^2$.  By \cite{Li:10, Li1:10}, for $(z,w)\in \mathbb{T}^2$ we have $P(z,w)=0$ when $(z,w)$ is real.   The values of $u$ when $P(-1,1)=0$ and $P(1,-1)=0$ do not have much significance for this paper.  By considering $P(1,1)=0$ and $P(-1,-1)=0$ we have that
\( \label{ap:charpoly}
	x^2 \pm 4 u x^2 + 4 u^2 x^2 - 2 u^2 x^4 \pm 4 u^3 x^4 + u^4 x^6=0
\)
Solving (\ref{ap:charpoly}) for $u$ in terms of $x$ gives $u=(-1 \pm \sqrt{1+x^2})/x^2$ and $u=(1 \pm \sqrt{1+x^2})/x^2$.  Therefore, the critical values for the dimer model on the Fisher lattice with parameters $(u,x)$ are given by $u=(\pm1+\sqrt{1+x^2})/x^2$ because the other two solutions are negative when $x>0$.  

For fixed $x<1$, we have that the critical value $u=(-1+\sqrt{1+x^2})/x^2 \in (0,1)$  while the other critical value is not in the interval $(0,1)$.  Recall that we defined $u_c=(-1+\sqrt{1+x^2})/x^2$.

\subsection{Factorization of $P(z,w)$} \label{app:sec:factorization}
In this subsection, we consider a factorization of the characteristic polynomial for $u \in (0,1)$ and for $x<1$.  We find that this factorization leads to a degree 4 polynomial (in one variable) whose roots turn out to govern the behavior of the model.  We introduce these roots and analyze their values for $u\in (0,1)$ and $x <1$.

Set $z=v w$, we have that
\( \nonumber
	P(v w,w)v w = C(u,x,v)\left(w- Y_+(v)\right) \left(w-Y_-(v) \right)
\)
where 
\( \nonumber
    C(u,x,v)=u x^2 (1-u^2 x^2)v   (1+v),
\) 
\( \label{app:eqn:Ypm}
	Y_{\pm}(v)=\frac{-u^2 - v - 2 u^2 v - u^2 v^2 + u^2 x^2 + u^2 v^2 x^2 - u^4 v x^4\pm A(v)}{2 (u v + u v^2 - u^3 v x^2 - u^3 v^2 x^2)}.
\)
and
 \begin{align} \label{app:expect:thmproof:av}
	(A(v))^2=(-4 (u + u v - u^3 x^2 - u^3 v x^2) (u v + u v^2 - u^3 v x^2 - u^3 v^2 x^2) +  \\ 				(u^2 + v + 2 u^2 v + u^2 v^2 - u^2 x^2 - u^2 v^2 x^2 + u^4 v x^4)^2). \nonumber
\end{align} 
We have that $|Y_+(v)|<1$ for all $|v|=1$ and $|Y_-(v)|>1$.  We have that $(A(v))^2$ is a polynomial of degree 4 in $v$ and its roots are given by $r(u,x), 1/r(u,x), s(u,x)$ and $1/s(u,x)$ where 
\( \label{ap:root1}
	r(u,x)=\frac{1-2u^2+u^4 x^4 +(1-u^2x^2)\sqrt{(1 - 2 u + u^2 x^2) (1 + 2 u + u^2 x^2)}}{2u^2 (1 + x)^2}
\)
and 
\( \label{ap:root2}
	s(u,x)=\frac{1-2u^2+u^4 x^4 +(1-u^2x^2)\sqrt{(1 - 2 u + u^2 x^2) (1 + 2 u + u^2 x^2)}}{2u^2 (1 - x)^2}
\)
Therefore, we can write
\( \nonumber
	(A(v))^2=  u^4(1-x^2)^2 (v-r(u,x)) ( v-{r(u,x)}^{-1} ) (v-s(u,x) ) ( v -{s(u,x)}^{-1} ).
\)
We now consider the values of the roots of $(A(v))^2$ for fixed $x<1$ and $u\in(0,1)$. 

By plugging in $u=u_c$ into (\ref{ap:root1}) and (\ref{ap:root2}), we obtain $s(u_c,x)=(1+x)^2/(1-x)^2$ and $r(u_c,x)=1$.  Therefore, the roots of $(A(v))^2$ are given by $(1-x)^2/(1+x)^2$, $1$ (twice) and $(1+x)^2/(1-x)^2$ when we take $u=u_c$.

For fixed $0<x<1$ and $0<u<u_c$, we have that $r(u,x)$ and $s(u,x)$ is real due to $1-2u+u^2 x^2>0$.  As the denominator of $r(u,x)$ is greater than the denominator of $s(u,x)$, we have that $r(u,x)<s(u,x)$.  We also have that $s(u,x)<1$. 

When $u=u_i$, we have that $r(u_i,x)=(1-x)/(1+x)$ and $s(u_i,x)=(1+x)/(1-x)$.  Furthermore, we can write 
\( \nonumber
	A(v)=\frac{\left(1-\sqrt{1-x^2}\right)^2}{x^4}(1-x^2)\left( v- \frac{1-x}{1+x} \right) \left( v- \frac{1+x}{1-x} \right)
\)
when $u=u_i$. 

For fixed $0<x<1$ and $u_c<u<u_i$, we still have $r(u,x)$ and $s(u,x)$ are both real due to $1-2u+u^2 x^2>0$ and $r(u,x)<s(u,x)$.  However, we now have $r(u,x)<1<s(u,x)$.

When $u>u_i$, both $r(u,x)$ and $s(u,x)$ are complex.  By manipulating the expression of $r(u,x)$, we have that $\overline{r(u,x)}=1/s(u,x)$.   It follows that the roots of $(A(v))^2$ are given by $r(u,x)$, $1/r(u,x)$, $\overline{r(u,x)}$ and $1/\overline{r(u,x)}$.  

The roots of $(A(v))^2$ can be summarized as follows:
\begin{itemize}
\item For $u<u_c$, the roots are $r(u,x),s(u,x),1/r(u,x)$ and $1/s(u,x)$ with $r(u,x), s(u,x) \in \R$ with $1<r(u,x)<s(u,x)$.
\item For $u=u_c$, the roots are $s(u_c,x),1$ (twice) and $1/s(u_c,x)$ where $s(u_c,x)=(1+x)^2/(1-x)^2>1$.
\item For $u_c<u<u_i$, the roots are $r(u,x),s(u,x),1/r(u,x)$ and $1/s(u,x)$ with $r(u,x)<1<s(u,x)$ and $r(u,x), s(u,x) \in \R$
\item For $u=u_i$, the roots are $r(u_i,x)$ (twice) and $1/r(u_i,x)$ (twice) where $r(u_i,x)=(1-x)/(1+x)<1$.
\item For $u>u_i$ ($u<1/x$), the roots are $r, \overline{r(u,x)}, 1/r(u,x)$ and $1/\overline{r(u,x)}$ for   $r(u,x) \in \mathbb{C}$ with $|r|<1$.
\end{itemize}

Finally, we consider the behavior of the roots defined in (\ref{ap:root1}) and (\ref{ap:root2}) when $x$ is small.  We have that for $u<1/2$ (not a function of $x$)
\( \label{app:rootruless}
	r(u,x)=\frac{1-2u^2+\sqrt{1-4u^2}}{2u^2} - \frac{1-2u^2+\sqrt{1-4u^2}}{u^2} x +O(x^2)
\) 
and
\( \label{app:rootsuless}
	s(u,x)=\frac{1-2u^2+\sqrt{1-4u^2}}{2u^2} + \frac{1-2u^2+\sqrt{1-4u^2}}{u^2} x +O(x^2).
\)

When $u=u_\g=(1-\sqrt{1-\g x^2})/(\g x^2)$ for $\g \not=0$ and $u_0=1/2$, we have that
\( \label{app:rootrug}
	r(u_\g,x)= 1 -(2+\sqrt{2-2\g} )x+(3+2\sqrt{2-2 \g}-\g) x^2+O(x^3)
\)
and
\( \label{app:rootsug}
	s(u_\g,x)=1 +(2+\sqrt{2-2\g} )x+(3-2\sqrt{2-2 \g}-\g) x^2+O(x^3).
\)

When $u>1/2$ and not a function of $x$, we have
\( \label{app:rootrubig}
	r(u,x)=\frac{1-2u^2 +i \sqrt{4u^2-1}}{2u^2} - \frac{1-2u^2+i \sqrt{4u^2-1}}{u^2} x +O(x^2)
\) 
and
\( \label{app:rootsubig}
	s(u,x)=\frac{1-2u^2 + i \sqrt{4u^2-1}}{2u^2} + \frac{1-2u^2+i \sqrt{4u^2-1}}{u^2} x +O(x^2).
\)
which means that $| |r(u,x)|-1|$ is $\Theta(x)$.  

Finally, we make the following remarks based on the above expansions:  if we choose $u<1/2$ independently of $x$, then we have $ \lim_{x \to 0} r(u,x)-s(u,x)=0$ with $\lim s(u,x)>1$.     However, for $u \geq u_c$, we have that  $\lim_{x \to 0} |r(u,x)| =\lim_{x \to 0} |s(u,x)|=1$.

\section{Cofactor Matrix of $K(z,w)$} \label{app:cofactor}

Let $C(z,w)_{i,j}$ be the $(i,j)^{th}$ entry of the cofactor matrix of $K(z,w)$ defined in (\ref{Kasteleyn}).  The complete cofactor matrix is a rather large expression; we will only list a few important entries of the cofactor matrix that are required in this paper.  Due to $K(z,w)$ is anti-symmetric, we have $C(z,w)_{i,j}=\overline{C(z,w)_{j,i}}$. We have
\( \nonumber
	C(z,w)_{1,1}=u x^2 ( w-w^{-1}),
\)
\( \nonumber
	C(z,w)_{1,2}=x^2( u^2 x^2 z w^{-1}-1-uw^{-1} -u z),
\)
\( \nonumber
	C(z,w)_{1,5}=x(1 + u w  + u  z - u^2 w x^2 z),
\)
\( \nonumber
	C(z,w)_{1,6}= x(1- u^2 x^2 + u w(1  - x^2) + u  z(1+u^2 x^4) - u^2 x^2 z w^{-1} - u^2 w x^2 z ),
\)
\( \nonumber
	C(z,w)_{2,2}=u x^2 (z^{-1} - z),
\)
\( \nonumber
	C(z,w)_{2,5}=x(1  - u^2 x^2+ uw(1+u^2 x^4)  +u z(1- x^2)  - u^2 w x^2 z^{-1} + u  z  - 
 u^2 w x^2 z),
\)
\( \nonumber
	C(z,w)_{2,6}=x(1+ u w  + u  z - u^2 w x^2 z),
\)
\( \nonumber
	C(z,w)_{3,3}= u^2 x^2(    z w^{-1}-w  z^{-1}),
\)
\( \nonumber
	C(z,w)_{3,4}=x(1+u^4 x^4 + u w(1 - u^2  x^2) +  u  z(1- u^2 x^2) - u^2 w x^2 z^{-1}  - u^2 x^2 z w^{-1}),
\)
\( \nonumber
	C(z,w)_{4,4}=u^2 x^2 (w  z^{-1} -  z w^{-1}),
\)
\( \nonumber
	C(z,w)_{5,5}=u x^2( z-z^{-1}),
\)
\( \nonumber
	C(z,w)_{5,6} =x^2(1 + u  w^{-1} + u z - u^2 x^2 z w^{-1}),
\)
and 
\( \nonumber
	C(z,w)_{6,6}=u x^2( w^{-1} -  w )
\)
where $a=x$ and $b=u x$.  

\section{Computations of Local Probabilities} \label{app:localprobs}

In this section, we provide some of the explicit computations used to compute the exact and approximations of the local probabilities of this model.  We also provide the proofs of Lemma~\ref{expect:lem:expand} and Lemma~\ref{exp:lem:sing}.  Due to the technical nature of these computation, we used the aid of computer algebra.

We have that
\begin{lemma} \label{app:therm:exp}
\begin{align} \label{app:exp:therm1}
	-b\frac{\partial}{\partial b}\log Z =  -\frac{2u^2 x^2}{1-u^2 x^2} +	\left\{ \begin{array}{ll} 
					f(x,u,r) & \mbox{if } u<u_c \\
					f(x,u,1) &  \mbox{if } u=u_c \\
					f(x,u,1/r) &   \mbox{if } u>u_c\\ \end{array}\right.
\end{align}
where $f(x,u,r)$ is defined in Theorem \ref{expect:thm:expectb} and $b=u x$.  In the case when $u=u_i$, we obtain
\begin{equation}
    -b\frac{\partial}{\partial b}\log Z=1-\sqrt{\frac{1-x}{1+x}}
\end{equation}

\end{lemma} 

\begin{proof}
From~\eqref{setup:logZ}, we can differentiate both sides with respect to $b$ and we obtain
\begin{equation*}
      -b\frac{\partial}{\partial b}\log Z = -\frac{b}{2(2\pi i)^2} \int_{|w|=1} \int_{|z|=1} \frac{ \der{b} P(z,w)}{P(z,w)} \frac{dz}{z} \, \frac{dw}{w}
\end{equation*}
by differentiation under the integral sign.  Under the change of variables $z \mapsto w v$, we have 
\begin{equation} \label{appB:expectedbchangeofvariables}
\begin{split}
      &-b\frac{\partial}{\partial b}\log Z = \frac{1}{2(2\pi i )^2} \int_{|v|=1}  \int_{|w|=1} dw \, dv \\
      &\frac{u (-(1 + v) (1 + 2 u (1 + v) w + v w^2) + u (2 (1 + v^2) w + 3 u (1 + v) (1 + v w^2)) x^2 - 4 u^3 v w x^4)}{v w (-(u + u v + v w) (1 + u w + u v w) + 
   u^2 (w + v^2 w + u (1 + v) (1 + v w^2)) x^2 - u^4 v w x^4)}.
\end{split}
\end{equation}
 The roots of the denominator of the integrand of~\eqref{appB:expectedbchangeofvariables} with respect to $w$ are given by 0, $Y_{\pm}(v)$ which are both defined in (\ref{app:eqn:Ypm}).  The residue at $w=0$ for the right hand side of~\eqref{appB:expectedbchangeofvariables} gives
\begin{equation}
	\frac{1}{2}\frac{1}{2 \pi i} \int_{|v|=1} \frac{-1 + 3 u^2 x^2}{ v  (-1 + u^2 x^2)} dv=  \frac{1 - 3 u^2 x^2}{2(1 - u^2 x^2)}. \label{expect:thm:wres0}
\end{equation}
As we have $|Y_+(v)|<1$ for all $|v|=1$ and $|Y_-(v)|>1$ for all $|v|=1$, we can compute the residue with respect to $w$ at $Y_+(v)$ for the right hand side of~\eqref{appB:expectedbchangeofvariables}.  This gives 
\begin{equation} \label{expect:thm:w1res}
	\frac{1}{2 \pi i}\int_{|v|=1}  \frac{Q(v)}{2 v (-1+u^2 x^2) A(v)} dv 
\end{equation}
where 
\begin{equation*}
	Q(v)=(1 + u^2 x^2) (v + u^4 v x^4 + u^2 (-1 + x^2 + v^2 (-1 + x^2) - 2 v (1 + 2 x^2))).
\end{equation*}

Recall that the roots of $A(v)$ are given $r(u,x),s(u,x),1/r(u,x),1/s(u,x)$, where $r(u,x)$ and $s(u,x)$ are defined in (\ref{ap:root1}) and (\ref{ap:root2}).  These roots were analyzed in Appendix \ref{app:sec:factorization}. For all values of $u \in (0,1)$, we have $|s(u,x)|>1$.  For $u<u_c$, we have $|r(u,x)|>1$.  For $u=u_c$, we have $r(u_c,x)=1$ and for $u>u_c$, we have $|r(u,x)|<1$.  We will consider~\eqref{expect:thm:w1res} under the three regimes $u<u_c$, $u=u_c$, $u>u_c$ and the special case when $u=u_i$ separately.   We denote $\mathcal{C}(\tilde{r},\tilde{s})$ to be a positively oriented contour surrounding the points $\tilde{r}(u,x)$ and $\tilde{s}(u,x)$ which does not contain the origin where $\tilde{r}$ and $\tilde{s}$ are functions of $u$ and $x$.

For $u<u_c$, the contour of integration for the integral in (\ref{expect:thm:w1res}) can be deformed to a contour surrounding the branch cut from $1/r(u,x)$ to $1/s(u,x)$ and a contour around $v=0$.  We can take the residue at $v=0$ and we find that~\eqref{expect:thm:w1res} is equal to 
\begin{equation}
\label{expect:thm:w1resuless}
	-\frac{1+u^2 x^2}{2(1-u^2 x^2)} +  \frac{1}{2 \pi i}\int_{\mathcal{C}(1/r,1/s)}  \frac{Q(v)}{2 v (-1+u^2 x^2) A(v)} dv. 
\end{equation}
The integral in the above equation can be rewritten as two line integrals.  We have that~\eqref{expect:thm:w1resuless} becomes 
\begin{equation}\label{expect:thm:w1resuless1}
      -\frac{1+u^2 x^2}{2(1-u^2 x^2)} +\frac{1}{2 \pi i}\int_{1/r}^{1/s}  \frac{Q(v)}{ v (-1+u^2 x^2) A(v)} dv. 
\end{equation}
Adding~\eqref{expect:thm:wres0} to~\eqref{expect:thm:w1resuless1} we find that for $u<u_c$
\begin{equation*}
-b \der{b} \log Z=-\frac{2u^2 x^2}{1-u^2 x^2} -\frac{1}{2 \pi i}\int_{1/s}^{1/r}  \frac{Q(v)}{ v (-1+u^2 x^2) A(v)} dv.
\end{equation*}
As $Q(v)$ is a polynomial in $v$ of degree 2, we have that the above line integral is a linear combination of  elliptic integrals.  We can simplify this linear combination of  elliptic integrals to obtain the expression found in~\eqref{app:exp:therm1} for $u<u_c$.

For $u=u_c$, the contour of integration for the integral in (\ref{expect:thm:w1res}) can be deformed to a contour surrounding the branch cut from $1/s(u_c,x)=(1-x)^2/(1+x)^2$ to $1$ and another contour around $v=0$.  We also have that $A(v)$ is given by 
\begin{equation} \label{expect:thm:u=ucAv}
	u_c^2(1-x^2)(v-1)\sqrt{(v-s(u_c,x))(v-1/s(u_c,x))}.
\end{equation}
Using~\eqref{expect:thm:u=ucAv}, the above contour deformation and computing the residue at $v=0$,~\eqref{expect:thm:w1res} becomes
\begin{equation*}
      -\frac{1+u^2 x^2}{2(1-u^2 x^2)} +  \frac{1}{2 \pi i}\int_{\mathcal{C}(1,1/s(u_c,x))}  \frac{Q(v)}{2 v u_c^2(1-x^2)(v-1)\sqrt{(v-s(u_c,x))(v-1/s(u_c,x))}} dv. 
\end{equation*}
The above integral can be reduced to two line integrals and we obtain
\begin{equation} \label{expect:thm:w1resu=uc}
      -\frac{1+u^2 x^2}{2(1-u^2 x^2)} +  \frac{1}{2 \pi i}\int_1^{1/s(u_c,x)}  \frac{Q(v)}{ v u_c^2(1-x^2)(v-1)\sqrt{(v-s(u_c,x))(v-1/s(u_c,x))}} dv. 
\end{equation}
Adding~\eqref{expect:thm:wres0} to~\eqref{expect:thm:w1resu=uc} we find that for $u=u_c$
\begin{equation*}
-b \der{b} \log Z=-\frac{2u^2 x^2}{1-u^2 x^2} + \frac{1}{2 \pi i}\int_1^{1/s(u_c,x)}  \frac{Q(v)}{ v u_c^2(1-x^2)(v-1)\sqrt{(v-s(u_c,x))(v-1/s(u_c,x))}} dv.
\end{equation*}
The line integral in the above equation can be evaluated explicitly by expanding $Q(v)/( v (v-1))$ in terms of its partial fractions decomposition which gives a linear combination of standard integrals.  Evaluating  gives the expression found in~\eqref{app:exp:therm1} for $u=u_c$.

For $u>u_c$, the contour of integration for the integral in (\ref{expect:thm:w1res}) can be deformed to a contour surrounding the branch cut from $r(u,x)$ to $1/s(u,x)$ and a contour around $v=0$.  We can take the residue at $v=0$ and we find that~\eqref{expect:thm:w1res} is equal to 
\begin{equation}
\label{expect:thm:w1resubig}
	-\frac{1+u^2 x^2}{2(1-u^2 x^2)} +  \frac{1}{2 \pi i}\int_{\mathcal{C}(r,1/s)}  \frac{Q(v)}{2 v (-1+u^2 x^2) A(v)} dv. 
\end{equation}
The integral in the above equation can be rewritten as two line integrals.  We have that~\eqref{expect:thm:w1resubig} becomes 
\begin{equation}\label{expect:thm:w1resubig1}
      -\frac{1+u^2 x^2}{2(1-u^2 x^2)} +\frac{1}{2 \pi i}\int_{r}^{1/s}  \frac{Q(v)}{ v (-1+u^2 x^2) A(v)} dv. 
\end{equation}
Adding~\eqref{expect:thm:wres0} to~\eqref{expect:thm:w1resubig1} we find that for $u>u_c$
\begin{equation*}
-b \der{b} \log Z=-\frac{2u^2 x^2}{1-u^2 x^2} -\frac{1}{2 \pi i}\int_{1/s}^{r}  \frac{Q(v)}{ v (-1+u^2 x^2) A(v)} dv.
\end{equation*}
The line integral in the above equation can be evaluated giving the elliptic integral found in~\eqref{app:exp:therm1} for $u>u_c$
      
Finally, when $u=u_i$, we could either simplify the expression for $\E[N_b]$ given in Theorem~\ref{expect:thm:expectb} for $u>u_c$ because $k=0$ for $u=u_i$ or we could use residue calculus.  We choose the latter.  By setting $u=u_i$ and simplifying, we find that~\eqref{expect:thm:w1res} is equal to 
\begin{equation} \label{expect:thm:wres1ui}
      \frac{1}{2\pi i}\int_{|v|=1} \frac{(-1 + v)^2 - (1 + (-6 + v) v) x^2}{2 v (1 + v (-1 + x) + x) (-1 + v + x + v x) \sqrt{1 - x^2}} dv=-\frac{1-2 x}{\sqrt{1-x^2}}
\end{equation}
by computing the residues at $v=0$ and $v=(1-x)/(1+x)$.  Adding~\eqref{expect:thm:wres0} to~\eqref{expect:thm:wres1ui} and setting $u=u_i$ gives the result.

\end{proof}

We now prove Lemma~\ref{expect:lem:expand}.

\begin{proof}[Proof of Lemma~\ref{expect:lem:expand}]
For $u<u_c$ (fixed), we have that $\E[N_b]$ is of the form 
\begin{equation}
      -\frac{2u^2 x^2}{1-u^2 x^2}+c_1(u,x) \mathcal{K}(k)+ c_2(u,x)(\Pi(rk,k)-\Pi(k/r,k)).
\end{equation}
by Theorem~\ref{expect:thm:expectb}.  We can use the expansions of $r(u,x)$ and $s(u,x)$ given in~\eqref{app:rootruless} and~\eqref{app:rootsuless} to find that
\begin{equation} \label{app:expectexpandcoeffuless1}
    c_1(u,x)=-\frac{2}{\pi}x+ \frac{2}{\pi \sqrt{1-4u^2}}x^2+O(x^3) 
\end{equation}
and
\begin{equation}\label{app:expectexpandcoeffuless2}
    c_2(u,x)=\frac{1}{\pi}+ \frac{2(2u^2-1)}{\pi\sqrt{1-4u^2}}x +O(x^2).
\end{equation}
Using the expansions of $r(u,x)$ and $s(u,x)$ given in~\eqref{app:rootruless} and~\eqref{app:rootsuless}, we can find the expansion of the modulus of the elliptic integrals.  This is given by
\begin{equation}
k= \frac{s-r}{rs-1}=\frac{4u^2 }{\sqrt{1-4u^2}} x+O(x^3).
\end{equation}
The expansions of the elliptic integrals are given by
\( \label{expect:lem:expandlemproof1}
	\mathcal{K}(k)= \frac{\pi}{2} + \frac{k^2 \pi}{4} +O(k^4)
\)
and 
\( \label{expect:lem:expandlemproof2}
	\Pi(rk,k)-\Pi(k/r,k)= \frac{k \pi (-1+r^2)}{4 r} +\frac{3 k^2 \pi (-1+r^4)}{16 r^2} + O(k^3).
\)
We can substitute the series expansions of $k$ and $r$ into~\eqref{expect:lem:expandlemproof1} and~\eqref{expect:lem:expandlemproof2} to give series expansions in terms of $x$ which are given by
\begin{equation}
\mathcal{K}(k)=\frac{\pi}{2} + \frac{4 u^4 \pi }{1-4u^2} x^2 +O(x^4)
\end{equation}
and 
\begin{equation}
      \Pi(rk,k)-\Pi(k/r,k)= \pi x+ \frac{1-2u^2}{\sqrt{1-4u^2}} \pi x^2 +O(x^3).
\end{equation}
These series expansions in $x$ can be substituted into the expression for $\E[N_b]$ given in Theorem~\ref{expect:thm:expectb} along with the expansions of the coefficients $c_1(u,x)$ and $c_2(u,x)$ computed in~\eqref{app:expectexpandcoeffuless1} and~\eqref{app:expectexpandcoeffuless2} to find the series expansion of $\E[N_b]$ for $u<u_c$ (fixed). 

For $u=u_c$ and $u=u_i$,  we have explicit expressions for $\E[N_b]$ which do not involve elliptic integrals given in Theorem ~\ref{expect:thm:expectb}.  In both cases, we can take a series expansions in $x$.

For $u>u_i$ (fixed), we also have that $\E[N_b]$ is of the form 
\begin{equation}
      -\frac{2u^2 x^2}{1-u^2 x^2}+c_1(u,x) \mathcal{K}(k)+ c_2(u,x)(\Pi(k/r,k)-\Pi(r k,k)).
\end{equation}
by Theorem~\ref{expect:thm:expectb}.  Using the expansions of $r(u,x)$ and $s(u,x)$ given in~\eqref{app:rootrubig} and~\eqref{app:rootsubig}, we have
\begin{equation} \label{app:expectexpandcoeffubig1}
      c_1(u,x)= - i\frac{\sqrt{4u^2-1}}{2\pi u^2} +O(x)
\end{equation}
and 
\begin{equation} \label{app:expectexpandcoeffubig2}
      c_2(u,x)= - i \frac{\sqrt{4u^2-1}}{4 \pi u^2 x} +O(x^0).
\end{equation}
 The modulus of each elliptic integral is given by $k=i/k'$ where $k'=|s/r -1|/|s-1/r|$ and $i=\sqrt{-1}$.  Using the expansions of $r(u,x)$ and $s(u,x)$ given in~\eqref{app:rootrubig} and~\eqref{app:rootsubig}, we have
\begin{equation}
    k'=\frac{4u^2}{\sqrt{4u^2-1}}x+O(x^2)
\end{equation}
The expansions of the elliptic integrals with modulus equal to $k=i/k'$ are given by
\( \label{expect:lem:expand2lemproof1}
	\mathcal{K}\left(\frac{i}{k'} \right)=k'\log \left( \frac{2}{k'} \right) +  k' \log 2+O(k'^2)
\)
and
\( \label{expect:lem:expand2lemproof2}
	\Pi\left(\frac{i}{r k'}, \frac{i}{k'} \right)-\Pi\left(\frac{i r}{k'}, \frac{i}{k'} \right)=-k' \log \frac{1}{r} +O(k'^2).
\)
These expansions are both computed explicitly in Chapter 3 of \cite{sc:11}.  We can rewrite the expansions given in~\eqref{expect:lem:expand2lemproof1} and~\eqref{expect:lem:expand2lemproof2} in terms of $u$ and $x$. We obtain
\(
    \mathcal{K}\left(\frac{i}{k'} \right)=i\frac{4u^2}{\sqrt{4u^2-1}} \left( \frac{1}{2} \log (4u^2-1) -2\log u -\log x \right) x +O(x^2)
\)
and
\begin{equation}
 \begin{split}
    \Pi\left(\frac{i}{r k'}, \frac{i}{k'} \right)-\Pi\left(\frac{i r}{k'}, \frac{i}{k'} \right)&= -i\frac{4u^2}{\sqrt{4u^2-1}} \log \left( \frac{1-2u^2 -i\sqrt{4u^2-1}}{2u^2} \right) x +O(x^2)\\ 
    &= -i\frac{4u^2}{\sqrt{4u^2-1}}  \arccos\left(\frac{1}{2u^2}-1 \right)x +O(x^2)\\
    &=-i\frac{8u^2}{\sqrt{4u^2-1}}  (\theta- \pi)x +O(x^2)
\end{split}
\end{equation}
where $\theta = \arccos(-1/(2u))$.  We can substitute these expansions of the elliptic integral and the expansions of the coefficients $c_1(u,x)$ and $c_2(u,x)$ computed in~\eqref{app:expectexpandcoeffubig1} and~\eqref{app:expectexpandcoeffubig2}  back into the expression given in Theorem~\ref{expect:thm:expectb}.  Extracting the highest order coefficient of $x$ (i.e. the $x^0$ term) gives the result. 

The last statement follows by finding the expansions  $\E [N_{a^c}]$ which can be achieved using the same method.

\end{proof}

We now give the proof of Lemma~\ref{exp:lem:sing}.

\begin{proof}
For $u=u_i$ or $u=u_c$, the same method used in Lemma  \ref{expect:lem:expand} holds.  Otherwise, notice that
$\E [N_X]$ is of the form 
\(\nonumber
	c_1(x,u) K(k)+c_2(x,u) (\Pi(r/k,k)-\Pi(1/(r k),k) +O(x^2)
\)
where $c_1(x,u), c_2(x,u)$ represent coefficients of the complete elliptic integrals and $k$ represents the appropriate modulus of the elliptic integrals, given in Theorem \ref{expect:thm:expectb}.   For $u<u_c$, $c_1$ and $c_2$ are both of order $x^2$, hence the expansions of the elliptic integrals given in the proof of Lemma \ref{expect:lem:expand} hold and we can follow the approach used in the proof of Lemma~\ref{expect:lem:expand} to obtain the result.  For $u>u_i$, $c_1(x,u)$ and $c_2(x,u)$ are both $\Theta(x)$, which implies that we can use the expansions of the elliptic integrals used in the proof of Lemma~\ref{expect:lem:expand} and follow the computation in the proof of Lemma~\ref{expect:lem:expand}.
\end{proof}

\subsection{For $u<1/2$}  \label{app:uless:examples}
 When $u<1/2$ (independent of $x$), the roots $r(u,x)$ and $s(u,x)$ defined in (\ref{ap:root1}) and (\ref{ap:root2}) are not close to the unit circle when $x$ is small.  In this case, we can use a series expansion in $x$ to find an expansion in $x$ of the inverse Kasteleyn matrix.  For example,
 \begin{lemma}
 \( \nonumber
 K^{-1}(v_3,v_4)=-\frac{1}{x}- \frac{2u^2-1+\sqrt{1-4u^2}}{\sqrt{1-4u^2}} x+O(x^3)
 \)
 \end{lemma}

\begin{proof}
We have that
\begin{equation} \label{app:uless:example1}
\begin{split}
	&K^{-1}(v_3,v_4)\\
	&= \frac{1}{(2\pi i)^2} \int_{\mathbb{T}^2} \left( \frac{-1}{x(1+uw+uz)} + \frac{(u^4 w^2 + u^3 w^3 + 2 u^4 w z + u^4 z^2 + u^3 z^3) x}{(1 + u w + u z)^2 (u w + u z + w z)}\right) \frac{dz}{z}\frac{ dw}{w} +O(x^3)
	\end{split}
\end{equation}
The first term in equation (\ref{app:uless:example1}) gives $-1/x$ when applying the residue theorem.   Computing the residue at $z=- uw/(u+w)$, the second term in (\ref{app:uless:example1}) becomes
\(
	\frac{1}{2\pi i } \int_{|w|=1} -\frac{ u^2 w^2 (4u^2+4uw+w^2)x}{(u+w)(u+w+uw^2)^2}  dw.
\)
Applying the residue theorem at $w=-u$ and $w=(-1+\sqrt{1-4u^2})/(2u)$ gives the result.

\end{proof}


\subsection{For $u=u_c$}

In this subsection, we give an example of computing an inverse Kasteleyn entry for $u=u_c$ and a list of some of the inverse Kasteleyn entries we can also compute using this technique.

\begin{lemma} \label{app:uc:exactcomputation}
For $u=u_c$, we have
\(
	K^{-1} (v_3,v_4)= \frac{ \pi x^2 -2 (1 + x^2)  \mathrm{arccotan}(x)}{2 \pi x (1 - x^2)}
\)
\end{lemma}

The computation is slightly different to the one used in Lemma~\ref{app:therm:exp} for $u=u_c$ -- we do not take the initial change of variables because for these entries  it makes the computation more involved.  

\begin{proof}
From using the cofactor expansion given in Appendix~\ref{app:cofactor}, we can write the numerator of the integrand of the integral we want to compute as
\begin{equation}
  -\frac{ u_c^2 x(w^2 x^2 + z (-2 + x^2 z) - 2 w (1 + (2 + x^2) z))}{ z w}
\end{equation}
and the denominator is given by
\begin{equation}
    z w P(z,w)= u_c^2 x^2 (w^2 (-1 + x^2 - 2 z) + z (-2 + (-1 + x^2) z) - 
  2 w (1 + (3 + x^2) z + z^2))
\end{equation}
Therefore, we can write $K^{-1}(v_3,v_4)$ as
\( \label{app:uc:k34}
	\frac{1}{(2\pi i)^2} \int_{\mathbb{T}^2} \frac{ -w^2 x^2 + z (2 - x^2 z) + 2 w (1 + (2 + x^2) z)}{w x z (w^2 (-1 + x^2 - 2 z) + z (-2 + (-1 + x^2) z) - 
   2 w (1 + (3 + x^2) z + z^2))} \frac{dz}{z}\frac{dw}{w}.
\)
The denominator of the above integrand can be rewritten as $ xz w (1+2w-x^2)(z-y_+(w))(z-y_-(w))$ where $y_{\pm}(w)$ are both defined in (\ref{hpart:uc:ypm}),  with $|y_-(w)|<1$, $|y_+(w)|>1$ for all $|z|=1$.  We have that 
\(
	\sqrt{
 1 + 4 w + 6 w^2 + 4 w^3 + w^4 + 4 w x^2 + 8 w^2 x^2 + 4 w^3 x^2}=(w+1)\sqrt{w-s(x))(w-s(x)^{-1})}
\)
where $s(x)= -1 - 2 x^2 + 2x \sqrt{1 + x^2}$.  

Computing the residue at $z=0$ of (\ref{app:uc:k34}) gives
\( \label{app:uc:k34residueatz0}
	\frac{1}{2\pi i} \int_{|w|=1}  \frac{2 - w x^2}{ w x (-2 - w + w x^2)}dw=-\frac{1}{x}.
\)
The residue at $z=y_-(w)$ gives
\begin{align}
	&\frac{1}{2\pi i} \int_{|w|=1} \frac{-1 - 2 w + w^2 x^2}{w x (1 + 2 w - x^2) (-2 + w (-1 + x^2))} \nonumber \\
	&-\frac{ (1 + w^3 x^2 + w (3 + 2 x^2) + w^2 (2 + x^2 - 2 x^4))}{ w  x (1+ 2w -x^2)(-2 + w (-1 + x^2)) \sqrt{(1 + w^2 + w (2 + 4 x^2))}} dw \label{app:uc:bigequation}
\end{align}
The above integrand has singularities at $w=0$, $w=-1/2(1-x^2)$ and $w \in [-1,s(x)]$.   The residues at $w=0$ and $w=-1/2(1-x^2)$ for the integral in (\ref{app:uc:bigequation}) give
\(\label{app:uc:residueeasy}
	\frac{1}{x(1-x^2)} -\frac{x}{1-x^2}=\frac{1}{x}.
\)
Adding the quantities from (\ref{app:uc:residueeasy}) and (\ref{app:uc:k34residueatz0}), we can rewrite $K^{-1}(v_3, v_4)$ as a line integral, namely
\( \label{app:uc:k34lineintegral}
	K^{-1} (v_3,v_4)= \frac{2}{2\pi i} \int_{-1}^{s(x)} \frac{ (1 + w^3 x^2 + w (3 + 2 x^2) + w^2 (2 + x^2 - 2 x^4))}{ w  x (1+ 2w -x^2)(-2 + w (-1 + x^2)) \sqrt{(1 + w^2 + w (2 + 4 x^2))}} dw.  
\)
In the above expression, we can write
\begin{equation*}
 \frac{ (1 + w^3 x^2 + w (3 + 2 x^2) + w^2 (2 + x^2 - 2 x^4))}{ w  x (1+ 2w -x^2)(-2 + w (-1 + x^2))}
\end{equation*}
in terms of partial fractions which means that~\eqref{app:uc:k34lineintegral} is a linear combination of standard integrals.  Each of these integrals can be computed and summing up gives the result. 
\end{proof}

\begin{lemma} \label{app:uc:exactcomputation2}
For $u=u_c$, we have
\begin{align}
	K^{-1} (v_4,v_5)&=-K^{-1}(v_4,v_6) \nonumber \\
	&= -\frac {\pi  \left(-1+x^2\right) \left(1+\sqrt{1+x^2}\right)+\sqrt{1+x^2} \left(3-3 x^2+\sqrt{1+x^2}\right) \arccos x }{4 \pi  \left(-1+x^2\right)} \nonumber \\
	&+ \frac{\sqrt{1+x^2} \left(-1+x^2+\sqrt{1+x^2}\right) \arctan \left(\frac{1-3 x^2}{3 x-x^3}\right)}{4 \pi  \left(-1+x^2\right)}, \nonumber \\
	&=-\frac{x}{\pi}+\frac{x^2}{2} +O(x^3) \nonumber
\end{align}
\begin{align} \nonumber
	K^{-1}(v_5,v_6)&= -\frac{(1+\sqrt{1+x^2}) (\pi (1-x^2)+\sqrt{1+x^2}(2-\sqrt{1+x^2}) (\pi -4 \arctan x ) }{4 \pi (1-x^2)}  \\
	&=-1+\frac{2 x}{\pi} -\frac{3x^2}{4} +O(x^3), \nonumber
\end{align}
\end{lemma}

\begin{proof}
	These values can be computed in a very similar fashion to Lemma \ref{app:uc:exactcomputation}.
\end{proof}

\section{Expansions and Computations of $H_u$ } \label{app:H}

In this section, we prove Lemma~\ref{uless:lem:H}, Lemma~\ref{hp:uc1lemma} and Lemma~\ref{cor:lemmacritical2}.

\subsection{$u<1/2$ (chosen independently of $x$) } \label{app:c:uless}

In this subsection, we prove Lemma \ref{uless:lem:H}.

\begin{proof}[Proof of Lemma \ref{uless:lem:H}]

We can use the analysis of the roots in Appendix \ref{app:sec:factorization} to show that $P(z,w)$ is analytic for $u<1/2$ for $(z,w)\in \mathbb{T}^2$ with $x \to 0$.  By letting 
\( \label{uless:lemproof:ptilde}
	\tilde{P}(z,w)= \lim_{x \rightarrow 0}\frac{1}{x^2} z w \det K(z,w)=(1+ z u +w u)(w z + u z + u w),
\)  
we can take a series expansion in terms of $x$ to find that
\(
	H_u(m,n)= \frac{x^{-2}}{(2\pi i)^2}\int_{|w|=1} \int_{|z|=1}\frac{z^m w^n}{\tilde{P}(z,w)}dz\, dw +O(x^0).
\)
by the Dominated Convergence Theorem and so $H_u^0(m,n)=\lim_{x\to 0 } x^2 H_u(m,n)$ for $m,n \in \Z$ exists. 

As $|u w|/|u+w|<1$ for all $|w|=1$, using the factorization of $\tilde{P}(z,w)$ in (\ref{uless:lemproof:ptilde}), we can compute the residue at $z=uw/(u+w)$ which gives
\(
	H_u^0(m,n)= \frac{1}{2\pi i} \int_{|w|=1} \left(-\frac{ u w}{u+w} \right)^n\frac{w^m}{u+w+u w^2} dw . \label{uless:green's}
\)
By partial fractions, we have
\( \label{uless:partialfractions}
	\frac{1}{u+w+u w^2}= \frac{1}{u(r_1-r_2)}\left( \frac{1}{w-r_1} -\frac{1}{w-r_2} \right),
\)
where $r_1$ and $r_2$ are the roots of the polynomial $u+w+uw^2$ with $|r_1|<1$ and $|r_2|>1$.  As $|r_1|<1$, the contribution for (\ref{uless:green's}) obtained by deforming the contour of integration to $|w-r_1|=\e$ is given by
\( \label{uless:lem:Hprooffirstcontribtuion}
	\frac{(-u)^nr_1^{m+n}}{u(r_2-r_1)(r_1+u)^n}.
\)
The remaining contribution for (\ref{uless:green's}) comes from deforming the contour of integration to $|w+u|=\e$.  For $v \in \mathbb{C}$ with $v\not= u$,  using the residue theorem at $w=-u$ gives
\begin{align}
	\frac{1}{2\pi i } \int_{|w+u|=\e}  \left(-\frac{ u w}{u+w} \right)^n\frac{w^m}{w+v} dw &= (-u)^{m+n}(m+n)!\sum_{k=0}^{n-1} \frac{(-1)^{n-k-1} }{k!(m+n-k)!} \left(\frac{u}{u-v}\right)^{n-k}   \nonumber \\
	&=- (-1)^m\left(-\frac{u}{u-v} \right)^n v^{m+n}   \nonumber \\
	&-(-u)^{m+n} \frac{(m+n)!}{m! n!} {_2 F_1(1,-m,1+n, 1-v/u)}. \label{uless:interactionlemma1}
\end{align}
We can substitute $v=-r_1$ and $v=-r_2$ and use the partial fractions decomposition given in~\eqref{uless:partialfractions} which gives the contribution of~\eqref{uless:green's} when the contour of integration is deformed to $|w+u|=\e$. This gives
\begin{equation}  \label{uless:lem:Hproofsecondcontribtuion}
 \begin{split}
 &\frac{(-u)^{m+n} (m+n)!}{u(r_1-r_2)} \left( \frac{{_2 F_1(1,-m,1+n, s_1^{-1})}}{m! n!} - s_2 \sum_{k=0}^{n-1} \frac{(-s_2)^{n-1-k}}{k! (m+n-k)!} \right)- \frac{(-u)^nr_1^{m+n}}{u(r_2-r_1)(r_1+u)^n}  
 \end{split}
\end{equation}
where $s_1$ and $s_2$ are defined in the statement of the lemma.  Adding~\eqref{uless:lem:Hprooffirstcontribtuion} and~\eqref{uless:lem:Hproofsecondcontribtuion} gives $H_u^0(m,n)$. 
	
For the third statement of the lemma, without loss of generality, suppose that $n \geq m$. We have
\(
	H_u^0(-m,n)= \frac{1}{(2\pi i)^2} \int_{|w|=1} \int_{|z|=1} \frac{w^n}{z^m (1+z u+ w u)(uw+uz+wz)} dz \,dw.
\)
By a partial fractions decomposition, we have
\(
	\frac{1}{(1+z u+ w u)(uw+uz+wz)} = \frac{1}{u+w+u w^2} \left(\frac{u+w}{uw+uz+wz}-  \frac{u}{1+uw+uz} \right).
\)
Completing the residue calculations with respect to $z$ leads to the simplification
\begin{align}
	H_u^0(-m,n)&=\frac{1}{2\pi i} \int_{|w|=1} \frac{(-1)^{m} u^m w^n }{(u+w+u w^2)(1+u w)^m} dw   \\
	&=\frac{(-1)^m u^m r_1^{n}}{(1+u r_1)^m\sqrt{1-4u^2} } 
\end{align}
where $r_1$ is the root lying inside of the unit circle of the polynomial $u + w+uw^2$. Simplifying the above equation gives the result.
\end{proof}

\subsection{Local dynamics $u=u_c$} \label{app:uc:local}

In this subsection, we prove Lemma \ref{hp:uc1lemma}.

\begin{proof}[Proof of Lemma \ref{hp:uc1lemma}]

The denominator of the integrand is equal to $P(z,w) z w$ which can be rewritten as  $u_c^2 x^2 (1+2w-x^2) (z-y_+(w))(z-y_-(w))$ where $y_{\pm}(w)$ are both defined in (\ref{hpart:uc:ypm}). As  we have $|y_-(w)|<1$ and $|y_+(w)|>1$ for all $|z|=1$ ,   we can compute the first integral with respect to $z$ by taking the residue at $y_-(w)$.  This gives
\(
	H_{u_c}(m,n)=\frac{1}{2\pi i} \int_{|w|=1}\frac{ x^2 \left(-(-1)^{m-n}+w^{-n} \left(-\frac{1+w^2+w \left(3+x^2\right)-\sqrt{(1+w)^2 \left(1+2 w+w^2+4 w x^2\right)}}{1+2 w-x^2}\right)^m\right)}{2 \left(2+x^2-2 \sqrt{1+x^2}\right) \sqrt{(1+w)^2 \left(1+w^2+w \left(2+4 x^2\right)\right)}} dw. \label{par:crit1cont}
\)

Note that there is a branch cut situated along the negative real axis from $1-2x^2 - 2 x\sqrt{1+x^2}$ to $s(x)=-1-2 x^2 + 2 x\sqrt{1+x^2}$. Therefore, the integral over $|w|=1$ in (\ref{par:crit1cont}) can be deformed to an integral from $w=-1$ to $w=s(x)$ (above the branch cut), an integral from $w=s(x)$ to $w= -1$ (below the branch cut)  and a contour integral around $w=\half(x^2-1)$. 

The contributions given near the branch cut can be computed  as follows.  By taking a series expansion of the numerator about the point $w=-1$ gives an integral of the following form

\(\label{par:crit1cont1}
	\frac{1}{2\pi i} \int_{-1}^{s(x)} \frac{x^2}{2(2+x^2-2 \sqrt{1+x^2})} \sum_{k=0}^\infty a_k \frac{(w+1)^{k+1}}{\sqrt{(w+1)^2(1+w^2+w(2+4 x^2))}} dw   
\)
where $a_k$ is a function of $m,n$ and $x$ which can be computed using the integrand given in~\eqref{par:crit1cont}.  For example, we find that
\begin{equation}
      a_0=\frac{(-1)^{m - n} x^2 (m + n - m x^2 + n x^2)}{2 (1 + x^2) (2 + x^2 - 2 \sqrt{1 + x^2})}.
\end{equation}
For $i \geq 1$, we define the integrals 
\(
	I_i :=\frac{1}{2\pi i} \int_{-1}^{s(x)} \frac{(w+1)^{i-1}}{\sqrt{(w+1)^2 +4 w x^2}} dw
\)
which are  standard integrals and can be evaluated explicitly.  To find a series expansion of~\eqref{par:crit1cont1}, it is enough to compute $I_i$ for $i=2$ up to $i=5$ because $I_i=O(x^{i-1})$, each $a_k$ is $O(x^{-2})$ and there is no $I_1$ in the expression~\eqref{par:crit1cont1}.  By expanding each $I_i$ and $a_k$ in terms of $x$, we can find a series expansion for~\eqref{par:crit1cont1} in $x$ which is given by  
$$-F_0 (m,n)x^{-2}+F_1(m,n)x^{-1}-F_2(m,n)+F_3(m,n) x.$$

It remains to calculate the other contribution in (\ref{par:crit1cont}), namely the contour integral around $w=\half(x^2-1)$.  Taking a series expansion in $x$ gives
\(  \label{par:lastcontribution}
	\frac{1}{2\pi i}\int_{\left| w+\half \right|=\e} \frac{2 (-1)^m w^{m-n} }{x^2(1+2 w)^m(1+w)^2}-\frac{2 \left((-1)^m w^{1+m-n} \left(1+2 w+m (1+w)^2\right)\right)}{(1+2 w)^{m+1}(1+w)^4} dw +O(x^2).
\)
This can be computed using the following integral
\(
	G(l)=:\frac{1}{2\pi i}\int_{\left| w+\half \right|=\e}\frac{w^{m-n}}{(1+2 w)^m (1+w)^l} dw.
\)
These can all be found by a higher dimensional residue formula and are given by
 $G(1)=(-1)^n$, $G(2)=(-1)^{n}(m+n)$
\(G(3)=\half (-1)^n (-2 + m^2 + n + n^2 + m (3 + 2 n))\) and
\(G(4)=\frac{1}{6} (-1)^n (2 + m + n) (-6 + m^2 + n + n^2 + m (7 + 2 n)).\)  
This means that we can compute~\eqref{par:lastcontribution} explicitly and hence we have computed all the terms which account for~\eqref{par:crit1cont}.
\end{proof}

\subsection{Local dynamics for $u=u_i$} \label{app:ui:local}


\begin{proof}[Proof of Lemma \ref{cor:lemmacritical2}]
We only provide the calculation for $m \geq n$ as the other case is analogous.   To simplify calculations, set $x$ to $\sqrt{2 t-1}/t$ as this removes the square root term. Using this simplification, a calculation gives
\begin{equation}
 P(z,w)= \frac{(  (2 + w + z)t-w-z) ( (w+ z+2 w z)t-w-z)}{t^2 ( 2 t-1) w z}
\end{equation}
We can compute the residue of $H_{u_i}(m,n)$ at $z=(w - t w)/(t + 2 t w-1)$ which gives
\(
 	H_{u_i}(m,-n)=\frac{1}{2\pi i} \int_{|w|=1} \frac{t(2t-1) w^{-n} (w-t w)^m}{2 (t+2t w-1)^{m}(t+2tw +tw^2-w^2-1)} dw.
\label{par:critlemma}
\)
The singularities of the integrand inside the region $|w|\leq 1$ are located at $w=(-t+\sqrt{2t-1})/(-1+t)$ and $w=(1-t)/(2 t)$.  The contribution from the residue at $w=(-t+\sqrt{2t-1})/(-1+t)$  is given by
\(
	\frac{1}{4} t \sqrt{2t-1} \left(\frac{t+\sqrt{2t-1}}{1-t} \right)^{m+n}.  \label{par:crit2contrib}
\)

It remains to find the residue for (\ref{par:critlemma}) at $w=(1-t)/(2 t)$.  This can be achieved by splitting the term $(-1+t+2t w-w^2+t w^2)=(t-1)(w-c_1 )(w-c_2)$, where $c_1=1/c_2$ and $|c_1|<1$.  Hence, the denominator in (\ref{par:critlemma}) can be separated into 
\(
	\frac{1}{(-1+t+2t w-w^2+t w^2)} = \frac{1}{2 \sqrt{2t-1}} \left( \frac{1}{w-c_1}-\frac{1}{w-c_2} \right).
\)
Each term can be computed using the high dimensional residue formula: 
\(
	\frac{1}{2 \pi i} \int_{\left|w+\half \frac{1-t}{2t}\right|=\epsilon} \frac{t\sqrt{2t-1} w^{m-n}(1-t)^m}{2(-1+t+ 2t w)^m (w-c_1) }  =-\frac{1}{4} t \sqrt{2t-1} \left(\frac{t+\sqrt{2t-1}}{1-t} \right)^{m+n}   \label{par:crit2contr1}
\)  
and
\(
	-\frac{1}{2 \pi i} \int_{\left|w+\half \frac{1-t}{2t}\right|=\epsilon} \frac{t\sqrt{2t-1} w^{m-n}(1-t)^m}{2(-1+t+ 2t w)^m (w-c_2) }=\frac{1}{4} t \sqrt{2t-1}  \left(\frac{t-\sqrt{2t-1}}{1-t} \right)^{m+n}.  \label{par:crit2contr2}
\)
Notice that contributions from (\ref{par:crit2contrib}),(\ref{par:crit2contr1}) and (\ref{par:crit2contr2}) sum up to give (\ref{par:critlemma}).  As the contributions from (\ref{par:crit2contrib}) and (\ref{par:crit2contr1}) cancel out, we obtain
\(
	H_{u_i}(m,-n)=\frac{1}{4} t \sqrt{2t-1}  \left(\frac{t-\sqrt{2t-1}}{1-t} \right)^{m+n}.
\) 
Setting $t=(1+\sqrt{1-x^2})/x^2$ gives 
\begin{equation} \nonumber
    H_{u_i}(m,-n)=\frac{(1+\sqrt{1-x^2})^2}{4x^3} \left( \frac{(-1+x)(1+\sqrt{1-x^2})}{1-x^2+\sqrt{1-x^2}} \right)^{m+n}
\end{equation}
which can be simplified to the expression given in Lemma~\ref{cor:lemmacritical2}.  
\end{proof}

\end{appendix}

\end{document}